\newcommand{\Author}{Zev Chonoles}
\newcommand{\Subject}{}
\newcommand{\Keywords}{zev chonoles, thesis, university of chicago, uchicago, math, algebraic topology, equivariant, classifying space}

\documentclass{ucetd}

\usepackage{enumitem}

\title{The $\RO(G)$-Graded Cohomology of the Equivariant \protect\\Classifying Space $B_{G}\SU(2)$}
\author{Zev Chonoles}
\department{Mathematics}
\division{Physical Sciences}
\degree{Doctor of Philosophy}
\date{June 2018}

\dedication{For Mom}
\epigraph{``They know enough who know how to learn.'' -- Henry Adams}

\usepackage{amsmath,amssymb,amsthm}

\usepackage{mathtools}

\usepackage{tikz-cd}

\let\originalleft\left
\let\originalright\right
\renewcommand{\left}{\mathopen{}\mathclose\bgroup\originalleft}
\renewcommand{\right}{\aftergroup\egroup\originalright}

\allowdisplaybreaks

\usepackage{parskip}

\usepackage{newtxtext}
\usepackage{newtxmath}

\usepackage[T1]{fontenc}

\usepackage{anyfontsize}

\usepackage{eucal}\let\eucal\mathcal  \let\cal\CMcal  \let\mathcal\cal

\usepackage{microtype}

\newcommand{\Z}{\mathbb{Z}}
\newcommand{\R}{\mathbb{R}}
\newcommand{\C}{\mathbb{C}}
\renewcommand{\H}{\mathbb{H}}

\newcommand{\irredC}{\Phi}
\newcommand{\irredH}{\Psi}
\newcommand{\repextension}[1]{\mathrm{e}#1}
\newcommand{\represtriction}[1]{\mathrm{r}#1}

\newcommand{\Irr}{\operatorname{Irr}}
\newcommand{\unitsphere}[1]{\mathbf{S}(#1)}
\newcommand{\diskletter}{\mathbf{D}}
\newcommand{\unitdisk}[1]{\diskletter(#1)}
\let\sectionsymb\S
\renewcommand{\S}[1]{S^{#1}}
\newcommand{\RO}{\mathrm{RO}}
\newcommand{\signrep}{\sigma}

\newcommand{\SO}{\mathrm{SO}}
\newcommand{\SU}{\mathrm{SU}}
\newcommand{\U}{\mathrm{U}}

\newcommand{\catnamestyle}[1]{\textsf{#1}}
\newcommand{\orbitcategory}[1]{\mathcal{O}_{#1}}
\newcommand{\op}{\textsf{op}}
\newcommand{\burnside}[1]{\mathcal{B}_{#1}}
\newcommand{\Ab}{\catnamestyle{Ab}}
\newcommand{\kletter}{k}
\newcommand{\kmod}{\kletter\catnamestyle{-mod}}
\newcommand{\mackey}[1]{\catnamestyle{Mack}_{#1}}

\newcommand{\Hhreduced}{\widetilde{H}}
\newcommand{\GG}{\text{\upshape{$\bullet$}}}
\newcommand{\Ge}{\text{\upshape{$\sun$}}}
\newcommand{\HH}{}
\newcommand{\HHreduced}{}
\DeclareRobustCommand{\HH}{\underline{H\mkern-2mu}\mkern2mu}
\DeclareRobustCommand{\HHreduced}{\underline{\widetilde{H}\mkern-2mu}\mkern2mu}

\newcommand{\placeholder}[1]{\delta_{#1}}

\newcommand{\coeffs}{}
\makeatletter
\DeclareRobustCommand{\coeffs}{\@ifstar\@coeffs\@@coeffs}
\def\@coeffs#1{;#1}
\def\@@coeffs#1{}
\makeatother

\renewcommand{\P}{\mathrm{P}}
\newcommand{\HP}{\H\P}

\newcommand{\PC}{}
\newcommand{\PH}{}

\newcommand{\flagstyle}[1]{\eucal{#1}}
\newcommand{\cellstyle}[1]{\mathfrak{#1}}

\makeatletter

\def\@mypiece#1{#1}

\DeclareRobustCommand\PC{\@ifnextchar[{\@with}{\@without}}
  \def\@with[#1]{\P_{\C}^{\mkern1.5mu\@mypiece{#1}}}
  \def\@without{\P_{\C}}

\DeclareRobustCommand\PH{\@ifnextchar[{\@witha}{\@withouta}}
  \def\@witha[#1]{\P_{\H}^{\mkern1.5mu\@mypiece{#1}}}
  \def\@withouta{\P_{\H}}  

\DeclareRobustCommand\Ctotop{\@ifnextchar[{\@withb}{\@withoutb}}
  \def\@withb[#1]#2{A_{#2}^{\@mypiece{#1}}}
  \def\@withoutb#1{A_{#1}}
  
\DeclareRobustCommand\Ctonext{\@ifnextchar[{\@withc}{\@withoutc}}
  \def\@withc[#1]#2{a_{#2}^{\@mypiece{#1}}}
  \def\@withoutc#1{a_{#1}}

\DeclareRobustCommand\Ctowhole{\@ifnextchar[{\@withd}{\@withoutd}}
  \def\@withd[#1]#2{p_{#1}^{\@mypiece{#2}}}
  \def\@withoutd#1{p^{\@mypiece{#1}}}
  
\DeclareRobustCommand\Htotop{\@ifnextchar[{\@withe}{\@withoute}}
  \def\@withe[#1]#2{B_{#2}^{\@mypiece{#1}}}
  \def\@withoute#1{B_{#1}}
  
\DeclareRobustCommand\Htonext{\@ifnextchar[{\@withf}{\@withoutf}}
  \def\@withf[#1]#2{b_{#2}^{\@mypiece{#1}}}
  \def\@withoutf#1{b_{#1}}

\DeclareRobustCommand\Htowhole{\@ifnextchar[{\@withg}{\@withoutg}}
  \def\@withg[#1]#2{q_{#1}^{\@mypiece{#2}}}
  \def\@withoutg#1{q^{\@mypiece{#1}}}

\DeclareRobustCommand\Htocofiber{\@ifnextchar[{\@withh}{\@withouth}}
  \def\@withh[#1]#2{\pi_{#2}^{\@mypiece{#1}}}
  \def\@withouth#1{\pi_{#1}}

\DeclareRobustCommand\Htowholecofiber{\@ifnextchar[{\@withi}{\@withouti}}
  \def\@withi[#1]#2{\chi_{#2}^{\@mypiece{#1}}}
  \def\@withouti#1{\chi_{#1}}
  
\DeclareRobustCommand\Hcofiber{\@ifnextchar[{\@withj}{\@withoutj}}
  \def\@withj[#1]#2{Z_{#2}^{\@mypiece{#1}}}
  \def\@withoutj#1{Z_{#1}}

\DeclareRobustCommand\HtoGe{\@ifnextchar[{\@withk}{\@withoutk}}
  \def\@withk[#1]{\rho_{#1}}
  \def\@withoutk{\rho}

\DeclareRobustCommand\HparttoGe{\@ifnextchar[{\@withl}{\@withoutl}}
  \def\@withl[#1]#2{\rho_{#1}^{#2}}
  \def\@withoutl#1{\rho^{#1}}

\newcommand{\Hcofiberrep}[1]{\cellstyle{w}_{#1}}

\newcommand{\HcofibertoGe}[1]{\hat{\rho}_{#1}}

\makeatother

\newcommand{\id}{\mathrm{id}}

\newcommand{\pt}{\mathrm{pt}}

\newcommand{\gsetstyle}[1]{\mathbf{#1}}

\renewcommand{\restriction}[1]{r_{#1}}
\newcommand{\transfer}[1]{t_{#1}}

\newcommand{\defstyle}[1]{\textbf{\textit{#1}}}

\newcommand{\from}{\colon}

\usepackage{wasysym}
\newcommand{\boxprod}{\mathbin{\square}}
\newcommand{\externalboxprod}{\mathbin{\overline{\square}}}

\newcommand{\tonatural}{\Rightarrow}

\newcommand{\gen}{\mathsf{c}}
\newcommand{\Gen}{\mathsf{C}}

\usepackage{aliascnt}

\newtheorem{theorem}{Theorem}[chapter]

\newaliascnt{lemmacnt}{theorem}
\newtheorem{lemma}[lemmacnt]{Lemma}

\newaliascnt{corollarycnt}{theorem}
\newtheorem{corollary}[corollarycnt]{Corollary}

\newaliascnt{propositioncnt}{theorem}
\newtheorem{proposition}[propositioncnt]{Proposition}

\theoremstyle{definition}

\newaliascnt{definitioncnt}{theorem}
\newtheorem{definition}[definitioncnt]{Definition}

\newaliascnt{examplecnt}{theorem}
\newtheorem{example}[examplecnt]{Example}

\newaliascnt{remarkcnt}{theorem}
\newtheorem{remark}[remarkcnt]{Remark}

\makeatletter \def\thm@space@setup{\thm@preskip=\parskip \thm@postskip=0pt} \makeatother

\makeatletter \renewenvironment{proof}[1][\proofname]{\pushQED{\qed}\normalfont
\partopsep=\z@skip \topsep=\z@skip \trivlist \item[\hskip\labelsep\itshape #1\@addpunct{.}]
\ignorespaces}{\popQED\endtrivlist\@endpefalse} \makeatother

\usepackage{graphicx}

\usepackage{xcolor}
\definecolor{myred}{rgb}{0.9,0.2,0.2}
\definecolor{mygreen}{rgb}{0.2,0.6,0.2}
\definecolor{myblue}{rgb}{0.2,0.2,1}

\usepackage{tikz}
\usepackage{pgfplots}
\pgfplotsset{compat=1.13}
\usepackage{mathtools}
\usetikzlibrary{arrows,calc,decorations,decorations.pathmorphing,decorations.markings,fadings,positioning,patterns,shapes,shadows}
\tikzstyle{mypoint}=[inner sep=0pt,outer sep=0pt,minimum size=5pt,fill,circle]

\usepackage[style=alphabetic,backend=bibtex]{biblatex}

\bibliography{mybiblio}

\usepackage{etoolbox}
\makeatletter
\patchcmd{\@chapter}{\addtocontents{lof}{\protect\addvspace{10\p@}}}{}{}{}
\patchcmd{\@chapter}{\addtocontents{lot}{\protect\addvspace{10\p@}}}{}{}{}
\makeatother

\usepackage{hyperref}
\hypersetup{%
	pdftitle      = {The RO(G)-Graded Cohomology of the Equivariant Classifying Space B\_GSU(2)},
	pdfauthor     = {\Author},
	pdfsubject    = {\Subject},
	pdfkeywords   = {\Keywords},
	pdfproducer   = {LaTeX},
	pdflang       = {en},
	linkbordercolor     = red,
	citebordercolor     = green,
	urlbordercolor      = blue,
	pdfpagemode   = UseNone,
	pdfstartview  = FitH
}
\usepackage{bookmark}

\begin{document}

\thispdfpagelabel{i}
\pagenumbering{roman}
\setcounter{page}{1}
\maketitle

\makecopyright
\makededication
\makeepigraph


\setcounter{tocdepth}{1}
\tableofcontents


\listoffigures


\acknowledgments

I am deeply grateful to my advisor, Professor Peter May, for all of his patient guidance and support as I learned about the enticing world of equivariant algebraic topology. 

My sincere appreciation also goes to Professor Glenn Stevens of Boston University, who helped me grow so much as mathematician during my time at PROMYS.

My father, Michael Chonoles, set me on my academic course by never missing an opportunity to share a fun and esoteric piece of knowledge with me, mathematical or otherwise. 

I would not be here at all, much less be the person I am today, without Isabel Leal as my friend. Thank you for everything.


\abstract

We compute the additive structure of the $\RO(C_n)$-graded Bredon equivariant cohomology of the equivariant classifying space $B_{C_n}\SU(2)$, for any $n$ that is either prime or a product of distinct odd primes, and we also compute its multiplicative structure for $n=2$. In particular, as an algebra over the cohomology of a point, we show that the cohomology of $B_{C_2}\SU(2)$ is generated by two elements subject to a single relation: writing $\sigma$ for the sign representation of $C_2$ in $\RO(C_2)$, the generators are an element $\gen$ in dimension $4\signrep$ and an element $\Gen$ in dimension $4+4\signrep$, satisfying the relation $\gen^2=\epsilon^4\gen+\xi^2\Gen$, where $\epsilon$ and $\xi$ are elements of the cohomology of a point. Throughout, we take coefficients in the Burnside ring Mackey functor $A$.

The key tools used are equivariant ``even-dimensional freeness'' and ``multiplicative comparison'' theorems for  $G$-cell complexes, both proven by Lewis in \cite{lewis_complex} and subsequently refined by Shulman in \cite{megan}, and with the former theorem extended by Basu and Ghosh in \cite{basu_ghosh}. The latter theorem enables us to compute the multiplicative structure of the cohomology of $B_{C_2}\SU(2)$ by embedding it in a direct sum of cohomology rings whose structure is more easily understood. Both theorems require the cells of the $G$-cell complex to be attached in a well-behaved order, and a significant step in our work is to give $B_{C_n}\SU(2)$ a satisfactory $C_n$-cell complex structure.  


\mainmatter


\chapter{Introduction}

\section{Overview}

Equivariant algebraic topology seeks to generalize classical topological notions such as the cohomology of a space $X$ to a context where the space in question has an action by a group $G$.  While the Borel cohomology of a $G$-space $X$ is an equivariant generalization that can be computed readily and is adequate for some purposes, it still loses much of the information encoded by the action. One striking example is that, as measured by Borel cohomology, the characteristic classes of a $G$-equivariant principal $\Pi$-bundle are in a certain precise sense determined entirely by non-equivariant information about $G$ and the bundle with its $G$-action forgotten (see \cite{may1987characteristic}). In contrast, Bredon cohomology, graded not just on $\Z$ but on the real represenation ring $\RO(G)$, retains far more information about the $G$-action on $X$ and has a richer internal structure.

However, the price paid for these benefits is that computations of $\RO(G)$-graded Bredon cohomology are surprisingly difficult to perform, so that very little is known about equivariant characteristic classes in this cohomology theory. In fact, it is a struggle to calculate the cohomology of a point, with its necessarily-trivial $G$-action. This was first achieved by Stong in \cite{stong_letters} for $G=C_p$, the cyclic group of prime order $p$, taking coefficients in the Burnside ring Mackey functor $A$, the equivariant analog of integer coefficients. Since then, there has only been progress for certain other cyclic groups in \cite{hhr} and \cite{basu_ghosh}, and very recently, for $G=S_3$ in \cite{kriz}; in \cite{hhr} the cohomology was not computed using coefficients in $A$, and in \cite{basu_ghosh} only the additive structure was computed. Any calculation of this cohomology for some new $G$-space $X$ helps relieve the current paucity of examples, so that eventually we might better see the general landscape.

The few computations of $\RO(G)$-graded Bredon cohomology which have been done so far have focused on the classifying spaces $B_G\Pi$ of $G$-equivariant principal $\Pi$-bundles, which are of interest because the cohomology of a classifying space determines the characteristic classes of that type of principal bundle, just as in the non-equivariant setting. Lewis provided the first example by computing the $\RO(C_p)$-graded Bredon cohomology of $B_{C_p}\U(1)$ in the paper \cite{lewis_complex}, which also established two crucial tools, an ``even-dimensional freeness'' theorem and a ``multiplicative comparison'' theorem, of which more later. Shulman did the same for $B_{C_p}\mathrm{O}(2)$ in \cite{megan}, and here the major tool was the equivariant Serre spectral sequence as developed in \cite{moerdijksvensson} and \cite{kronholm}. As mentioned above, in the paper \cite{basu_ghosh} Basu and Ghosh computed the additive structure of the $\RO(G)$-graded Bredon cohomology of a point when $G=C_{pq}$ for distinct odd primes $p,q$, and they also extended the even-dimensional freeness theorem to this case, and so were able to compute the additive structure of the cohomology of $B_{C_{pq}}\U(1)$.

\section{Main results}

In this thesis, we compute the additive structure of the $\RO(C_n)$-graded Bredon cohomology of the equivariant classifying space $B_{C_n}\SU(2)$, for any $n$ that is either prime or a product of distinct odd primes, and we also compute its multiplicative structure for $n=2$. 
Throughout, we use cohomology with coefficients in the Burnside ring Mackey functor $A$. 
The work is organized as follows:

In Chapter 2, we first recall some basic facts about equivariant algebraic topology, as well as two important results which first appeared in \cite{lewis_complex}, an ``even-dimensional freeness'' theorem and a ``multiplicative comparison'' theorem. Both have corrected proofs given in \cite{megan}.  These results consider a $C_p$-cell complex $X$ with even-dimensional cells whose cells are attached in a sufficiently well-behaved order. The former states that the cohomology of $X$ is free as a module over the cohomology of a point, and the latter states that the multiplicative structure of the cohomology of $X$ can be embedded in a direct sum of cohomology rings whose structure is more easily understood. As mentioned above, Basu and Ghosh in \cite{basu_ghosh} extended the even-dimensional freeness theorem to the group $C_{pq}$ for distinct odd primes $p,q$.

In Chapter 3, we establish several lemmas regarding quaternionic representations and their projective spaces. The key results here are cofiber sequences which, given a quaternionic $C_n$-representation $W$, describe how the $C_n$-cell complex structure of both its projective space $\PH(W)$ and the fixed points $\PH(W)^{dC_n}$ thereof depend on the order in which the irreducibles of $W$ are added.

In Chapter 4, we show that the projective space $\PH(Q)$ of a complete quaternionic $C_n$-universe $Q$ provides a model for the equivariant classifying space $B_{C_n}\SU(2)$, and we then use our work from Chapter 3 to create a $C_n$-cell complex structure on $\PH(Q)$ which is sufficiently well-behaved, so that the even-dimensional freeness and multiplicaive comparison theorems apply.

In Chapter 5, we take our work from Chapter 4 and reap the benefits of the even-dimensional freeness theorem, finding the additive structure of the $\RO(C_n)$-graded Bredon cohomology of $B_{C_n}\SU(2)$ for all $C_n$ for which the freeness theorem has been established, i.e., when $n$ is prime or a product of distinct odd primes.

Finally, in Chapter 6, we perform the harder work of computing the multiplicative structure of the $\RO(C_2)$-graded Bredon cohomology of $B_{C_2}\SU(2)$. The reason for the restriction to $n=2$ is that the multiplicative structure of the cohomology of a point is quite convoluted for odd primes $n$ (see \cite[p.65]{lewis_complex}), and completely unknown for composite $n$ (which would at any rate not be sufficient, since we also lack a multiplicative comparison theorem for composite $n$). The general approach for our computation is similar to that of Lewis in \cite{lewis_complex}: we construct our generators explicitly, and show that they generate the cohomology of $B_{C_2}\SU(2)$ by proving inductively that they generate the cohomology of each piece of the filtration of $B_{C_2}\SU(2)$ by its $C_2$-cell complex structure.



\chapter{Preliminaries}


Throughout, let $G$ be a finite group.

\section{Equivariant analogs of cell complexes}

\subsection{$G$-CW complexes}

\begin{definition}
A \defstyle{$G$-CW complex} is a $G$-space $X$ and an increasing filtration of sub-$G$-spaces $X_n\subseteq X$ with the following properties.
\begin{itemize}
\item $X_0$ is a discrete $G$-set.
\item For each $n$, the subspace $X_{n+1}$ is formed from $X_n$ by attaching cells of the form $G/K \times \diskletter^{n+1}$ along boundary $G$-maps $G/K\times \S{n}\to X_n$, where $K\subseteq G$ is a subgroup that can differ for different cells, and $\diskletter^{n+1}$ and $\S{n}$ have the trivial $G$-action.
\item $X$ is the colimit of the increasing subspaces $X_0\subseteq X_1\subseteq\cdots$.
\end{itemize}
\end{definition}

\begin{remark}
A $G$-CW-complex $X$ may be thought of as an ordinary CW-complex
equipped with a cellular $G$-action $G\times X\to X$ such that, for each $n$-cell $\diskletter^n$ and $g\in G$, the action of $g$ on $X$ either fixes $\diskletter^n$ pointwise or gives a homeomorphism from $\diskletter^n$ to a distinct second $n$-cell. Note also that maps $G/K\times\S{n}\to X$ are in bijective correspondence with maps $\S{n}\to X^K$.
\end{remark}

The following are some illustrative examples of $G$-CW complexes.

\begin{example}
The space $X=\S{1}$ with the action of $G=C_2=\langle g\rangle$ where $g$ acts via rotation by $\pi$.
\begin{itemize}
\item $X_0$ is comprised by a single $0$-cell $\alpha\cong C_2/\{e\}\times \pt$ (indicated by circle).
\item $X_1$ is formed by attaching a single $1$-cell $\beta\cong C_2/\{e\}\times \diskletter^1$ (indicated by line) via the map $\varphi\from C_2/\{e\}\times \S{0}\to X_0$ defined by $\varphi(h,0)=h$ and $\varphi(h,1)=gh$.
\end{itemize}
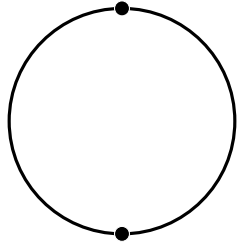
\begin{center}
\begin{tikzpicture}[scale=1.5]
\draw[black,very thick] (0,0) circle (1);
\node[draw=white,fill=black,inner sep=2pt,circle] at (0,1) {};
\node[draw=white,fill=black,inner sep=2pt,circle] at (0,-1) {};
\end{tikzpicture}
\end{center}
\end{example}
\begin{example}
The space $X=\diskletter^2$ with the action of $G=S_3=\langle s,r\rangle$ where $s$ acts via a reflection and $r$ acts via rotation by $2\pi/3$.
\begin{itemize}
\item $X_0$ is comprised by three $0$-cells, $\alpha_0\cong S_3/S_3\times \pt$ (indicated by circle) and $\alpha_1,\alpha_2\cong S_3/\langle s\rangle\times \pt$ (indicated by diamond, star).
\item $X_1$ is formed by attaching three $1$-cells, $\beta_0\cong S_3/\{e\}\times \diskletter^1$ (indicated by solid line) and $\beta_1,\beta_2\cong S_3/\langle s\rangle\times \diskletter^1$ (indicated by dashed, dotted line).
\item $X_2$ is formed by attaching a single $2$-cell $\gamma\cong S_3/\{e\}\times \diskletter^2$ (indicated by shaded area).
\end{itemize}
\begin{center}
\begin{tikzpicture}[scale=1.5]
\filldraw[fill=gray!20,draw=black,very thick] (0,0) circle (1);
\draw[dashed,very thick] (0,0) -- (90:1);
\draw[dashed,very thick] (0,0) -- (210:1);
\draw[dashed,very thick] (0,0) -- (330:1);
\draw[dotted,very thick] (0,0) -- (30:1);
\draw[dotted,very thick] (0,0) -- (150:1);
\draw[dotted,very thick] (0,0) -- (270:1);
\node[draw=white,fill=black,inner sep=2pt,circle] at (0,0) {};
\node[draw=white,fill=black,inner sep=2pt,diamond] at (90:1) {};
\node[draw=white,fill=black,inner sep=2pt,diamond] at (210:1) {};
\node[draw=white,fill=black,inner sep=2pt,diamond] at (330:1) {};
\node[draw=white,fill=black,inner sep=3pt,star, star points=5, star point ratio=.6] at (30:1) {};
\node[draw=white,fill=black,inner sep=3pt,star, star points=5, star point ratio=.6] at (150:1) {};
\node[draw=white,fill=black,inner sep=3pt,star, star points=5, star point ratio=.6] at (270:1) {};
\end{tikzpicture}
\end{center}
\end{example}

\subsection{$G$-cell complexes}

\begin{definition}
Let $V$ be a real $G$-representation. It is a classical result from representation theory that we may put an inner product on $V$ such that the $G$-representation is orthogonal, i.e., elements of $G$ act by norm-preserving maps. Then
\begin{itemize}
\item $\unitsphere{V}$, the unit sphere of $V$,
\item $\unitdisk{V}$, the unit disk of $V$, and
\item $\S{V}$, the one-point compactification of $V$
\end{itemize}
each inherit an action of $G$ from its action on $V$.
\end{definition}

\begin{definition}
Given a right $G$-set $X$ and a left $G$-set $Y$, their \defstyle{balanced product} $X\times_G Y$ is  $X\times Y/{\sim}$ where $(xg,y)\sim (x,gy)$ for $g\in G$.
\end{definition}

\begin{example}
Given a group $G$ and a subgroup $K\subseteq G$, consider $G$ as a right $K$-set. Naturally, $G$ is also a left $G$-set. Then given a left $K$-set $X$, we can form the left $G$-set $G\times_K X$.
\end{example}

\begin{definition}
A \defstyle{$G$-cell complex} is a $G$-space $X$ and an increasing filtration of sub-$G$-spaces $X_n\subseteq X$ with the following properties.
\begin{itemize}
\item $X_0$ is a discrete $G$-set.
\item For each $n$, the subspace $X_{n+1}$ is formed from $X_n$ by attaching cells of the form $G\times_K \unitdisk{V}$ along boundary $G$-maps $G\times_K\unitsphere{V}\to X_n$, where $K\subseteq G$ is a subgroup that can differ for different cells, and $V$ is a real $K$-representation.
\item $X$ is the colimit of the increasing subspaces $X_0\subseteq X_1\subseteq\cdots$.
\end{itemize}
\end{definition}

\begin{remark}
The cells attached to form the $n$th filtration $X_n$ are not required to have any particular dimension, in contrast to the usual convention for CW complexes. 
\end{remark}

Here is an illustrative example of a $G$-cell complex.

\begin{example}
The space $X=\S{V}$, where $V=\R^2$ with the action of $G=S_3=\langle s,r\rangle$ where $s$ acts via a reflection and $r$ acts via rotation by $2\pi/3$.
\begin{itemize}
\item $X_0$ is comprised by the cell $\alpha\cong S_3\times_{S_3}\unitdisk{\mathbb{R}^0}\cong \unitdisk{\mathbb{R}^0}$, i.e., a point (indicated by circle).
\item $X_1$ is formed by attaching a cell $\beta\cong S_3\times_{S_3} \unitdisk{V}\cong \unitdisk{V}$ (indicated by sphere).
\end{itemize}

%
%

\begin{center}
\begin{tikzpicture}
\fill[ball color=gray!60!white] (0,0) circle (1.3); 
\begin{scope}[every label/.style={black!60!white}]
\node[inner sep=0pt,outer sep=0pt,black!60!white] [label=$\beta$]  at (225:1.95) {};
\end{scope}
\begin{scope}[every label/.style={gray!20}]
\node[fill,circle,inner sep=0pt,outer sep=0pt,minimum size=5pt,gray!20] [label=355:$\alpha$]  at (0,1) {};
\end{scope}
\end{tikzpicture}
\end{center}
\end{example}

\section{Mackey functors}


\subsection{Burnside category}

\begin{definition}
\label{def:orbit-category}
An \defstyle{orbit} of a finite group $G$ is a $G$-set of the form $G/H$ for a subgroup $H\subseteq G$. The \defstyle{orbit category} $\orbitcategory{G}$ of a finite group $G$ has the orbits of $G$ as its objects, and morphisms are morphisms of $G$-sets.
\end{definition}

\begin{remark}
\label{rem:subconjugacy}
There exists a map $G/H\to G/K$ in $\orbitcategory{G}$ if and only if $H$ is subconjugate to $K$, i.e., there exists some $g\in G$ for which $g^{-1}Hg\subseteq K$ (note that if $\alpha\from G/H\to G/K$ has $\alpha(eH)=gK$, then $g^{-1}Hg\subseteq K$).
\end{remark}

\begin{definition}
\label{def:coefficient-system}
A covariant \defstyle{coefficient system} is a functor $\orbitcategory{G}\to \Ab$, and a contravariant coefficient system is a functor $\orbitcategory{G}^\op\to \Ab$.
\end{definition}

One important example is the following:

\begin{example}
\label{ex:g-homotopy-group}
Given a based $G$-space $X$, we can make a homotopy group coefficient system by defining $\underline{\pi\!}_{\,n}(X)(G/H)=\pi_n(X^H)$, and where the map $X^K\to X^H$ induced by $\alpha\from G/H\to G/K$ with $\alpha(eH)=gK$ sends $x\in X^K$ to $gx\in X^H$.
\end{example}

For an overview of the equivariant stable homotopy category, 
see \cite[\sectionsymb 2]{greenless_may}. In particular, recall that for any $G$-space $X$, we can form a suspension $G$-spectrum $\Sigma^\infty X$.


\begin{definition}
\label{def:burnside-original}
The \defstyle{Burnside category} $\burnside{G}$ of a group $G$ is the full subcategory of the equivariant stable homotopy category on the objects $\Sigma^\infty(\gsetstyle{b}_{+})$, where $\gsetstyle{b}$ is a finite $G$-set.
\end{definition}

\begin{remark}
Equivalently, one can define $\burnside{G}$ to be the category whose objects are  finite $G$-sets $\gsetstyle{b}$ and whose morphisms are stable $G$-maps
\[\burnside{G}(\gsetstyle{b},\gsetstyle{c})=\{\gsetstyle{b}_{+},\gsetstyle{c}_{+}\}_{G}=\operatornamewithlimits{colim}\limits_{V\subset U}[\Sigma^V(\gsetstyle{b}_{+}),\Sigma^V(\gsetstyle{c}_{+})]_G\]
where $U$ is a given complete $G$-universe and the colimit is over the finite dimensional sub-$G$-representations of $U$.
\end{remark}

The following definition gives us a more concrete, useful way of working with the Burnside category.

\begin{definition}
\label{def:useful-burnside}
For any two finite $G$-sets $\gsetstyle{x}$ and $\gsetstyle{y}$, let $S(\gsetstyle{x},\gsetstyle{y})$ be the isomorphism classes of spans $\gsetstyle{x}\leftarrow \gsetstyle{z}\to\gsetstyle{y}$. 
Note that $S(\gsetstyle{x},\gsetstyle{y})$ is an abelian monoid under the operation of $\sqcup$, which takes spans $\gsetstyle{x}\leftarrow \gsetstyle{z}\to\gsetstyle{y}$ and $\gsetstyle{x}\leftarrow \gsetstyle{w}\to\gsetstyle{y}$ and produces the span $\gsetstyle{x}\leftarrow \gsetstyle{z}\sqcup\gsetstyle{w}\to\gsetstyle{y}$. 

Define the category $\burnside{G}'$ to have as its objects the finite $G$-sets, and for its morphisms, $\burnside{G}'(\gsetstyle{x},\gsetstyle{y})$ is the abelian group formed by applying the Grothendieck construction to $S(\gsetstyle{x},\gsetstyle{y})$. The composition of $\gsetstyle{x}\leftarrow \gsetstyle{u}\to\gsetstyle{y}$ and $\gsetstyle{y}\leftarrow \gsetstyle{v}\to\gsetstyle{z}$ is the span $\gsetstyle{x}\leftarrow \gsetstyle{w}\to\gsetstyle{z}$ formed by taking the pullback:
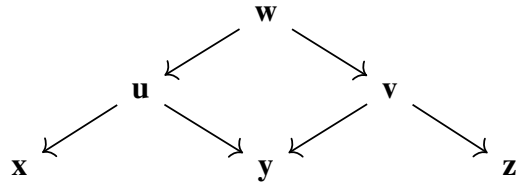
\begin{center}
\begin{tikzcd}[row sep = 0.5cm]
{} &  & \gsetstyle{w} \ar{rd} \ar{ld} & & \\
 & \gsetstyle{u}\ar{rd} \ar{ld} & & \gsetstyle{v}\ar{rd} \ar{ld} & \\
\gsetstyle{x} & & \gsetstyle{y} & & \gsetstyle{z}
\end{tikzcd}
\end{center}
\end{definition}

The theorem from \cite{alaska} below establishes the relationship between the two categories. From now on we will refer to either of the categories from \autoref{def:burnside-original} or  \autoref{def:useful-burnside} as $\burnside{G}$.

\begin{theorem}
The Burnside category $\burnside{G}$ is equivalent to the category $\burnside{G}'$ from \autoref{def:useful-burnside}.
\end{theorem}

Some useful classes of morphisms in $\burnside{G}$ are the restrictions and transfers.

\begin{definition}
For any map of finite $G$-sets $\alpha\from \gsetstyle{x}\to\gsetstyle{y}$, we will define the \defstyle{restriction} $\restriction{\alpha}\in\burnside{G}(\gsetstyle{x},\gsetstyle{y})$ associated to $\alpha$ to be the span
\begin{center}
\begin{tikzcd}[row sep = 0.5cm]
 & \gsetstyle{x}\ar{rd}{\alpha} \ar{ld}[swap]{\id} & \\
\gsetstyle{x} & & \gsetstyle{y}
\end{tikzcd}
\end{center}
and the \defstyle{transfer} $\transfer{\alpha}\in\burnside{G}(\gsetstyle{y},\gsetstyle{x})$ associated to $\alpha$ to be the span
\begin{center}
\begin{tikzcd}[row sep = 0.5cm]
 & \gsetstyle{x}\ar{rd}{\id} \ar{ld}[swap]{\alpha} & \\
\gsetstyle{y} & & \gsetstyle{x}
\end{tikzcd}
\end{center}
\end{definition}

\begin{remark}
\label{rem:orbit-burnside-embedding}
There are two embeddings of the orbit category $\orbitcategory{G}$ into the Burnside category $\burnside{G}$, one covariant and the other contravariant; on objects an orbit $G/H$ is sent to itself, and a morphism $\alpha\from G/H\to G/K$ in $\orbitcategory{G}$ can be sent to either $\restriction{\alpha}$ or $\transfer{\alpha}$ for the covariant or contravariant embedding, respectively.
\end{remark}
\begin{remark}
\label{rem:span-decomp}
Observe that the disjoint union $\sqcup$ is both the product and coproduct in $\burnside{G}$, and that any finite $G$-set is a disjoint union of orbits; therefore, any span $\gsetstyle{x}\leftarrow \gsetstyle{z}\to\gsetstyle{y}$ can be decomposed into spans of the form $G/H \xleftarrow{\alpha} G/L\xrightarrow{\beta} G/K$. In fact, we can decompose one step further and note that such a span is itself the composition  $\transfer{\alpha}\circ\restriction{\beta}$.
\end{remark}


\subsection{Mackey functors}

\begin{definition}
A \defstyle{Mackey functor} is an additive functor $\burnside{G}^{\op}\to \Ab$. The collection of Mackey functors for a group $G$ forms a category $\mackey{G}$, whose morphisms are the natural transformations of such functors.
\end{definition}

%


\begin{remark}
\label{rem:mackey-additive}
Because Mackey functors are additive, a Mackey functor $M$ is determined by its values $M(G/H)$ on orbits, and similarly a morphism $M \to N$ of Mackey functors is determined by the maps $M(G/H)\to N(G/H)$. Moreover, to know what a Mackey functor does on morphisms, by \autoref{rem:span-decomp} it suffices to know what it does on restrictions $\restriction{\alpha}$ and transfers $\transfer{\alpha}$ of maps $\alpha\from G/H\to G/K$ between orbits. Thus, by \autoref{rem:orbit-burnside-embedding}, the Mackey functor $M\from\burnside{G}^\op\to\Ab$ determines a covariant coefficient system $M_*\from\orbitcategory{G}\to\Ab$ and a contravariant coefficient system $M^*\from\orbitcategory{G}^\op\to\Ab$, and a Mackey functor is determined by such a pair of coefficient systems, as long as they agree on objects and satisfy a compatibility condition that arises from composition in $\burnside{G}$.
\end{remark}

Here we give two simple examples of Mackey functors.

%

\begin{example}
\label{def:constant-mackey}
For any abelian group $C$, the \defstyle{constant Mackey functor} of $C$ is the Mackey functor $\underline{C}\from\burnside{G}^\op\to\Ab$ where $\underline{C}(G/H)=C$ for any orbit $G/H$, and for any $\alpha\from G/H\to G/K$, the behavior on the restriction $\underline{C}(\restriction{\alpha})\from C\to C$ is the identity, while the behavior on the transfer $\underline{C}(\transfer{\alpha})\from C\to C$ is multiplication by $|H|/|K|$ (which is an integer because, by \autoref{rem:subconjugacy}, $H$ is subconjugate to $K$). 
\end{example}

\begin{example}
\label{def:burnside-mackey}
The \defstyle{Burnside ring Mackey functor} $A\from\burnside{G}^\op\to\Ab$ is $\burnside{G}({-},G/G)$, the representable functor represented by the $G$-orbit $G/G$.
\end{example}

Although any Mackey functor can be seen as having two underlying coefficient systems, not every coefficient system extends to a Mackey functor, which is an important point considering \autoref{thm:extends-to-mackey}.


Mackey functors also appear outside of algebraic topology, for example in \cite[\sectionsymb 2, p.298]{boltje}:

\begin{example}
Let $L/K$ be a finite Galois extension of number fields with $G=\mathrm{Gal}(L/K)$, and let $\mathrm{cl}(F)$ denote the ideal class group of a number field $F$. We can form a $G$-Mackey functor $M$ with $M(G/H)=\mathrm{cl}(L^H)$. Given a map $\alpha\from G/H\to G/K$ with $\alpha(eH)=gK$, for restriction $M(\restriction{\alpha})$ we define the map $\mathrm{cl}(L^H)\to\mathrm{cl}(L^K)$ by the appropriate combination of acting by $g$ and extending fractional ideals, and for transfer $M(\transfer{g})$ we define the map $\mathrm{cl}(L^K)\to\mathrm{cl}(L^H)$ by the appropriate combination of acting by $g$ and taking norms.
\end{example}


\subsection{Box product of Mackey functors}

\begin{definition}
The \defstyle{box product $M\boxprod N$} of two Mackey functors $M,N\from \burnside{G}^\op\to\Ab$ is defined to be the Day tensor product of such functors. That is, if we let $M\externalboxprod N$ be the ``external'' product
\[M\externalboxprod N\from \burnside{G}^\op\times\burnside{G}^\op\to \Ab,\qquad (M\externalboxprod N)(\gsetstyle{x},\gsetstyle{y})=M(\gsetstyle{x})\otimes N(\gsetstyle{y})\]
then the box product $M\boxprod N$ is the left Kan extension of $M\externalboxprod N$ along the cartesian product functor,
\begin{center}
\begin{tikzcd}
\burnside{G}^\op\times\burnside{G}^\op \ar{rr}{M\externalboxprod N} \ar{rd}[swap]{\times} & & \kmod\\
& \burnside{G}^\op \ar{ru}[swap]{M\boxprod N}
\end{tikzcd}
\end{center}
Thus, $M\boxprod N$ is characterized by the universal property that for any $P\from \burnside{G}^\op\to \Ab$,
\[[\burnside{G}^\op,\Ab](M\boxprod N,P)\cong[\burnside{G}^\op\times\burnside{G}^\op,\Ab](M\externalboxprod N,P\circ{\times})\]
\end{definition}

\begin{remark}
The category $\mackey{G}$ is a symmetric monoidal category under $\boxprod$, with unit $A$.
\end{remark}

\begin{definition}
\label{def:shifted-mackey}
Let $M$ be a Mackey functor and let $\gsetstyle{b}$ be a $G$-set. Then the \defstyle{shifted Mackey functor} $M_{\gsetstyle{b}}$ is defined on objects by
\[M_{\gsetstyle{b}}(\gsetstyle{c})=M(\gsetstyle{b}\times\gsetstyle{c})\]
and then with the natural behavior on morphisms. An alternative characterization of this is that $M_{\gsetstyle{b}}=M\boxprod A_{\gsetstyle{b}}$, where $A$ is the Burnside ring Mackey functor.
\end{definition}

\subsection{Mackey functors with extra structure}

\begin{definition}
A \defstyle{Green functor} is a monoid in the category $\mackey{G}$, i.e., a Mackey functor $T\from\burnside{G}^\op\to\Ab$ together with natural transformations $\mu\from T\boxprod T\tonatural T$ and $\eta\from A\tonatural T$, satisfying the standard commutative diagrams
\begin{center}
\begin{tikzcd}
A\boxprod T\ar{r}{\eta\boxprod \alpha} \ar{rd}[swap]{\cong} & T\boxprod T \ar{d}{\mu} & T\boxprod A \ar{l}[swap]{{\id}\boxprod \eta} \ar{ld}{\cong} \\
& T 
\end{tikzcd}
\qquad
\begin{tikzcd}
T\boxprod T\boxprod T \ar{d}[swap]{\id\boxprod\mu} \ar{r}{\mu\boxprod\id} & T\boxprod T \ar{d}{\mu}\\
T\boxprod T \ar{r}[swap]{\mu} & T
\end{tikzcd}
\end{center}
A Green functor is \defstyle{commutative} when the following diagram also commutes:
\begin{center}
\begin{tikzcd}
T\boxprod T \ar{rr}{\text{swap}} \ar{rd}[swap]{\mu} & & T \boxprod T \ar{ld}{\mu}\\
 & T
\end{tikzcd}
\end{center}
\end{definition}

\begin{definition}
A \defstyle{module over a Green functor} $T$ is a Mackey functor $M\from\burnside{G}^\op\to\Ab$ with an action natural transformation $\phi\from T\boxprod M\Rightarrow M$, satisfying the standard commutative diagrams
\begin{center}
\begin{tikzcd}
A\boxprod M \ar{rr}{\eta\boxprod\id} \ar{rd}[swap]{\cong} & & T \boxprod M \ar{ld}{\phi}\\
 & M
\end{tikzcd}
\qquad
\begin{tikzcd}
T\boxprod T\boxprod M \ar{r}{\mu\boxprod\id} \ar{d}[swap]{\id\boxprod\phi} & T \boxprod M \ar{d}{\phi}\\
T\boxprod M \ar{r}[swap]{\phi} & M
\end{tikzcd}
\end{center}
\end{definition}

\begin{definition}
A module over a Green functor $T$ is \defstyle{free} when it is isomorphic to $T\boxprod A_{\gsetstyle{b}}\cong T_{\gsetstyle{b}}$ for some $G$-set $\gsetstyle{b}$. (See \cite[\sectionsymb 2.2.7]{megan} for justification.)
\end{definition}

\begin{definition}
An \defstyle{$\RO(G)$-graded Mackey functor} $M^*$ is just a collection of Mackey functors $M^\alpha$ for each $\alpha\in \RO(G)$. We can define a graded box product $M^*\boxprod N^*$ by
\[(M^*\boxprod N^*)^\alpha=\bigoplus_{\beta_1+\beta_2=\alpha}(M^{\beta_1}\boxprod N^{\beta_2})\]
and the unit for $\boxprod$ is the graded Burnside ring Mackey functor $A^*$, which is just $A$ concentrated in degree $0$. The category of graded Mackey functors is then a symmetric monoidal category under $\boxprod$. One can then define $\RO(G)$-graded Green functors and $\RO(G)$-graded modules over $\RO(G)$-graded Green functors using the same diagrams as above.
\end{definition}

\begin{remark}
For more details about what it means for a graded Green functor to be \textit{commutative}, see the discussion in \cite[\sectionsymb2.2.5 and \sectionsymb2.5.1]{megan}.
\end{remark}

\begin{definition}
\label{def:tensoring-mackey-with-group}
If $M^*$ is an $\RO(G)$-graded Mackey functor and $B^*$ is a $\Z$-graded abelian group, then $M^*\otimes B^*$ is the $\RO(G)$-graded Mackey functor defined by
\[(M^*\otimes B^*)^\alpha=\bigoplus_{\beta+n=\alpha}M^\beta\otimes B^n.\]
\end{definition}

\subsection{$C_p$-Mackey functors}

\begin{remark}
\label{rem:dot-and-sun}
Observe that there are only two $C_p$-orbits: the singleton $C_p/C_p$, and the orbit $C_p/\{e\}$ that is isomorphic to $C_p$ itself. Throughout this work, we will often follow a notational convention from \cite{megan} regarding them: instead of $C_p/C_p$ we will write $\GG$, and instead of writing $C_p/\{e\}$ we will write $\Ge$, with the symbols $\GG$ and $\Ge$ intended to represent the structure of the relevant orbits.
\end{remark}

\begin{remark}
In addition to the fact that there are only two $C_p$-orbits, $\GG$ and $\Ge$, the morphisms between them in $\burnside{C_p}$ are also easy to describe. Observe that there is a single map $\rho\from \Ge\to \GG$, and no maps $\GG\to\Ge$. There are $p$ maps $\Ge\to\Ge$, corresponding to the actions of each element of $C_p$. Thus, by \autoref{rem:mackey-additive}, to describe a $C_p$-Mackey functor $M$ on objects, we only need to know $M(\GG)$ and $M(\Ge)$, and to describe it on morphisms, we only need to know how $M$ acts on $\restriction{\rho}$, $\transfer{\rho}$, and the $p$ elements $\Ge\xleftarrow{\id}\Ge\xrightarrow{g}\Ge$ in $\burnside{G}(\Ge,\Ge)$. The latter piece of information manifests itself as a $C_p$-action on $M(\Ge)$ (see more in \cite[\sectionsymb2.4.1]{megan}).

The following notational convention for $C_p$-Mackey functors originated in \cite{lewis_complex}:
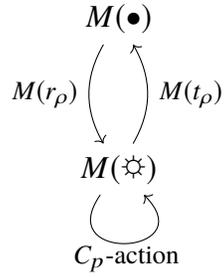
\begin{center}
\begin{tikzpicture}
\node (a) at (0,2) {$M(\GG)$};
\node (b) at (0,0) {$M(\Ge)$};
\draw[->] (a) [bend right] to node[left] {$\scriptstyle M(\restriction{\rho})$} (b);
\draw[->] (b) [bend right] to node[right] {$\scriptstyle M(\transfer{\rho})$} (a);
\draw[->] (b) [out=235,in=315,looseness=5] to node[pos=0.5,shift={(0,-0.2)}] {$\scriptstyle C_p\text{-action}$} (b);
\end{tikzpicture}
\end{center}
\end{remark}

With $G=C_p$ for a prime $p$, the following Mackey functors will play an important role.

\begin{definition}
\label{def:important-mackeys}
The $C_p$-Mackey functors $A$, $A[d]$, $R$, $L$, $R_{-}$, $L_{-}$, $\langle C\rangle$.
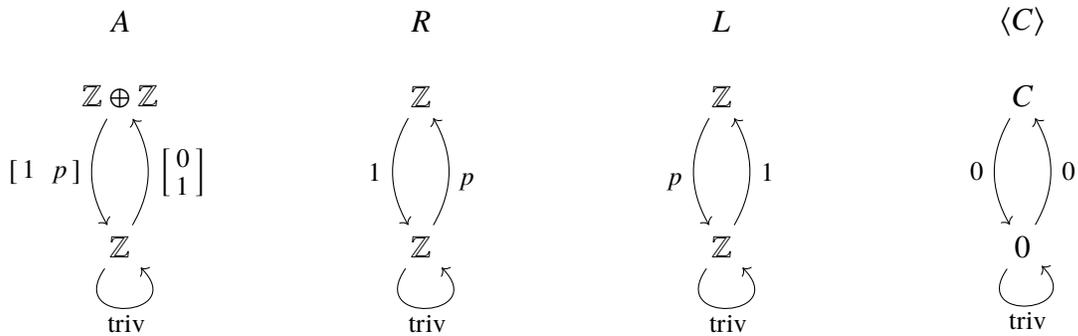
\begin{center}
\begin{tikzpicture}
\begin{scope}[shift={(0,0)}]
\path (-1.5,-1.5) rectangle (1.5,3.5);
\node at (0,3) {$A$};
\node (a) at (0,2) {$\Z\oplus\Z$};
\node (b) at (0,0) {$\Z$};
\draw[->] (a) [bend right] to node[left] {$\scriptstyle\left[\begin{smallmatrix}
1\, & \,p
\end{smallmatrix}\right]$} (b);
\draw[->] (b) [bend right] to node[right] {$\scriptstyle\left[\begin{smallmatrix}
0\\ 1^{\rule{0pt}{5pt}}
\end{smallmatrix}\right]$} (a);
\draw[->] (b) [out=235,in=315,looseness=5] to node[pos=0.5,shift={(0,-0.2)}] {$\scriptstyle\text{triv}$} (b);
\end{scope}
\begin{scope}[shift={(4,0)}]
\path (-1.5,-1.5) rectangle (1.5,3.5);
\node at (0,3) {$R$};
\node (a) at (0,2) {$\Z$};
\node (b) at (0,0) {$\Z$};
\draw[->] (a) [bend right] to node[left] {$\scriptstyle 1$} (b);
\draw[->] (b) [bend right] to node[right] {$\scriptstyle p\strut$} (a);
\draw[->] (b) [out=235,in=315,looseness=5] to node[pos=0.5,shift={(0,-0.2)}] {$\scriptstyle\text{triv}$} (b);
\end{scope}
\begin{scope}[shift={(8,0)}]
\path (-1.5,-1.5) rectangle (1.5,3.5);
\node at (0,3) {$L$};
\node (a) at (0,2) {$\Z$};
\node (b) at (0,0) {$\Z$};
\draw[->] (a) [bend right] to node[left] {$\scriptstyle p\strut$} (b);
\draw[->] (b) [bend right] to node[right] {$\scriptstyle 1$} (a);
\draw[->] (b) [out=235,in=315,looseness=5] to node[pos=0.5,shift={(0,-0.2)}] {$\scriptstyle\text{triv}$} (b);
\end{scope}
\begin{scope}[shift={(12,0)}]
\path (-1.5,-1.5) rectangle (1.5,3.5);
\node at (0,3) {$\langle C\rangle$};
\node (a) at (0,2) {$C$};
\node (b) at (0,0) {$0$};
\draw[->] (a) [bend right] to node[left] {$\scriptstyle 0$} (b);
\draw[->] (b) [bend right] to node[right] {$\scriptstyle 0$} (a);
\draw[->] (b) [out=235,in=315,looseness=5] to node[pos=0.5,shift={(0,-0.2)}] {$\scriptstyle\text{triv}$} (b);
\end{scope}
\end{tikzpicture}
\end{center}
\end{definition}

The ``twisted Burnside'' Mackey functor arises in $\HHreduced_{C_p}^*(\S{0}\coeffs*{A})$ when $p$ is odd.

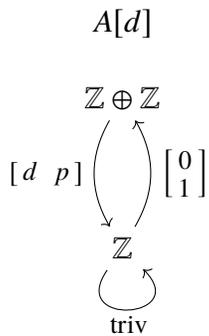
\begin{center}
\begin{tikzpicture}
\begin{scope}[shift={(0,0)}]
\path (-1.5,-1.5) rectangle (1.5,3.5);
\node at (0,3) {$A[d]$};
\node (a) at (0,2) {$\Z\oplus\Z$};
\node (b) at (0,0) {$\Z$};
\draw[->] (a) [bend right] to node[left] {$\scriptstyle\left[\begin{smallmatrix}
d\, & \,p
\end{smallmatrix}\right]$} (b);
\draw[->] (b) [bend right] to node[right] {$\scriptstyle\left[\begin{smallmatrix}
0\\ 1^{\rule{0pt}{5pt}}
\end{smallmatrix}\right]$} (a);
\draw[->] (b) [out=235,in=315,looseness=5] to node[pos=0.5,shift={(0,-0.2)}] {$\scriptstyle\text{triv}$} (b);
\end{scope}
\end{tikzpicture}
\end{center}

The ``signed'' versions of $R$ and $L$ are arise in $\HHreduced_{C_2}^*(\S{0}\coeffs*{A})$.

\newcommand{\ghostquot}{\makebox[0pt][l]{$\scriptstyle \smash{\text{quotient}}$}}
\begin{center}
\begin{tikzpicture}
\begin{scope}[shift={(0,0)}]
\path (-1.5,-1.5) rectangle (1.5,3.5);
\node at (0,3) {$R_{-}$};
\node (a) at (0,2) {$0$};
\node (b) at (0,0) {$\Z$};
\draw[->] (a) [bend right] to node[left] {$\scriptstyle 0$} (b);
\draw[->] (b) [bend right] to node[right] {$\scriptstyle 0\strut$} (a);
\draw[->] (b) [out=235,in=315,looseness=5] to node[pos=0.5,shift={(0,-0.2)}] {$\scriptstyle -1$} (b);
\end{scope}
\begin{scope}[shift={(4,0)}]
\path (-1.5,-1.5) rectangle (1.5,3.5);
\node at (0,3) {$L_{-}$};
\node (a) at (0,2) {$\Z/2$};
\node (b) at (0,0) {$\Z$};
\draw[->] (a) [bend right] to node[left] {$\scriptstyle 0\strut$} (b);
\draw[->] (b) [bend right] to node[right] {$\scriptstyle\text{quot}$} (a);
\draw[->] (b) [out=235,in=315,looseness=5] to node[pos=0.5,shift={(0,-0.2)}] {$\scriptstyle -1$} (b);
\end{scope}
\end{tikzpicture}
\end{center}

\section{\texorpdfstring{$\RO(G)$-graded equivariant cohomology}{RO(G)-graded equivariant cohomology}}





We forgo a more motivated treatment of $\RO(G)$-graded equivariant cohomology and simply list the axioms and certain basic facts we will need to refer to. We use the formulation in \cite[\sectionsymb 2.3.2]{megan}.

Let $G\catnamestyle{-Top}_{*}$ denote the category of based $G$-spaces, where the basepoint is required to be $G$-fixed. Then  the homotopy category  $\catnamestyle{ho-}G\catnamestyle{-Top}_{*}$ is formed by localizing at the collection of weak $G$-equivalences, i.e., the $G$-equivariant maps $f\from X\to Y$ for which $f_*^H\from\underline{\pi}_n(X)\to\underline{\pi}_n(Y)$ is an isomorphism of coefficient systems for all $n$ (see \autoref{ex:g-homotopy-group}).

\begin{definition}
A \defstyle{reduced ordinary equivariant cohomology theory} $\HHreduced_{G}^*$ indexed on $\RO(G)$ consists of functors $\HHreduced_G^\alpha\from \catnamestyle{ho-}G\catnamestyle{-Top}_{*}\to \mackey{G}$ satisfying the following properties:
\begin{itemize}
\item \textit{Weak equivalence}: Weak $G$-equivalences induce isomorphisms on $\Hhreduced_G^*$.
\item \textit{Exactness}: If $i\from A\to X$ is a $G$-cofibration with cofiber $X/A$, then for each $\alpha\in\RO(G)$
\[\HHreduced_{G}^\alpha(X/A)\to \HHreduced_{G}^\alpha(X)\to \HHreduced_{G}^\alpha(A)\]
is an exact sequence in the abelian category $\mackey{G}$.
\item \textit{Additivity}: If $X=\bigvee_i X_i$ as based $G$-spaces, then the inclusions $X_i\to X$ induce an isomorphism $\HHreduced_{G}^*(X)\to \prod_i \HHreduced_{G}^*(X_i)$.
\item \textit{Suspension}: For each $\alpha\in\RO(G)$ and real $G$-representation $V$, there is a natural isomorphism
\[\Sigma^V\from \HHreduced_{G}^\alpha(X)\to\HHreduced_{G}^{\alpha+V}(\Sigma^VX)=\HHreduced_{G}^{\alpha+V}(\S{V}\wedge X).\]
\item \textit{Dimension}: $\HHreduced_{G}^n(\S{0})=0$ for all integers $n\in\Z$, $n\neq 0$.
\item \textit{Bookkeeping}: $\Sigma^0$ is the identity natural transformation, and the suspension isomorphism 
\[\Sigma^V\from\HHreduced_{G}^\alpha(X)\to \HHreduced_{G}^{\alpha+V}(\Sigma^VX)\]
is covariantly natural in $V$ and contravariantly natural in $X$. Moreover, if $U,V,W$ are real $G$-representations with $V\cong W$, and $f\from \S{V}\to\S{W}$ is the stable homotopy class of the map associated to a particular isomorphism $V\to W$, then $\Sigma^U\Sigma^V\cong \Sigma^{U\oplus V}$, and the following diagrams commute:
\begin{center}
\begin{tikzcd}
\HHreduced_{G}^\alpha(X) \ar{r}{\Sigma^V} \ar{d}[swap]{\Sigma^{W}} & \HHreduced_{G}^{\alpha+V}(\Sigma^VX) \ar{d}{\HHreduced_{G}^{\id+f}(\id)}\\
\HHreduced_{G}^{\alpha+W}(\Sigma^WX)\ar{r}[swap]{\Sigma^f\id} & \HHreduced_{G}^{\alpha+W}(\Sigma^VX)
\end{tikzcd}
\end{center}
\end{itemize}
\end{definition}

%
%
%
%
%

One can define ordinary $\Z$-graded equivariant cohomology using a cellular chain complex, with coefficients that are a coefficient system. The following result from \cite{lewis_may_mcclure} establishes when such a cohomology theory comes from one with an $\RO(G)$-grading.

\begin{theorem}
\label{thm:extends-to-mackey}
The ordinary $\Z$-graded cohomology $\Hhreduced_G^*({-};M)$ extends to an $\RO(G)$-grading if and only if $M$ is the underlying contravariant coefficient system of a Mackey functor.
\end{theorem}

\begin{remark}
Given a map of $G$-spaces $f:X\to Y$, we will abuse the notation $f^*$ and let it denote the induced map in cohomology at any level of generality, e.g.,
\begin{itemize}
\item the map of $\RO(G)$-graded $\HHreduced_{G}^*(\S{0};T)$-algebras $\HHreduced_{G}^*(Y;T)\to\HHreduced_{G}^*(X;T)$
\item the map of Mackey functors $\HHreduced_{G}^\alpha(Y;M)\to\HHreduced_{G}^\alpha(X;M)$
\item the map of abelian groups $\HHreduced_{G}^\alpha(Y;M)(\gsetstyle{b})\to\HHreduced_{G}^\alpha(X;M)(\gsetstyle{b})$
\end{itemize}
where $T$ is a Green functor, $M$ is a Mackey functor, $\alpha\in\RO(G)$, and $\gsetstyle{b}$ is a $G$-set.
\end{remark}

\section{The structure of \texorpdfstring{$\HHreduced_{G}^*(\S{0}\coeffs*{A})$}{H\_G*(S0)}}

The computation of the $\RO(G)$-graded equivariant cohomology of a point is a non-trivial undertaking. Stong made the first such computation in unpublished correspondence \cite{stong_letters}, though the results were finally published in \cite[Appendix]{lewis_complex}, with $G=C_p$ for $p$ prime and coefficients in the Burnside ring Mackey functor $A$. That was the state of this question until the groundbreaking paper \cite{hhr}, where in the course of their work, Hill, Hopkins, and Ravenel performed some calculations of the $\RO(C_{2^n})$-graded cohomology of a point, although with the coefficients being the constant Mackey functor $\underline{\Z}$. More recently, Basu and Ghosh computed the additive structure of the $\RO(C_{pq})$-graded cohomology of a point in \cite{basu_ghosh} using methods similar to Stong's, with the situation made more intricate due to the additional complexity of $\RO(C_{pq})$. 

\subsection{The case of $C_2$}

We will describe here some of the structure of the cohomology ring $\HHreduced_{C_2}^*(\S{0}\coeffs*{A})$. For a full accounting of it, see the description in \cite{megan}, or the proof in \cite{lewis_complex}.

\begin{definition}
Let $\signrep$ be a particular copy of the sign representation of $C_2$ of dimension one. Then every $\alpha\in\RO(C_2)$ has a representative of the form $m+n\signrep$ for $m,n\in\Z$. 

\end{definition}

\begin{figure}
\newcommand{\ghostminus}{_{-}}
\begin{center}
\begin{tikzpicture}[xscale=1.0,yscale=0.85, every node/.style={fill=white,inner sep=3pt,outer sep=4pt}]
\path[use as bounding box] (-6.9,-5.1) rectangle (6.9,5.9);
\draw[gray!50,<->] (-6,0) -- (6,0);
\draw[gray!50,<->] (0,-5) -- (0,5);

\node at (0,5.5) {$|\alpha|$};
\node at (6.5,0) {$|\alpha^{C_2}|$};

\node[text=gray] at (-5,4)  {$\cdot$};
\node[text=gray] at (-5,3)  {$\cdot$};
\node[text=gray] at (-5,2)  {$\cdot$};
\node[text=gray] at (-5,1)  {$\cdot$};
\node[text=gray] at (-5,-4) {$\cdot$};
\node[text=gray] at (-5,-3) {$\cdot$};
\node[text=gray] at (-5,-2) {$\cdot$};
\node[text=gray] at (-5,-1) {$\cdot$};

\node at (-4,4) {\strut$\langle \mathbb{Z}/2\rangle$};
\node at (-4,3) {\strut$\langle \mathbb{Z}/2\rangle$};
\node at (-4,2) {\strut$\langle \mathbb{Z}/2\rangle$};
\node at (-4,1) {\strut$\langle \mathbb{Z}/2\rangle$};

\node[text=gray] at (-4,-4) {$\cdot$};
\node[text=gray] at (-4,-3) {$\cdot$};
\node[text=gray] at (-4,-2) {$\cdot$};
\node[text=gray] at (-4,-1) {$\cdot$};

\node[text=gray] at (-3,4)  {$\cdot$};
\node[text=gray] at (-3,3)  {$\cdot$};
\node[text=gray] at (-3,2)  {$\cdot$};
\node[text=gray] at (-3,1)  {$\cdot$};
\node[text=gray] at (-3,-4) {$\cdot$};
\node[text=gray] at (-3,-3) {$\cdot$};
\node[text=gray] at (-3,-2) {$\cdot$};
\node[text=gray] at (-3,-1) {$\cdot$};

\node at (-2,4) {\strut$\langle \mathbb{Z}/2\rangle$};
\node at (-2,3) {\strut$\langle \mathbb{Z}/2\rangle$};
\node at (-2,2) {\strut$\langle \mathbb{Z}/2\rangle$};
\node at (-2,1) {\strut$\langle \mathbb{Z}/2\rangle$};

\node[text=gray] at (-2,-4) {$\cdot$};
\node[text=gray] at (-2,-3) {$\cdot$};
\node[text=gray] at (-2,-2) {$\cdot$};
\node[text=gray] at (-2,-1) {$\cdot$};

\node[text=gray] at (-1,4)  {$\cdot$};
\node[text=gray] at (-1,3)  {$\cdot$};
\node[text=gray] at (-1,2)  {$\cdot$};
\node[text=gray] at (-1,1)  {$\cdot$};
\node[text=gray] at (-1,-4) {$\cdot$};
\node[text=gray] at (-1,-3) {$\cdot$};
\node[text=gray] at (-1,-2) {$\cdot$};
\node[text=gray] at (-1,-1) {$\cdot$};

\node at (-5,0) {\strut$R\ghostminus$};
\node at (-4,0) {\strut$R$};
\node at (-3,0) {\strut$R\ghostminus$};
\node at (-2,0) {\strut$R$};
\node at (-1,0) {\strut$R\ghostminus$};
\node at (0,0)  {\strut$A$};
\node at (1,0)  {\strut$R\ghostminus$};
\node at (2,0)  {\strut$L$};
\node at (3,0)  {\strut$L\ghostminus$};
\node at (4,0)  {\strut$L$};
\node at (5,0)  {\strut$L\ghostminus$};

\node at (0,4) {\strut$\langle \mathbb{Z}\rangle$};
\node at (0,3) {\strut$\langle \mathbb{Z}\rangle$};
\node at (0,2) {\strut$\langle \mathbb{Z}\rangle$};
\node at (0,1) {\strut$\langle \mathbb{Z}\rangle$};
\node at (0,-4) {\strut$\langle \mathbb{Z}\rangle$};
\node at (0,-3) {\strut$\langle \mathbb{Z}\rangle$};
\node at (0,-2) {\strut$\langle \mathbb{Z}\rangle$};
\node at (0,-1) {\strut$\langle \mathbb{Z}\rangle$};

\node[text=gray] at (1,4)  {$\cdot$};
\node[text=gray] at (1,3)  {$\cdot$};
\node[text=gray] at (1,2)  {$\cdot$};
\node[text=gray] at (1,1)  {$\cdot$};
\node[text=gray] at (1,-4) {$\cdot$};
\node[text=gray] at (1,-3) {$\cdot$};
\node[text=gray] at (1,-2) {$\cdot$};
\node[text=gray] at (1,-1) {$\cdot$};

\node[text=gray] at (2,4)  {$\cdot$};
\node[text=gray] at (2,3)  {$\cdot$};
\node[text=gray] at (2,2)  {$\cdot$};
\node[text=gray] at (2,1)  {$\cdot$};
\node[text=gray] at (2,-4) {$\cdot$};
\node[text=gray] at (2,-3) {$\cdot$};
\node[text=gray] at (2,-2) {$\cdot$};
\node[text=gray] at (2,-1) {$\cdot$};

\node[text=gray] at (3,4)  {$\cdot$};
\node[text=gray] at (3,3)  {$\cdot$};
\node[text=gray] at (3,2)  {$\cdot$};
\node[text=gray] at (3,1)  {$\cdot$};
\node at (3,-4) {\strut$\langle \mathbb{Z}/2\rangle$};
\node at (3,-3) {\strut$\langle \mathbb{Z}/2\rangle$};
\node at (3,-2) {\strut$\langle \mathbb{Z}/2\rangle$};
\node at (3,-1) {\strut$\langle \mathbb{Z}/2\rangle$};

\node[text=gray] at (4,4)  {$\cdot$};
\node[text=gray] at (4,3)  {$\cdot$};
\node[text=gray] at (4,2)  {$\cdot$};
\node[text=gray] at (4,1)  {$\cdot$};
\node[text=gray] at (4,-4) {$\cdot$};
\node[text=gray] at (4,-3) {$\cdot$};
\node[text=gray] at (4,-2) {$\cdot$};
\node[text=gray] at (4,-1) {$\cdot$};

\node[text=gray] at (5,4)  {$\cdot$};
\node[text=gray] at (5,3)  {$\cdot$};
\node[text=gray] at (5,2)  {$\cdot$};
\node[text=gray] at (5,1)  {$\cdot$};

\node at (5,-4) {\strut$\langle \mathbb{Z}/2\rangle$};
\node at (5,-3) {\strut$\langle \mathbb{Z}/2\rangle$};
\node at (5,-2) {\strut$\langle \mathbb{Z}/2\rangle$};
\node at (5,-1) {\strut$\langle \mathbb{Z}/2\rangle$};

\end{tikzpicture}
\end{center}
\caption{The cohomology $\HHreduced_{C_2}^\alpha(\S{0}\coeffs*{A})$}
\label{figure-HS0-2}
\end{figure}

\begin{figure}
\begin{center}
\begin{tikzpicture}[xscale=1.0,yscale=0.85, every node/.style={fill=white,inner sep=3pt,outer sep=4pt}]
\path[use as bounding box] (-3.9,-3.1) rectangle (3.9,3.1);
\draw[gray!50,<->] (-3,0) -- (3,0);
\draw[gray!50,<->] (0,-3) -- (0,3);

\node at (0,3.5) {$|\alpha|$};
\node at (3.6,0) {$|\alpha^{C_2}|$};

\node[text=gray] at (-2,2) {$\cdot$};
\node[text=gray] at (-2,1) {$\cdot$};

\node[text=gray] at (-2,-2) {$\cdot$};
\node[text=gray] at (-2,-1) {$\cdot$};

\node[text=gray] at (-1,2)  {$\cdot$};
\node[text=gray] at (-1,1)  {$\cdot$};
\node[text=gray] at (-1,-2) {$\cdot$};
\node[text=gray] at (-1,-1) {$\cdot$};

\node at (-2,0) {\strut$\xi$};
\node at (0,0)  {\strut$1$};

\node at (0,1) {\strut$\epsilon$};

\node[text=gray] at (1,2)  {$\cdot$};
\node[text=gray] at (1,1)  {$\cdot$};
\node[text=gray] at (1,-2) {$\cdot$};
\node[text=gray] at (1,-1) {$\cdot$};

\node[text=gray] at (2,2)  {$\cdot$};
\node[text=gray] at (2,1)  {$\cdot$};
\node[text=gray] at (2,-2) {$\cdot$};
\node[text=gray] at (2,-1) {$\cdot$};

\node at (0,-3) {$\HHreduced_{C_2}^\alpha(\S{0})(\GG)$};
\end{tikzpicture}
\begin{tikzpicture}[xscale=1.0,yscale=0.85, every node/.style={fill=white,inner sep=3pt,outer sep=4pt}]
\path[use as bounding box] (-3.9,-3.1) rectangle (3.9,3.1);
\draw[gray!50,<->] (-3,0) -- (3,0);
\draw[gray!50,<->] (0,-3) -- (0,3);

\node at (0,3.5) {$|\alpha|$};
\node at (3.6,0) {$|\alpha^{C_2}|$};

\node[text=gray] at (-2,2) {$\cdot$};
\node[text=gray] at (-2,1) {$\cdot$};

\node[text=gray] at (-2,-2) {$\cdot$};
\node[text=gray] at (-2,-1) {$\cdot$};

\node[text=gray] at (-1,2)  {$\cdot$};
\node[text=gray] at (-1,1)  {$\cdot$};
\node[text=gray] at (-1,-2) {$\cdot$};
\node[text=gray] at (-1,-1) {$\cdot$};

\node at (-1,0) {\strut$\iota^{-1}$};
\node at (0,0)  {\strut$\restriction{\rho}(1)$};
\node at (1,0)  {\strut$\iota$};


\node[text=gray] at (1,2)  {$\cdot$};
\node[text=gray] at (1,1)  {$\cdot$};
\node[text=gray] at (1,-2) {$\cdot$};
\node[text=gray] at (1,-1) {$\cdot$};

\node[text=gray] at (2,2)  {$\cdot$};
\node[text=gray] at (2,1)  {$\cdot$};
\node[text=gray] at (2,-2) {$\cdot$};
\node[text=gray] at (2,-1) {$\cdot$};

\node at (0,-3) {$\HHreduced_{C_2}^\alpha(\S{0})(\Ge)$};
\end{tikzpicture}
\end{center}
\caption{The location of some important elements of $\HHreduced_{C_2}^\alpha(\S{0}\coeffs*{A})$}
\label{figure-HS0-2-important-elements}
\end{figure}

\begin{definition}
The element $1\in \HHreduced_{C_2}^0(\S{0}\coeffs*{A})(\GG)\cong A(\GG)$ is the image of the identity element of $A(\GG)$ under the unit map $A\to \HHreduced_{C_2}^*(\S{0}\coeffs*{A})$ of the graded Green functor $\HHreduced_{C_2}^*(\S{0}\coeffs*{A})$.
\end{definition}

\begin{definition}
\label{def:epsilon-for-2}
The element $\epsilon\in \HHreduced_{C_2}^{\signrep}(\S{0}\coeffs*{A})(\GG)\cong\langle \Z\rangle(\GG)$ is the image of $1\in\HHreduced_{C_2}^0(\S{0}\coeffs*{A})$ under the map
\[A\cong\HHreduced_{C_2}^0(\S{0})\cong\HHreduced_{C_2}^{\signrep}(\S{\signrep})\to\HHreduced_{C_2}^{\signrep}(\S{0})\cong\langle\Z\rangle\] 
induced by the inclusion $\S{0}\hookrightarrow\S{\signrep}$. Observe that the inclusion $\S{0}\hookrightarrow\S{n\signrep}$ therefore represents the element $\epsilon^n\in\HHreduced_{C_2}^{n\signrep}(\S{0}\coeffs*{A})(\GG)$, and more generally $\S{a+b\signrep}\hookrightarrow\S{a+c\signrep}$ represents $\epsilon^{c-b}$.
\end{definition}

\begin{definition}
\label{def:iota-for-2}
Fixing a non-equivariant identification of $\S{\signrep}$ with $\S{1}$ induces isomorphisms
\[\HHreduced_{C_2}^0(\S{0})(\Ge)\cong\HHreduced_{C_2}^1(\S{1})(\Ge)\to\HHreduced_{C_2}^1(\S{\signrep})\cong\HHreduced_{C_2}^{1-\signrep}(\S{0})(\Ge)\]
and
\[\HHreduced_{C_2}^0(\S{0})(\Ge)\cong\HHreduced_{C_2}^{\signrep}(\S{\signrep})(\Ge)\to\HHreduced_{C_2}^\signrep(\S{1})\cong\HHreduced_{C_2}^{\signrep-1}(\S{0})(\Ge)\]
Define $\iota$ and $\iota^{-1}$ to be the images of $\restriction{\rho}(1)\in\HHreduced_{C_2}^0(\S{0})(\Ge)$ under these maps, where $\rho\from \Ge\to\GG$ is the projection map.
\end{definition}

\begin{definition}
\label{def:xi-for-2}
Define $\xi\in\HHreduced_{C_2}^{2\signrep-2}(\S{0})(\GG)$ to be the unique element whose image under $\restriction{\rho}$ is $(\iota^{-1})^2$.
\end{definition}


\subsection{The case of $C_p$ and $C_{pq}$ for odd primes $p,q$}

For the full details of the cohomology $\HHreduced_{C_p}^*(\S{0}\coeffs*{A})$ when $p$ is odd, again see the description in \cite{megan}, or the proof in \cite{lewis_complex}. However, one key observation to make about it, as can be seen in \autoref{figure-HS0-p}, is that the isomorphism class of the Mackey functor $\HHreduced_{C_p}^{\alpha}(\S{0}\coeffs*{A})$ no longer depends completely on the values $(|\alpha^{C_p}|,|\alpha|)$; when $|\alpha|=|\alpha^{C_p}|=0$, the Mackey functor $\HHreduced_{C_p}^{\alpha}(\S{0}\coeffs*{A})$ depends on an integer $d_\alpha$ which is determined in a convoluted way described in \cite[p.65]{lewis_complex}. As Lewis comments there,
\begin{quote}
The major source of unpleasantness in the description of the multiplicative structure of the equivariant cohomology of a point and of complex projective spaces is this lack of a canonical choice for $d$.
\end{quote}
This is the case for quaternionic projective spaces as well.

Basu and Ghosh compute the additive structure of the cohomology $\HHreduced_{C_{pq}}^*(\S{0}\coeffs*{A})$ for distinct odd primes $p$ and $q$ in \cite{basu_ghosh}. They find that a similar problem arises regarding dependence on the specific $\alpha\in\RO(C_{pq})$, although it is also limited in scope to particular $\alpha$ which are analogous the case of $C_p$: in \cite[Theorem 6.5]{basu_ghosh}, we see that

\begin{theorem}
Suppose $\alpha\in\RO(C_{pq})$ is such that at least one of $|\alpha^H|$ or $|\alpha^K|$ is non-zero whenever $(K,H)\in\{(C_p,e),(C_q,e),(C_{pq},C_p),(C_{pq},C_q)\}$. Then, up to isomorphism, the Mackey functor $\HHreduced_{C_{pq}}^\alpha(\S{0})$ depends only on the fixed points $|\alpha|,|\alpha^{C_p}|,|\alpha^{C_q}|,|\alpha^{C_{pq}}|$.
\end{theorem}

\begin{figure}[ht]
\begin{center}
\begin{tikzpicture}[xscale=1.0, yscale=0.85,
every node/.style={fill=white,inner sep=2pt,outer sep=2pt}]
\path[use as bounding box] (-6.9,-5.1) rectangle (6.9,5.9);
\draw[gray!50,<->] (-6,0) -- (6,0);
\draw[gray!50,<->] (0,-5) -- (0,5);

\node at (0,5.5) {$|\alpha|$};
\node at (6.5,0) {$|\alpha^{C_p}|$};

\node[text=gray] at (1,1)  {$\cdot$};
\node[text=gray] at (1,3)  {$\cdot$};
\node[text=gray] at (2,2)  {$\cdot$};
\node[text=gray] at (2,4)  {$\cdot$};
\node[text=gray] at (3,1)  {$\cdot$};
\node[text=gray] at (3,3)  {$\cdot$};
\node[text=gray] at (4,2)  {$\cdot$};
\node[text=gray] at (4,4)  {$\cdot$};
\node[text=gray] at (5,1)  {$\cdot$};
\node[text=gray] at (5,3)  {$\cdot$};

\node[text=gray] at (-1,-1) {$\cdot$};
\node[text=gray] at (-1,-3) {$\cdot$};
\node[text=gray] at (-2,-2) {$\cdot$};
\node[text=gray] at (-2,-4) {$\cdot$};
\node[text=gray] at (-3,-1) {$\cdot$};
\node[text=gray] at (-3,-3) {$\cdot$};
\node[text=gray] at (-4,-2) {$\cdot$};
\node[text=gray] at (-4,-4) {$\cdot$};
\node[text=gray] at (-5,-1) {$\cdot$};
\node[text=gray] at (-5,-3) {$\cdot$};

\node[text=gray] at (-1,1) {$\cdot$};
\node[text=gray] at (-1,3) {$\cdot$};
\node[text=gray] at (-3,1) {$\cdot$};
\node[text=gray] at (-3,3) {$\cdot$};
\node[text=gray] at (-5,1) {$\cdot$};
\node[text=gray] at (-5,3) {$\cdot$};

\node[text=gray] at (1,-1) {$\cdot$};
\node[text=gray] at (1,-3) {$\cdot$};
\node[text=gray] at (2,-2) {$\cdot$};
\node[text=gray] at (2,-4) {$\cdot$};
\node[text=gray] at (4,-2) {$\cdot$};
\node[text=gray] at (4,-4) {$\cdot$};

\node at (-4,4) {\strut$\langle \mathbb{Z}/p\rangle$};
\node at (-4,2) {\strut$\langle \mathbb{Z}/p\rangle$};

\node at (-2,4) {\strut$\langle \mathbb{Z}/p\rangle$};
\node at (-2,2) {\strut$\langle \mathbb{Z}/p\rangle$};

\node at (5,-3) {\strut$\langle \mathbb{Z}/p\rangle$};
\node at (5,-1) {\strut$\langle \mathbb{Z}/p\rangle$};

\node at (3,-3) {\strut$\langle \mathbb{Z}/p\rangle$};
\node at (3,-1) {\strut$\langle \mathbb{Z}/p\rangle$};

\node at (-4,0) {\strut$R$};
\node at (-2,0) {\strut$R$};
\node at (0,0)  {\strut$A[d_\alpha]$};
\node at (2,0)  {\strut$L$};
\node at (4,0)  {\strut$L$};

\node at (0,4) {\strut$\langle \mathbb{Z}\rangle$};
\node at (0,2) {\strut$\langle \mathbb{Z}\rangle$};
\node at (0,-4) {\strut$\langle \mathbb{Z}\rangle$};
\node at (0,-2) {\strut$\langle \mathbb{Z}\rangle$};

\end{tikzpicture}
\end{center}
\caption{The cohomology $\HHreduced_{C_p}^\alpha(\S{0}\coeffs*{A})$ for $p$ odd}
\label{figure-HS0-p}
\end{figure}


\section{Even-dimensional freeness}


In this section we will look at a result which generalizes the standard result in the non-equivariant setting that the cohomology of a cell complex built only out of even-dimensional cells must be free.

\begin{definition}
\label{def:even-rep}
An element $\alpha\in\RO(G)$ is said to be even if $|\alpha^H|$ is even for all subgroups $H\leq G$.
\end{definition}

\begin{definition}
\label{def:even-cell}
Let $V$ be a representation of $H\leq G$. A cell $G\times_{H} \unitdisk{V}$ of a $G$-cell complex is called even if $V$ is even in $\RO(H)$.
\end{definition}

\begin{definition}
Given representations $W$ and $V$ of $G$, we say that $W\ll V$ if, whenever $|W^S|<|V^S|$ for a subgroup $S\leq G$, we also have $|W^T|\leq|V^T|$ for all subgroups $S\leq T\leq G$.
\end{definition}

\begin{definition}
\label{def:even-monotone}
A $G$-cell complex structure on a space $X$ is said to be \defstyle{even}, or of \defstyle{even type}, when
\begin{itemize}
\item[(a)] every cell is even (as in \autoref{def:even-cell}).
\end{itemize}
We will further say it is \defstyle{properly even} when it also satisfies
\begin{itemize}
\item[(b)] if the cell $G\times_H \unitdisk{W}$ is attached before the cell $G\times_H \unitdisk{V}$, then $W\ll V$, and
\item[(c)] for any positive integer $N$, there are only finitely many cells $G\times_{H} \unitdisk{V}$ where $|V^K|\leq N$ for all subgroups $K\leq G$, and
\item[(d)] each piece $X_n$ of the filtration is a finite $G$-cell complex.
\end{itemize}
\end{definition}

\begin{theorem}
\label{thm:even-dim-free}
Let $G=C_p$ for any prime $p$, or $C_{pq}$ for any odd primes $p$ and $q$.
If $X$ is a properly even $G$-cell complex, then the cohomology $\HHreduced_{G}(X\coeffs*{A})$ is free as a module over the Green functor $\HHreduced_{G}(\S{0}\coeffs*{A})$. Furthermore, $\HHreduced_{G}(X\coeffs*{A})$ decomposes as a direct sum, with summands
\begin{itemize}
\item one copy of $\HHreduced_{G}^{*}({X_0}_{+}\coeffs*{A})$, and 
\item one copy of $\Sigma^V\HHreduced_{G}^{*}(G/H_{+}\coeffs*{A})$ for each cell of the form $G\times_{H}\unitdisk{V}$.
\end{itemize}
\end{theorem}

This theorem was proven for $G=C_p$ by Lewis in \cite{lewis_complex}, with his proof corrected by Shulman in \cite{megan}, and this theorem was proven for $G=C_{pq}$ by Basu and Ghosh in \cite{basu_ghosh}.

The basic idea in each proof is the same as in the non-equivariant case: to show that the boundary maps are all zero in the long exact sequence in cohomology arising from the cofiber sequences
\[X_{n}\to X_{n+1}\to X_{n+1}/X_n.\]
One of the corrections that Shulman made to Lewis' proof was to introduce condition (c) from the definition of properly even, which is necessary to ensure that $\HHreduced_{G}^{*}(X_{+}\coeffs*{A})$, which  \textit{a priori} would be an infinite product, does in fact agree with the corresponding coproduct and therefore decomposes as a direct sum.

\begin{remark}
Ferland, in his thesis \cite{ferland}, established a freeness result for the cohomology of $G$-cell complexes $X$ which are even, but not necessarily properly even. In this situation, we have to consider non-trivial boundary maps in the long exact sequence in cohomology, which curiously do not change the fact that the cohomology is free, but which do change the dimensions in which generators occur, so that the they no longer necessarily live in the dimensions corresponding to the representation cells that comprise $X$. 
\end{remark}

\section{\texorpdfstring{Multiplicative comparison theorem for $C_p$}{Multiplicative comparison theorem for C\_p}}

\subsection{Statement of the theorem}

The following surprising result first appeared in \cite{lewis_complex} but with a flawed proof; it was subsequently corrected by Shulman in \cite[Theorem 3.10, p.55]{megan}. We have termed it a ``multiplicative comparison'' theorem because it allows us to compare the multiplicative structure of $\HHreduced_{C_p}^*(X\coeffs*{A})$ with that of other objects which are more easily understood.

\begin{theorem}
\label{thm:mult-comparison}
Let $X$ be a properly even $C_p$-cell complex. Let $i\from X^{C_p}\to X$ be the inclusion of its fixed points. Let {\upshape$\rho\from\Ge\to\GG$} be the projection, which induces a map {\upshape$\HHreduced_{C_p}^*(X_{+})\to\HHreduced_{C_p}^*(X_{+}\wedge {C_p}_{+})\cong \HHreduced_{C_p}^*(X_{+})_{\Ge}$}. If $\alpha\in \RO(C_p)$ is even, the map {\upshape
\[\rho\oplus i^*:\HH_{C_p}^\alpha(X\coeffs*{A})\to\HH_{C_p}^\alpha(X\coeffs*{A})_{\Ge}\oplus \HH_{C_p}^\alpha(X^{C_p}\coeffs*{A})\]}
is injective.
\end{theorem}

\begin{remark}
Observe that in Lewis's original formulation of this result, \cite[Corollary 2.7, p.70]{lewis_complex}, the map $\rho$ used here is instead thought of as the map $M(\restriction{\rho})\from M(\GG)\to M(\Ge)$, where $M=\HHreduced_{C_p}^*(X\coeffs*{A})$. This is ultimately equivalent to Shulman's version, as described in \cite[p.54]{megan}.
\end{remark}

\subsection{Relationships with non-equivariant cohomology}

We have seen in \autoref{thm:mult-comparison} that, for certain $C_p$-cell complexes $X$ with ``even-dimensional'' cells, the graded Green functor $\HHreduced_{C_p}^*(X\coeffs*{A})$ can be embedded in the direct sum of
\begin{itemize}
\item $\HHreduced_{C_p}^*(X\coeffs*{A})_{\Ge}$, the same Green functor but shifted by the $C_p$-set $\Ge$ as in \autoref{def:shifted-mackey}, and 
\item $\HHreduced_{C_p}^*(X^{C_p}\coeffs*{A})$, the cohomology of the fixed points of $X$.
\end{itemize}
Crucially, both of these pieces can be understood in terms of non-equivariant cohomology, which makes \autoref{thm:mult-comparison} a powerful tool in computing the $\RO(C_p)$-graded cohomology of such spaces.

\begin{remark}
\label{nonequivariant}
In order to understand the first term $\HHreduced_{C_p}^*(X\coeffs*{A})_{\Ge}$, the shifted cohomology of $X$, we first make a simple observation about shifting $C_p$-Mackey functors. Recall the notation $\GG$ and $\Ge$ from \autoref{rem:dot-and-sun}. Then for any $C_p$-Mackey functor $M$, we have by \autoref{def:shifted-mackey}  that
\[M_{\Ge}(\GG)=M(\Ge),\qquad M_{\Ge}(\Ge)=M(\Ge\times\Ge) \cong\oplus_{g\in C_p}M(\Ge).\]
Thus, $\HHreduced_{C_p}^*(X\coeffs*{A})_{\Ge}$ depends only on $\HHreduced_{C_p}^*(X\coeffs*{A})(\Ge)$.

Now, by \cite[Example 1.1g, p.60]{lewis_complex}, for any $C_p$-space $X$ and any $\alpha\in\RO(C_p)$, 
we have \[\HHreduced_G^\alpha(X\coeffs*{A})(\Ge)\cong\widetilde{H}^{|\alpha|}(X\coeffs*{\Z}),\] which establishes an elegant relationship with non-equivariant cohomology. Therefore, we identify $\HHreduced_G^*(X\coeffs*{A})(\Ge)$ with $\underline{H}^{|*|}(X\coeffs*{\Z})$, i.e., the $\RO(G)$-graded abelian group whose value at $\alpha\in\RO(G)$ is the group $H^{|\alpha|}(X\coeffs*{Z})$. As Shulman points out in \cite[p.30]{megan},
\begin{quote}
The identification of $\HHreduced_{C_p}^\alpha(X;M)(\Ge)$ with the nonequivariant cohomology group shown uses the adjunction $[{C_p}_{+}\wedge X,HM]_{C_p}\cong [X,HM]$ together with the fact that the underlying nonequivariant spectrum of $HM$ is a $H(M(\Ge))$. The $C_p$-action comes from the action on the $k[C_p]$-module $M(\Ge)$.
\end{quote}
\end{remark}

The second term in \autoref{thm:mult-comparison} is $\HHreduced_{C_p}^*(X^{C_p}\coeffs*{A})$, which makes it valuable to understand what the $\RO(C_p)$-graded cohomology of the space $X$ with ``even-dimensional'' cells looks like when $X$ has a trivial $C_p$-action. The following result from \cite{lewis_complex} shows that it is just the $\HHreduced_{C_p}^*(\S{0}\coeffs*{A})$-algebra generated by the non-equivariant cohomology of $X$.

\begin{theorem}
\label{thm:trivial-even-cohomology}
If a CW complex $X$ with cells only in even dimensions is regarded as a $C_p$-space with trivial $C_p$-action, then there is an isomorphism of $\RO(C_p)$-graded Mackey functors
\[\HHreduced_{C_p}^*(X\coeffs*{A})\cong \HHreduced_{C_p}^*(\S{0}\coeffs*{A})\otimes \Hhreduced^*(X\coeffs*{\Z})\]
which preserves cup products (here, $\otimes$ is in the sense of \autoref{def:tensoring-mackey-with-group}).
\end{theorem}

\begin{remark}
\label{rem:mult-comparison-useful}
We will see in \autoref{thm:fixed-points-c} and \autoref{thm:fixed-points-h} that when $X$ is the projective space of either a complex or a quaternionic $C_p$-representation, the fixed points $X^{C_p}$ are a disjoint union of complex or quaternionic projective spaces, whose non-equivariant cohomology is known classically, so that via \autoref{thm:trivial-even-cohomology}, the codomain of $i^*$ is a direct sum of easily-understood $\HHreduced_{C_p}^*(\S{0}\coeffs*{A})$-algebras. We will therefore find it useful in \autoref{chap:mult} to break up $i^*$ into components $i_k^*$, one for each in the direct sum. Observe that in this case, to specify an element of $\HHreduced_{C_p}^\alpha(X\coeffs*{A})$ uniquely, it suffices to describe its image under $\rho$ and each $i_k^*$, though one must first verify that an element exists with those images.
\end{remark}



\chapter{Representations and projective spaces}

\section{Generalities about representations}

\begin{definition}
Let $G$ be a finite group, and let $R$ be a ring. Define a $G$-representation over $R$ to be a left $R$-module $V$ together with an action $\rho:G\times V\to V$ such that the action of each element $p_g:V\to V$ is left $R$-linear. By abuse of notation, we will refer to the representation as $V$.
\end{definition}

\begin{remark}
\label{rep-notation-terminology}
In this work the only rings $R$ we will use will be $\R$, $\C$, or $\H$, in which case we refer to a representation over $R$ as real, complex, or quaternionic, respectively.
\end{remark}

Much of classical representation theory over $\R$ and $\C$ carries over to $\H$. For example, if $G$ is compact and $W$ is a quaternionic $G$-representation, then $W$ has a $G$-invariant inner product. As in the real and complex cases, this implies $W$ decomposes as a direct sum of irreducibles. See \cite{gtm98}.

\begin{definition}
For any representation $V$, no matter whether it is introduced as a real, complex, or quaternionic representation, the notation $|V|$ will always be defined to mean the dimension $|V|_{\R}$ of $V$ considered as a real vector space.
\end{definition}

\begin{definition}
Given a complex $G$-representation $V$, define $\repextension{V}$ to be the \defstyle{extension} of $V$, which is the quaternionic $G$-representation $\H\otimes_\C V$, where $\H$ is seen as a right $\C$-module via right multiplication:
\[q\otimes zv=qz\otimes v\qquad q\in\H,\;z\in\C,\;v\in V\]
Given a quaternionic $G$-representation $W$, define $\represtriction{W}$ to be the \defstyle{restriction} of $W$, which is the quaternionic $G$-representation $W$ viewed as a complex $G$-representation. We will often write $W$ for $\represtriction{W}$, only using the latter for emphasis.
\end{definition}

\begin{definition}
We say that a complex $G$-representation $V$ has
\begin{itemize}
\item \defstyle{real type} if $V$ admits a conjugate-linear $G$-map $J:V\to V$ with $J^2=\mathrm{id}$,
\item \defstyle{complex type} if $V\not\cong\overline{V}$, and
\item \defstyle{quaternionic type} if $V$ admits a conjugate-linear $G$-map $J:V\to V$ with $J^2=-\mathrm{id}$.
\end{itemize}
\end{definition}

\begin{definition}
For $K\in\{\R,\C,\H\}$, define $\mathrm{Irr}(G,K)$ to be the set of isomorphism classes of irreducible $G$-representations over $K$. We also specify these subsets of $\Irr(G,\C)$:
\begin{align*}
\Irr(G,\C)_\R&=\{V\in\Irr(G,\C):V\text{ has real type}\}\\
\Irr(G,\C)_\C&=\{V\in\Irr(G,\C):V\text{ has complex type}\}\\
\Irr(G,\C)_\H&=\{V\in\Irr(G,\C):V\text{ has quaternionic type}\}
\end{align*}
as well as these subsets of $\Irr(G,\H)$:
\begin{align*}
\Irr(G,\H)_\R&=\{\repextension{V}:V\in\Irr(G,\C)_\R\}\\ \Irr(G,\H)_\C&=\{\repextension{V}:V\in\Irr(G,\C)_\C\}\\ \Irr(G,\H)_\H&=\{W:\represtriction{W}\in\Irr(G,\C)_\H\}
\end{align*}
\end{definition}
Observe that, by definition, $V\in\Irr(G,\C)_\C$ if and only if $V\not\cong\overline{V}$, so the elements of $\Irr(G,\C)_\C$ can be paired off, each with its conjugate.
\begin{proposition}
\label{prop:reps-have-unique-type}
Let $\frac{1}{2}\Irr(G,\C)_\C$ be a subset of $\Irr(G,\C)_\C$ containing exactly one element from each conjugate pair. If $\Irr(G,\C)_\H$ is empty, then extension $\repextension{}$ is a bijection \[\Irr(G,\C)_\R\sqcup\tfrac{1}{2}\Irr(G,\C)_\C\xrightarrow{\;\;\repextension{}\;\;}\Irr(G,\H).\]
\end{proposition}
\begin{proof}
By {\cite[Theorem II.6.3, p.97]{gtm98}}, any $V\in\Irr(G,\C)$ or  $W\in\Irr(G,\H)$ has exactly one type, so that
\[\Irr(G,\C)=\Irr(G,\C)_\R \sqcup \Irr(G,\C)_\C \sqcup \Irr(G,\C)_\H\]
and
\[\Irr(G,\H)=\Irr(G,\H)_\R \sqcup \Irr(G,\H)_\C \sqcup \Irr(G,\H)_\H.\]
By \cite[Exercise II.6.10.7, p.100]{gtm98}, extension gives bijections
\[\Irr(G,\C)_\R\xrightarrow{\;\;\repextension{}\;\;}\Irr(G,\H)_\R, \qquad\qquad\tfrac{1}{2}\Irr(G,\C)_\C\xrightarrow{\;\;\repextension{}\;\;}\Irr(G,\H)_\C.\]
Finally, the set $\Irr(G,\H)_\H$ must be empty because any element of $\Irr(G,\H)_\H$ by definition would have to restrict to an element of $\Irr(G,\C)_\H$, which we have assumed to be empty.
\end{proof}

This result lets us more easily identify the type of an irreducible complex $G$-representation.

\begin{proposition}[{\cite[Theorem II.6.8, p.100]{gtm98}}]
\label{prop:rep-type-integral}
For any $V\in\Irr(G,\C)$ we have
\[\frac{1}{|G|}\sum_{g\in G}\chi_V(g^2)=\begin{cases}
\hphantom{-{}}1 &\iff V\text{ has real type},\\
\hphantom{-{}}0 &\iff V\text{ has complex type},\\
-1 &\iff V\text{ has quaternionic type}.
\end{cases}\]
\end{proposition}

\section{\texorpdfstring{$C_n$-representations}{C\_n-representations}}

\subsection{Classification of irreducibles}

Throughout this section, we will let $\mathbf{1}_n\in C_n$ denote the usual generator of $C_n$, and let $\zeta_n=e^{2\pi i/n}$.

\begin{definition}
\label{def:irred-C}
For any integer $r$, define the complex $C_n$-representation ${}_n\irredC_r$ (or $\irredC_r$, if the choice of $n$ is understood) to be $\C$ with the action $\rho\from C_n\times\C\to\C$ defined by $\rho(\mathbf{1}_n,z)=z\cdot \zeta_n^r$.
\end{definition}

\begin{remark}
If $r\equiv s\bmod n$ we have $\zeta_n^r=\zeta_n^s$, and hence $\irredC_r= \irredC_s$. Nevertheless, it will be convenient to allow the subscript to be any integer. Also, observe that for any $r$, the conjugate representation $\overline{\irredC_r}$ is $\C$ with $\mathbf{1}_n$ acting as $\overline{\zeta_n^r}=\zeta_n^{-r}$, which is just $\irredC_{-r}$.
\end{remark}

The following classification is a standard result.

\begin{theorem}
\label{thm:class-of-irred-C}
The irreducible complex $C_n$-representations are, up to isomorphism,  \[\irredC_0,\ldots,\irredC_{n-1}.\]
Moreover,
\[\irredC_r\cong \irredC_s\iff  r\equiv s\bmod n.\]
\end{theorem}

Now we will want a classification of irreducible quaternionic $C_n$-representations.

\begin{definition}
\label{def:irred-H}
For any integer $r$, define the quaternionic $C_n$-representation ${}_n\irredH_r$ (or $\irredH_r$, if the choice of $n$ is understood) to be $\H$ with the action $\rho\from C_n\times\H\to\H$ defined by $\rho(\mathbf{1}_n,q)=q\cdot \zeta_n^k$.
\end{definition}

\begin{remark}
Note that each map $\rho(g,\,\cdot\,)\from\H\to\H$ is left $\H$-linear, as required by the definition. Also note that $\irredH_r$ is the extension $\repextension{\irredC_r}$.
\end{remark}

\begin{lemma}
\label{lem:irredH-as-irredC}
There is an isomorphism of complex $C_n$-representations $\irredH_r\cong \irredC_r\oplus \irredC_{-r}$.
\end{lemma}
\begin{proof}
The complex subspaces of $\irredH_r$
\[\{a+bi+cj+ck\in \irredH_r:c=d=0\},\qquad \{a+bi+cj+ck\in \irredH_r:a=b=0\}\]
are complements of each other in $\irredH_r$, they are $C_n$-invariant, and they are isomorphic to $\irredC_r$ and $\irredC_{-r}$, respectively.
\end{proof}

\begin{theorem}
\label{thm:class-of-irred-H}
The irreducible quaternionic $C_n$-representations are, up to isomorphism,
\[\irredH_0,\irredH_1,\ldots,\irredH_{\lfloor n/2\rfloor}.\]
Moreover,
\[\irredH_r\cong \irredH_s\iff  r\equiv \pm s\bmod n.\]
\end{theorem}
\begin{proof}
The character of the representation $\irredC_r$ is $\chi_{\irredC_r}(k)=\zeta_n^{rk}$, and
\[\frac{1}{|C_n|}\sum_{k=0}^{n-1}\chi_{\irredC_r}(2k)=\frac{1}{n}\sum_{k=0}^{n-1}(\zeta_n^{2r})^k=\begin{cases}
1& \text{ if }2r\equiv 0 \bmod n,\\
0&\text{ otherwise.}
\end{cases}\]
By \autoref{prop:rep-type-integral}, this implies that $\Irr(C_n,\C)_\H$ is always empty, that
\[\Irr(C_n,\C)_\R=\begin{cases}
\irredC_0&\text{ if }n\text{ is odd},\\
\irredC_0,\irredC_{\frac{n}{2}}&\text{ if }n\text{ is even},
\end{cases}\]
and that $\Irr(C_n,\C)_\C$ contains the rest, which can be grouped into conjugate pairs as $(\irredC_r,\irredC_{n-r})$.

We can choose a subset $\frac{1}{2}\Irr(G,\C)_\C$ containing exactly one element from each pair, 
\[\tfrac{1}{2}\Irr(C_n,\C)_\C=\begin{cases}
\irredC_1,\irredC_2,\ldots,\irredC_{\lfloor\frac{n}{2}\rfloor}&\text{ if }n\text{ is odd},\\
\irredC_1,\irredC_2,\ldots,\irredC_{\frac{n}{2}-1}&\text{ if }n\text{ is even.}
\end{cases}\]
Because $\irredH_r$ is the extension $\repextension{\irredC_r}$, by \autoref{prop:reps-have-unique-type} we conclude that regardless of the parity of $n$, \[\Irr(C_n,\H)=\{\irredH_0,\irredH_1,\ldots,\irredH_{\lfloor\frac{n}{2}\rfloor}\}.\]

If $\irredH_r\cong \irredH_s$, then by \autoref{lem:irredH-as-irredC}, we have an isomorphism of complex $C_n$-representations \[\irredC_r\oplus \irredC_{-r}\cong \irredC_s\oplus \irredC_{-s}.\] The decomposition of a complex $C_n$-representation into irreducibles is unique (\autoref{prop:isotypical-decomposition}) and $\irredC_r\cong \irredC_s\iff r\equiv s\bmod n$ (\autoref{thm:class-of-irred-C}), so we must have $r\equiv \pm s\bmod n$. For the other direction, it is clear that $\irredH_r\cong \irredH_s$ when $r\equiv s\bmod n$, so assume that $r\equiv -s\bmod n$. Then there is a left $\H$-linear $C_n$-equivariant isomorphism $f\from\irredH_r\to \irredH_s$ defined by $f(q)=qj$, so that $\irredH_r\cong \irredH_s$.
\end{proof}

\subsection{Isotypical components}

\begin{definition}
\label{def:isotypical}
For a complex $C_n$-representation $V$ and integer $r$, define $V(r;\C)$ to be the isotypical component of $V$ associated to the irreducible complex representation $\irredC_r$:
\[V(r;\C)=\bigoplus_{\substack{U\subseteq V\\ U\cong \irredC_r}}U.\]
Similarly, for a quaternionic $C_n$-representation $W$ and integer $r$, define $W(r;\H)$ to be the isotypical component of $W$ associated to the irreducible quaternionic representation $\irredH_r$:
\[W(r;\H)=\bigoplus_{\substack{U\subseteq W\\ U\cong \irredH_r}}U.\]
\end{definition}

\begin{proposition}
\label{prop:isotypical-decomposition}
Any complex or quaternionic $C_n$-representation decomposes as an internal direct sum into its isotypical components:
\[V=\bigoplus_{r=0}^{n-1}V(r;\C),\qquad\qquad W=\bigoplus_{r=0}^{\lfloor n/2\rfloor}W(r;\H).\]
\end{proposition}
\begin{proof}
The group rings $\C[C_n]$ and $\H[C_n]$ are semisimple.
\end{proof}

\begin{proposition}
\label{prop:H-isotypical-as-C-isotypical}
For any quaternionic $C_n$-representation $W$, we have
\[W(r;\H)=W(-r;\H)=\begin{cases}
W(r;\C) & \text{if }2r\equiv 0\bmod n,\\
W(r;\C)\oplus W(-r;\C) & \text{otherwise}
\end{cases}\]
and therefore
\[|W(r;\C)|=\begin{cases} \hphantom{\frac{1}{2}}|W(r;\H)| & \text{if }2r\equiv 0\bmod n,\\ \frac{1}{2}|W(r;\H)| & \text{otherwise}. \end{cases}\]
\end{proposition}
\begin{proof}
This can be seen by combining \autoref{lem:irredH-as-irredC}, the fact that $\irredH_r\cong\irredH_s$ if and only if $r\equiv\pm s\bmod n$ from \autoref{thm:class-of-irred-H}, and the fact that $r\equiv -r\bmod n$ if and only if $2r\equiv 0\bmod n$.
\end{proof}

\subsection{Fixed subspaces and restricted representations}
\label{subsec:fixed-restricted}

For any $G$-representation $V$ and any subgroup $H\leq G$, we can form the restricted representation $V|_{H}$. The following proposition describes how complex and quaternionic irreducible $C_n$-representations behave under restriction to a subgroup $dC_n\leq C_n$. Recall the notation ${}_n\irredC_r$ and ${}_n\irredH_r$ from \autoref{def:irred-C} and \autoref{def:irred-H}, and observe that there is a natural isomorphism $dC_n\cong C_{n/d}$ by sending $d\cdot\mathbf{1}_n$ to $\mathbf{1}_{n/d}$, so that a $dC_n$-representation can be treated as a $C_{n/d}$-representation.

\begin{proposition}
\label{prop:Cnd-restricted}
For any integers $n$ and $r$, and any $d\mid n$, we have that
\[{}_n\irredC_r|_{dC_n}\cong {}_{n/d}\irredC_r,\qquad {}_n\irredH_r|_{dC_n}\cong {}_{n/d}\irredH_r.\]
\end{proposition}
\begin{proof}
In the representation ${}_n\irredC_r$, the element $\mathbf{1}_{n}\in C_n$ acts as $\zeta_n^r$. Therefore the element $d\cdot\mathbf{1}_n\in C_n$ acts as $(\zeta_n^r)^d=\zeta_{n/d}^r$. When we restrict ${}_n\irredC_r$ to the representation ${}_n\irredC_r|_{dC_n}$ and consider it as a $C_{n/d}$-representation, the element $\mathbf{1}_{n/d}\in C_{n/d}$ acts the way $d\cdot\mathbf{1}_n\in C_n$ does, so that ${}_n\irredC_r|_{dC_n}\cong{}_{n/d}\irredC_r$. 
(The argument is identical in the quaternionic case.)
\end{proof}

Now we can extend this subscript notation more generally. If $V$ is any $C_n$-representation, we will write ${}_nV$ to emphasize that it is a $C_n$-representation, and by ${}_{n/d}V$ we will mean $V|_{dC_n}$ considered as a $C_{n/d}$-representation by restriction to $dC_n\leq C_n$.

\begin{proposition}
\label{prop:dCn-fixed}
For any complex $C_n$-representation $V$, and any $d\mid n$, we have
\[({}_nV)^{dC_n}=({}_{n/d}V)^{C_{n/d}}.\]
\end{proposition}
\begin{proof}
This follows directly from the definition, because the action of $C_{n/d}$ on ${}_{n/d}V=V|_{dC_n}$ is precisely the action of $dC_n$ on ${}_nV$.
\end{proof}

\subsection{Tensoring with a given irreducible}

\begin{proposition}
\label{prop:shifted-isotypical-components}
For any complex $C_n$-representation $V$, and any integers $r$ and $k$,
\[(\irredC_r\otimes_{\C}V)(k;\C)\cong \irredC_r\otimes_{\C}V(k-r;\C).\]
\end{proposition}
\begin{proof}
Tensor products distribute over direct sums, $V$ is a direct sum of irreducibles $\irredC_s$, and
\[\irredC_r\otimes_{\C}\irredC_s\cong \irredC_k\iff s\equiv k-r\bmod n.\qedhere\]
\end{proof}

\begin{theorem}
\label{thm:trivial-sphere}
For any complex $C_n$-representation $V$, and any integer $r$,
\[|(\irredC_{r}\otimes_{\C} V)^{C_n}|=|V(-r;\C)|\]
and $\irredC_r\otimes_{\C} V(-r;\C)$ is a trivial $C_n$-representation of dimension $|V(-r;\C)|$, so that 
\[\S{\irredC_r\otimes_{\C} V(-r;\C)}=\S{|V(-r;\C)|}.\]
\end{theorem}
\begin{proof}
Using \autoref{prop:shifted-isotypical-components} with $k=0$, we conclude that
\[(\irredC_r\otimes V)^{C_n}=(\irredC_r\otimes V)(0;\C)\cong \irredC_r\otimes_{\C} V(-r;\C)\]
is fixed by $C_n$, i.e., it is a trivial representation. Moreover,
\[|\irredC_r\otimes_{\C} V(-r;\C)|=|V(-r;\C)|\]
because tensoring with $\irredC_{-r}$ over $\C$ doesn't change dimension. Therefore we can use the fact (rather, the good notation) that $\S{V}$ is just the sphere $\S{|V|}$ for a trivial $C_n$-representation $V$, so that 
\[\S{\irredC_r\otimes_{\C} V(-r;\C)}=\S{|V(-r;\C)|}.\qedhere\]
\end{proof}

%

\begin{remark}
In the proof of \autoref{thm:PH-even-monotone}, we will apply \autoref{thm:trivial-sphere} to a quaternionic $C_n$-representation $W$, while keeping in mind \autoref{prop:H-isotypical-as-C-isotypical}.
\end{remark}

\section{\texorpdfstring{Projective space of a $C_n$-representation}{Projective space of a C\_n-representation}}

\subsection{Definitions}

\begin{definition}
\label{def:projective-space}
For a complex vector space $V$, we define the complex projective space of $V$ to be the space of one-dimensional complex subspaces of $V$, i.e.,
\[\PC(V)=(V\setminus\{0\})/{\sim}\]
where $v\sim \lambda v$ for any $\lambda\in \C^\times$, and we write the equivalence class of $v$ in $\PC(V)$ as $\langle v\rangle$.

Similarly, for a left $\H$-module $W$, we define the projective space of $W$ to be the space of one-dimensional left $\H$-subspaces of $W$, i.e.,
\[\PH(W)=(W\setminus\{0\})/{\sim}\]
where $w\sim \lambda w$ for any $\lambda\in \H^\times$, and we write the equivalence class of $w$ in $\PH(W)$ as $\langle w\rangle$.
\end{definition}

\begin{remark}
We use the usual Euclidean topology on $\H$. If $W$ is a finite dimensional left $\H$-module, then it is isomorphic to $\H^n$ for some $n$, and we give $W$ that topology; if $W$ is countably infinite dimensional, we topologize it as the colimit of its finite dimensional sub-$\H$-modules.

Naturally, $\PH(W)$ has the quotient topology. Note that if $U$ is a sub-$\H$-module of $W$, there is an inclusion $\PH(U)\subset\PH(W)$. If $W$ is countably infinite dimensional, then the quotient topology of $\PH(W)$ is the same as the the topology it gets as the colimit of the subspaces $\PH(U)$ as $U$ ranges over the finite dimensional sub-$\H$-modules. See \cite[p.71]{lewis_complex}.
\end{remark}

\begin{remark}
\label{rem:induced-action-projective}
If $V$ is a complex $G$-representation, there is a natural $G$-action on $\PC(V)$ defined by $g\langle v\rangle=\langle gv\rangle$, which is well-defined since the action of $G$ is via linear maps. Similarly, there is a natural $G$-action on $\PH(W)$ where $W$ is a quaternionic $G$-representation.
\end{remark}

\begin{definition}
\label{def:fixed-point-pieces}
For any complex $C_n$-representation $V$ and quaternionic $C_n$-representation $W$, define the spaces
\[\PC[r](V)=\PC(V(r;\C)). \qquad\qquad\PH[r](W)=\begin{cases}
\PH(W(r;\H)) & \text{if }2r\equiv 0\bmod n,\\
\PC(W(r;\C)) & \text{otherwise}.
\end{cases}\]
These spaces have a \emph{trivial} $C_n$-action, i.e., $\PC[r](V)\subseteq \PC(V)^{C_n}$ and $\PH[r](W)\subseteq\PH(W)^{C_n}$, because for any given point in these spaces, the action of $C_n$ scales every non-zero coordinate by the same factor (namely, $\zeta_n^r$) which has no effect in projective space. See \autoref{thm:fixed-points-c} and \autoref{thm:fixed-points-h} for the full details.
\end{definition}

\begin{remark}
When $2r\equiv 0\bmod n$, we have that $W(r;\H)=W(r;\C)$ by \autoref{prop:H-isotypical-as-C-isotypical}, so that we could also say that $\PH[r](W)=\PH(W(r;\C))$ when $2r\equiv 0\bmod n$, but the notation $W(r;\H)$ emphasizes that it is something one can take the quaternionic projective space of.

When instead $2r\not\equiv0\bmod n$, the notation $\PH[r](W)$ is more than a little misleading because this is in fact a \emph{complex} projective space.
\end{remark}

\subsection{Cofiber sequences}

\begin{theorem}
\label{thm:cell-structure-c}
For any complex $C_n$-representation $V$ and any $k$, there is a cofiber sequence
\[\PC(V)_{+}\to\PC(V\oplus \irredC_k)_{+}\to \S{\irredC_{-k}\otimes_{\C} V}.\]
Thus to form $\PC(V\oplus \irredC_k)$ from $\PC(V)$, we attach a cell of the form $\unitdisk{\irredC_{-k}\otimes_{\C} V}$.
\end{theorem}
\begin{proof}
Write an element of $\PC(V\oplus \irredC_k)$ as $\langle v, x\rangle$, where $v\in V$ and $x\in \irredC_r$ are not both $0$, and $\langle \lambda v,\lambda x\rangle=\langle v,x\rangle$ for any $\lambda\in\C^\times$. Then the complement of the image of $\PC(V)_{+}\to\PC(V\oplus \irredC_k)_{+}$ is $\{\langle v,1\rangle:v\in V\}$, which is clearly non-equivariantly homeomorphic to $V$, and hence also to the open unit disk of $V$. But the $C_n$-action on this space is different: if $\mathbf{1}_n$ is our generator of $C_n$, then
\[\mathbf{1}_n\cdot \langle v,1\rangle = \langle \mathbf{1}_n\cdot v,\mathbf{1}_n\cdot 1\rangle =\langle \mathbf{1}_n\cdot v,\zeta_n^k\rangle=\langle \zeta_n^{-k}(\mathbf{1}_n\cdot v),1\rangle\]
Thus, this is $C_n$-homeomorphic to the open unit disc of $\irredC_{-k}\otimes_{\C} V$. In other words, the cell we attached to go from $\PC(V)$ to $\PC(V\oplus \irredC_k)$ is $\unitdisk{\irredC_{-k}\otimes_{\C} V}$.
\end{proof}
\begin{remark}
In the case $V=0$, we have $\PC(V)=\varnothing$ and $\PC(V\oplus \irredC_k)=\pt$, so the disjoint basepoints are necessary to make the cofiber $\S{0}$ as desired.
\end{remark}

\begin{theorem}
\label{thm:cell-structure-h}
For any quaternionic $C_n$-representation $W$ and any $k$, there is a cofiber sequence
\[\PH(W)_{+}\to\PH(W\oplus \irredH_k)_{+}\to \S{\irredC_{-k}\otimes_{\C} W}.\]
Thus to form $\PH(W\oplus \irredH_k)$ from $\PH(W)$, we attach a cell of the form $\unitdisk{\irredC_{-k}\otimes_{\C} W}$.
\end{theorem}
\begin{proof}
The argument is essentially identical as for \autoref{thm:cell-structure-c}. Note that a quaternionic $C_n$-representation $W$ is a complex $C_n$-representation by restriction, so there is no problem forming the tensor product $\irredC_{-k}\otimes_{\C} W$, which is again a complex $C_n$-representation.
\end{proof}

\begin{theorem}
\label{thm:fixed-cell-structure-c}
For any complex $C_n$-representation $V$ and any $k$ and $r$, there is a cofiber sequence
\[\PC[r](V)_{+}\hookrightarrow\PC[r](V\oplus \irredC_k)_{+}\to \begin{cases}
\S{|V(r;\C)|} & \text{ if }k\equiv r\bmod n,\\
\pt & \text{ otherwise}.
\end{cases}\]
\end{theorem}
\begin{proof}
When $k\equiv r\bmod n$, we have $\irredC_k=\irredC_r$, and
\[(V\oplus \irredC_r)(r;\C)\cong V(r;\C)\oplus \irredC_r\]
so that by \autoref{thm:cell-structure-c}, the cofiber of the inclusion is $\S{\irredC_{-r}\otimes_{\C} V(r;\C)}$, which is the same as $\S{|V(r;\C)|}$ by \autoref{thm:trivial-sphere}. When $k\not\equiv r\bmod n$, we have that
\[(V\oplus\irredC_{k})(r;\C)\cong V(r;\C)\]
so that the inclusion is just the identity map, and its cofiber is therefore just a point.
\end{proof}

\begin{theorem}
\label{thm:fixed-cell-structure-h}
For any quaternionic $C_n$-representation $W$ and any $k$ and $r$, there is a cofiber sequence
\[\PH[r](W)_{+}\hookrightarrow \PH[r](W\oplus \irredH_k)_{+}\to \begin{cases}
\S{|W(r;\C)|} & \text{ if }k\equiv\pm r\bmod n,\\
\pt & \text{ otherwise.}
\end{cases}\]
\end{theorem}
\begin{proof}
When $k\equiv \pm r\bmod n$, we have $\irredH_k=\irredH_r\cong\irredC_{r}\oplus\irredC_{-r}$ by \autoref{lem:irredH-as-irredC}. If $2r\not\equiv 0\bmod n$, then $\PH[r](W)=\PC(W(r;\C))$ and
\[(W\oplus \irredH_r)(r;\C)=(W\oplus \irredC_r\oplus \irredC_{-r})(r;\C)\cong W(r;\C)\oplus \irredC_r\]
so that the result follows from \autoref{thm:fixed-cell-structure-c}. If $2r\equiv 0\bmod n$, then $\PH[r](W)=\PH(W(r;\H))$ and
\[(W\oplus \irredH_r)(r;\H)\cong W(r;\H)\oplus \irredH_r\]
so that the result follows from \autoref{thm:cell-structure-h} and \autoref{thm:trivial-sphere}. When $k\not\equiv \pm r\bmod n$, the inclusion is just the identity map, and again this implies its cofiber is therefore just a point.
\end{proof}

\subsection{Fixed points}

\begin{theorem}
\label{thm:fixed-points-c}
For any complex $C_n$-representation $V$, 
\[\PC(V)^{C_n}=\bigsqcup_{r=0}^{n-1}\PC[r](V)\]
and
\[\PC(V)^{dC_n}=\PC({}_{n/d}V)^{C_{n/d}}.\]
\end{theorem}

\begin{proof}
Choosing an isomorphism $V\cong\bigoplus_{r=0}^{n-1}V(r;\C)$, write $(v_0,v_1,\ldots,v_{n-1})$ for an element of $V$. The action of $C_n$ on $V$ is
\[\mathbf{1}_n\cdot (v_0,v_1,\ldots,v_{n-1})=(v_0,v_1\zeta_n,\ldots,v_{n-1}\zeta_n^{n-1}).\]
Any element of $\PC(V)$ is represented by a $v\in V\setminus\{0\}$ with $v_r= 1$ for some $r$. Then $\mathbf{1}_n\cdot\langle v\rangle=\langle v\rangle$ in $\PC(V)$ if and only if
\[(v_0, v_1\zeta_n,\ldots,\zeta_n^r,\ldots,v_{n-1}\zeta_n^{n-1})=(\lambda v_0,\lambda v_1,\ldots,\lambda,\ldots,\lambda v_{n-1})\]
for some $\lambda\in\C^\times$, which happens precisely when $\lambda =\zeta_n^r$. This forces $v_k=0$ for $k\neq r$, because 
\[v_k\zeta_n^k=\zeta_n^kv_k=\zeta_n^rv_k\]
implies $(\zeta_n^k-\zeta_n^r)v_k=0$, so that $v\in V(r;\C)\setminus\{0\}$. Moreover, for any $v\in V(r;\C)\setminus\{0\}$, we do indeed have
\[\mathbf{1}_n\cdot(0,0,\ldots,v,\ldots,0)=(0,0,\ldots,v\zeta_n^r,\ldots,0)=(0,0,\ldots,\zeta_n^rv,\ldots,0)=(0,0,\ldots,\lambda v,\ldots,0).\]
Thus, $\PC(V)^{C_n}$ is a disjoint union of the spaces $\PC[r](V)=\PC(V(r;\C))$. The more general claim for $\PC(V)^{dC_n}$ follows from a similar argument as the one in \autoref{prop:dCn-fixed}.
\end{proof}

\begin{remark}
At least when $V$ is of countable dimension, we can also see this from the cell structure of $\PC[r](V)$ via an inductive argument, since the claim is clearly true for $V=0$, and \autoref{thm:fixed-cell-structure-c} shows
how the cells that get added to $\PC[r](V)$ are of dimension $|V(r;\C)|$, which increases by 2 when a copy of $\irredC_r$ is added to $V$.
\end{remark}


\begin{theorem}
\label{thm:fixed-points-h}
For any quaternionic $C_n$-representation $W$, 
\[\PH(W)^{C_n}=\bigsqcup_{r=0}^{\lfloor n/2\rfloor}\PH[r](W)\]
and
\[\PH(W)^{dC_n}=\PH({}_{n/d}W)^{C_{n/d}}.\]
\end{theorem}

\begin{proof}
Choosing an isomorphism $W\cong\bigoplus_{r=0}^{\lfloor n/2\rfloor}W(r;\H)$, write $(w_0,w_1,\ldots,w_{\lfloor n/2\rfloor})$ for an element of $W$. The action of $C_n$ on $W$ is
\[\mathbf{1}_n\cdot (w_0,w_1,\ldots,w_{\lfloor n/2\rfloor})=(w_0,w_1\zeta_n,\ldots,w_{\lfloor n/2\rfloor}\zeta_n^{\lfloor n/2\rfloor}).\]
Any element of $\PH(W)$ is represented by a $w\in W\setminus\{0\}$ with $w_r= 1$ for some $r$. Then $\mathbf{1}_n\cdot\langle w\rangle=\langle w\rangle$ in $\PH(W)$ if and only if
\[(w_0,w_1\zeta_n ,\ldots,\zeta_n^r,\ldots,w_{\lfloor n/2\rfloor}\zeta_n^{\lfloor n/2\rfloor})=(\lambda v_0,\lambda v_1,\ldots,\lambda,\ldots,\lambda v_{n-1})\]
for some $\lambda\in\H^\times$, which happens precisely when $\lambda =\zeta_n^r$. This forces $w_k=0$ for $k\neq r$, because if we write $w_k=z_k+s_kj$, then
\[(z_k+s_kj)\zeta_n^k=\zeta_n^kz_k+\zeta_n^{-k}s_kj=\zeta_n^rz_k+\zeta_n^rs_kj=\zeta_n^r(z_k+s_kj)\]
implies $(\zeta_n^k-\zeta_n^r)z_k=0$ and $(\zeta_n^{-k}-\zeta_n^r)s_k=0$, so that $w\in W(r;\H)\setminus\{0\}$. However, for any $w=z+sj\in W(r;\H)\setminus\{0\}$, we have
\[\mathbf{1}_n\cdot(0,0,\ldots,w,\ldots,0)=(0,0,\ldots,w\zeta_n^r,\ldots,0)=(0,0,\ldots,\zeta_n^rw,\ldots,0)=(0,0,\ldots,\lambda w,\ldots,0)\] Thus, $\PH(W)^{C_n}$ is a disjoint union of the spaces $\PH[r](W)=\PH(W(r;\C))$. The more general claim for $\PC(V)^{dC_n}$ follows from a similar argument as the one in \autoref{prop:dCn-fixed}.
\end{proof}

%

\begin{remark}
At least when $W$ is of countable dimension, we can also see this from the cell structure of $\PH[r](W)$ via an inductive argument, since the claim is clearly true for $W=0$, and \autoref{thm:fixed-cell-structure-h} shows
how the cells that get added to $\PH[r](W)$ are of dimension $|W(r;\C)|$, which increases by either 4 or 2 when a copy of $\irredH_r$ is added to $W$, depending on whether $2r\equiv 0\bmod n$.
\end{remark}


\chapter{\texorpdfstring{Constructing $B_{G}\SU(2)$}{Constructing B\_GSU(2)}}

\section{Equivariant principal bundles and classifying spaces}

We recall some material from \cite[Chapter VII, p.59]{alaska}. We allow $\Pi$ to be any compact Lie group, but continue with our restriction that $G$ be a finite group. 

\begin{definition}
A \defstyle{principal $(G,\Pi)$-bundle} is a principal $\Pi$-bundle $p\from E\to B$ such that $E$ and $B$ are $G$-spaces, $p$ is a $G$-equivariant map, and the actions of $G$ and $\Pi$ on $E$ commute.
\end{definition}

\begin{remark}
\label{rem:lambda-action}
Recall that in a principal $\Pi$-bundle $p\from E\to B$, the space $E$ has a free right $\Pi$-action, which we write on the left as $(\pi,e)\mapsto e\pi^{-1}$. Thus, having the left action of $G$ and the right action of $\Pi$ on $E$ commute is equivalent to having a left $(
G\times\Pi)$-action on $E$, where $(g,\pi)\cdot e = ge\pi^{-1}$.
\end{remark}

By the end of this chapter, we will be most interested in the case where $G=C_n$ and $\Pi=\SU(2)$.

\begin{definition}
A principal $(G,\Pi)$-bundle $p\from E\to B$ is \defstyle{universal} if, for any $X$ with the homotopy type of a $G$-CW complex, pullback of $p$ along $G$-maps $X\to B$ gives a bijection from $[X,B]_{G}$ to the set of equivalence classes of $(G,\Pi)$-bundles on $X$. When this is the case, the $G$-space $B=E/\Pi$ is said to be an  \defstyle{(equivariant) classifying space} for principal $(G,\Pi)$-bundles.
\end{definition}

\begin{remark}
\label{rem:classifying-lambdas}
If  $\Lambda< G\times \Pi$ is a closed subgroup such that $\Lambda\cap \Pi$ is trivial, then \[\Lambda=\{(h,\rho(h)):h\in H\}\] for some subgroup $H\leq G$ and homomorphism $\rho\from H\to \Pi$. Such a $\Lambda$ acts on the total space $E$ of a principal $(G,\Pi)$-bundle as in \autoref{rem:lambda-action}.
\end{remark}

Recall that the universal principal $\Pi$-bundle $E\to B$, is characterized, up to equivalence, by its total space $E$ being contractible. The following result from \cite[Thm. 2.14, p.268]{lashof-bundles} generalizes this fact and allows us to recognize universal principal $(G,\Pi)$-bundles.

\begin{theorem}
\label{thm:classifying-lambda}
A principal $(G,\Pi)$-bundle $E\to B$ is universal if and only if $E^\Lambda$ is contractible for all closed subgroups $\Lambda< G\times \Pi$ such that $\Lambda\cap \Pi$ is trivial.
\end{theorem}



In particular, in a universal principal $(G,\Pi)$-bundle $E\to B$, we must have that $E$ itself is contractible, which leads us to the following corollary.

\begin{corollary}
\label{cor:universal-pi}
A universal $(G,\Pi)$-bundle $p:E\to B$ is also a universal $\Pi$-bundle.
\end{corollary}

This result from \cite[Theorem 10]{lashof-may-bundles} describes the $H$-fixed points of the equivariant classifying space $B(G,\Pi)$ for any subgroup $H\leq G$ as a union of classifying spaces.

\begin{theorem}
\label{thm:general-fixed-points}
For any subgroup $H\leq G$, we have 
\[B(G,\Pi)^H=\bigsqcup B(\Pi\cap N_{G\times\Pi}(\Lambda))\]
where the union runs over the $\Pi$-conjugacy classes of subgroups $\Lambda$ of $G\times\Pi$ such that $\Lambda\cap \Pi=1$ and the image of $\Lambda$ in $G$ under projection from $G\times\Pi$ is $H$.
\end{theorem}

\begin{remark}
In fact, the classifying spaces occurring in this union are other equivariant classifying spaces, and $B(G,\Pi)^H$ has the structure of a $W_GH$-space, where $W_GH$ is the Weyl group $N_GH/H$ (see the paragraph following the theorem in \cite{lashof-may-bundles}). However, we do not have need of this additional information here.
\end{remark}

\begin{remark}
We can use this result to check our computations in \autoref{thm:fixed-points-c} and \autoref{thm:fixed-points-h}. Let $G=C_n$, $H=dC_n$, and $\Pi=\SO(2)\cong \S{1}$. By \autoref{rem:classifying-lambdas},  subgroups $\Lambda$ of the kind mentioned in \autoref{thm:general-fixed-points} correspond to homomorphisms $\rho\from dC_n\to \SO(2)$, which in turn correspond to a choice of $0\leq k < n/d$ for $\rho(d\cdot\mathbf{1}_n)=\zeta_{n/d}^{k}$, of which there are $n/d$. They are all in distinct conjugacy classes because $G\times\Pi$ is abelian. Moreover, because $G\times\Pi$ is abelian we have that $\Pi\cap N_{G\times\Pi}(\Lambda)=\Pi\cap G\times \Pi=\Pi$, so that $B_{C_n}\SO(2)^{dC_n}$ is confirmed to be (non-equivariantly) a disjoint union of $n/d$ copies of $B\SO(2)\cong\C\P^\infty$.

Now let $G=C_n$, $H=dC_n$, and $\Pi=\SU(2)$. By \autoref{rem:classifying-lambdas},  subgroups $\Lambda$ of the kind mentioned in \autoref{thm:general-fixed-points} correspond to homomorphisms $\rho\from dC_n\to \SU(2)$. Any matrix in $\SU(2)$ of finite order is necessarily diagonalizable, and conjugation by $\bigl(\begin{smallmatrix}
0 & -1\\ 1 & 0
\end{smallmatrix}\bigr)$ has the effect of switching the two entries on the diagonal, so that up to conjugacy such homomorphisms correspond to a choice of $0\leq k < \lfloor n/2d\rfloor$ for
\[\rho(d\cdot\mathbf{1}_n)=\begin{pmatrix}
\zeta_{n/d}^{k} & 0 \\ 0 & \zeta_{n/d}^{-k}
\end{pmatrix}\]
of which there are $\lfloor n/2d\rfloor$. Considering $\SO(2)$ as the subgroup of $\SU(2)$ consisting of the diagonal matrices, we have that the normalizer of the cyclic subgroup $\Lambda<C_n\times\SU(2)$ is
\[N_{C_n\times\SU(2)}(\Lambda)=\begin{cases}
\SU(2) & \text{if }\rho(d\cdot\mathbf{1}_n)=\bigl(\begin{smallmatrix}
\pm1 & 0 \\ 0 & \pm1
\end{smallmatrix}\bigr) \\
\SO(2) & \text{otherwise}
\end{cases}\]
so that $B_{C_n}\SU(2)^{dC_n}$ is confirmed to be (non-equivariantly) a disjoint union of $\lfloor n/2d\rfloor$ pieces, with either one or two copies of $B\SU(2)\cong\HP^\infty$ (two when $n/d$ is even), and the rest copies of $B\SO(2)\cong\C\P^\infty$.
\end{remark}


%
%

\section{\texorpdfstring{A model for $B_{G}\SU(2)$}{A model for B\_GSU(2)}}

\begin{definition}
A complete quaternionic $G$-universe $Q$ is a left $\H$-module that is a direct sum of a countably infinite number of copies of each irreducible quaternionic $G$-representation.
\end{definition}

\begin{theorem}
\label{classifying-space}
If $Q$ is a complete quaternionic $G$-universe, then $\PH(Q)$ is a model for $B_{G}\SU(2)$.
\end{theorem}
\begin{proof}
We claim it is sufficient to prove that, for any closed $\Lambda<G\times\SU(2)$ such that $\Lambda\cap \SU(2)$ is trivial, there is a quaternionic $G$-representation $V$ for which $V^\Lambda$ is a non-trivial real subspace of $V$, where the action of $G\times \SU(2)$ on $V$ is via $(g,s)\cdot v = g(sv)$ for $g\in G$, $s\in \SU(2)$, and $v\in V$. 

Take this claim as given for now, and recall that \autoref{rem:classifying-lambdas} classifies such subgroups $\Lambda<G\times\SU(2)$. 

First, let  $\rho\from G\to \SU(2)$ be a homomorphism, and let $\Lambda=\{(g,\rho(g)):g\in G\}$. Define the quaternionic $G$-representation $V$ to be $\H$, with action $\varphi\from G\times V\to V$ defined by $\varphi(g,v)= v\rho(g)^{-1}$, which is left $\H$-linear as required. Then the action of $\Lambda$ on $V$ is
\[(g,\rho(g))\cdot v = g(\rho(g)v)=\rho(g)v\rho(g)^{-1},\]
so at least the real subspace of $V=\H$ spanned by the element $1\in\H$ is fixed. Thus, for such subgroups $\Lambda$, we have found an irreducible $G$-representation $V$ for which $V^\Lambda$ is a non-trivial real subspace of $V$.

More generally, if $\Lambda$ is formed by taking a subgroup $H<G$ and a homomorphism $\rho\from H\to \SU(2)$ so that $\Lambda=\{(h,\rho(h)):h\in H\}$, then we can similarly define the quaternionic $H$-representation $W$ to be $\H$, with action $\varphi\from H\times W\to W$ defined by $\varphi(h,w)= w\rho(h)^{-1}$, and consider the induced quaternionic $G$-representation $V=\mathrm{ind}_H^GV$, since its subspace $W$ is invariant under the action of $\Lambda$, and hence the fixed real subspace of $W$ is also a fixed real subspace of $V$.

Now we will prove our earlier claim that this is sufficient. Let $\unitsphere{Q}$ be the unit sphere of $Q$, and recall that the unit sphere in the quaternions is $\unitsphere{\H}=\SU(2)$. There is a left $G$-action on $\unitsphere{Q}$, which comes from $Q$ being a $G$-representation, and there is a right $\SU(2)$-action on $\unitsphere{Q}$, which comes from scalar multiplication, $(q,u)\mapsto u^{-1}q$ for $q\in\unitsphere{Q}$ and $u\in\SU(2)$. These actions commute because the left $G$-action on a quaternionic $G$-representation must be left $\H$-linear.

Observe that the quotient of $\unitsphere{Q}$ by the right $\SU(2)$-action is $\PH(Q)$, and the quotient map $p\from\unitsphere{Q}\to\PH(Q)$ is $G$-equivariant. Thus $p$ is a principal $(G,\SU(2))$-bundle, and by \autoref{rem:lambda-action}, we can consider $\unitsphere{Q}$ as a left $(G\times\SU(2))$-space, where $(g,s)\cdot q = g(sq)$ for $g\in G$, $s\in \SU(2)$, and $q\in\unitsphere{Q}$.

Consider $Q$ as a left $(G\times\SU(2))$-space with the same action. Then for any subgroup $\Lambda\leq G\times\SU(2)$,
\[\unitsphere{Q}^{\Lambda}=Q^\Lambda\cap \unitsphere{Q}=\unitsphere{Q^\Lambda}.\]
In fact, $Q^\Lambda$ is a real subspace of $Q$, because for any $q_1,q_2\in Q^\Lambda$ and any $(g,s)\in\Lambda$ we have
\[(g,s)\cdot (q_1+q_2) = g(s(q_1+q_2))=g(sq_1+sq_2)=g(sq_1)+g(sq_2)=(g,s)\cdot q_1 + (g,s)\cdot q_2=q_1+q_2\]
and for any $q\in Q^\Lambda$, any $r\in\mathbb{R}$, and any $(g,s)\in\Lambda$ we have
\[(g,s)\cdot (rq)=g(s(rq))=g((sr)q)=g((rs)q)=rg(sq)=rq\]
because the action of $G$ is left $\H$-linear, and real numbers commute with any quaternion.

If $Q^\Lambda$ is an infinite-dimensional real subspace of $Q$ for a given $\Lambda\leq G\times\SU(2)$, then $\unitsphere{Q}^\Lambda$ is an infinite-dimensional sphere, and hence contractible. Thus, if we prove that $Q^\Lambda$ is infinite-dimensional for all closed $\Lambda<G\times\SU(2)$ such that $\Lambda\cap \SU(2)$ is trivial, then \autoref{thm:classifying-lambda} implies that the map $p\from\unitsphere{Q}\to\PH(Q)$ is a universal principal $(G,\SU(2))$-bundle, and we can conclude that $\PH(Q)$ is a model for $B_{G}\SU(2)$. Therefore, we can conclude that given such a subgroup $\Lambda$, if there is any quaternionic $G$-representation $V$ such that $V^\Lambda$ is a non-trivial real subspace of $V$, then $Q^\Lambda$ is infinite-dimensional because $Q$ contains infinitely many isomorphic copies of $V$.
\end{proof}

\begin{remark}
Observe that non-equivariantly, $\PH(Q)$ is $\HP^\infty$, which is a model of $B\SU(2)$. This is consistent with \autoref{cor:universal-pi}.
\end{remark}

%
%

\section{\texorpdfstring{A properly even $C_n$-cell complex structure for $B_{C_n}\SU(2)$}{A properly even C\_n-cell complex structure for B\_\{C\_n\}SU(2)}}

\begin{definition}
Given a quaternionic $C_n$-representation $W$ of countable dimension, a \defstyle{split full flag} of $W$ is a collection $\flagstyle{U}$ of irreducible quaternionic subrepresentations $\{U_k\}_{k\geq 0}$ of $W$ such that $W=\bigoplus_{k\geq 0} U_k$. For any split full flag $\flagstyle{U}$ of $W$, define $\flagstyle{U}_k=\bigoplus_{i=0}^{k-1}U_i$, so that $\flagstyle{U}_{k+1}=U_k\oplus\flagstyle{U}_k$. Let
\[\pt=\PH(\flagstyle{U}_1)\subseteq\PH(\flagstyle{U}_2)\subseteq\cdots\subseteq \PH(W)\]
be the $\flagstyle{U}$-filtration of $\PH(W)$. Lastly, let $\cellstyle{u}_k$ be $\irredC_{-r}\otimes_{\C} \flagstyle{U}_{k}$ considered as a real representation, where $r$ is determined by $U_{k}\cong \irredH_r$.
\end{definition}
\begin{remark}
Note that $r$ is only determined up to sign modulo $n$, since $\irredH_r\cong\irredH_{-r}$, and even though $\irredC_{r}\not\cong\irredC_{-r}$ as complex $C_n$-representations, this does not introduce any ambiguity in $\cellstyle{u}_k$, which is considered as a real representation.
\end{remark}

\begin{proposition}
If $\flagstyle{U}$ is a split full flag of a quaternionic $C_n$-representation $W$, then the $\flagstyle{U}$-filtration of $\PH(W)$ defines a $C_n$-cell complex structure on $\PH(W)$, where exactly one cell $\unitdisk{\cellstyle{u}_k}$ is attached to $\PH(\flagstyle{U}_k)$ to construct $\PH(\flagstyle{U}_{k+1})$.
\end{proposition}
\begin{proof}
This follows directly from \autoref{thm:cell-structure-h}.
\end{proof}



\begin{definition}
\label{def:PH-even-monotone}
Let $Q$ be a complete quaternionic $C_n$-universe, and let $\flagstyle{W}$ be the split full flag on $Q$ defined by $W_k=\irredH_k$, so that $\flagstyle{W}_k=\bigoplus_{i=0}^{k-1}\irredH_i$ and  \[\cellstyle{w}_k=\irredC_{-k}\otimes_{\C} \flagstyle{W}_k=\irredC_{-k}\otimes_{\C}\bigoplus_{i=0}^{k-1}\irredH_i\cong\bigoplus_{i=0}^{k-1}(\irredC_{i-k}\oplus\irredC_{-i-k}).\]
\end{definition}

\begin{lemma}
\label{lem:key-even-monotone-lemma}
With $\flagstyle{W}$ as in \autoref{def:PH-even-monotone}, we have
\[|\cellstyle{w}_k^{C_n}|=|\flagstyle{W}_k(k;\C)|=4\left\lfloor \frac{k}{n}\right\rfloor+ 2\left\lfloor\frac{2(k-\lfloor k/n\rfloor n)}{n+1}\right\rfloor.\]
This quantity is non-decreasing with $k$.
\end{lemma}
\begin{proof}
By \autoref{thm:trivial-sphere},
\[|\cellstyle{w}_k^{C_n}|=|(\irredC_{-k}\otimes_{\C}\flagstyle{W}_{k})^{C_n}|=|\flagstyle{W}_k(k;\C)|.\]
Observe that by \autoref{thm:class-of-irred-H}, $\flagstyle{W}_{an}=\bigoplus_{i=0}^{an-1}\irredH_i$ contains $a$ copies of $\irredH_i$ when $2i\equiv 0\bmod n$, and $2a$ copies of $\irredH_i$ otherwise. By \autoref{prop:H-isotypical-as-C-isotypical}, we conclude that $|\flagstyle{W}_{an}(k;\C)|=4a$ for any $k$ and $a$. Again using \autoref{thm:class-of-irred-H}, we have that for any $k$,
\[\flagstyle{W}_{k}\cong\flagstyle{W}_{\lfloor k / n\rfloor n}\oplus \flagstyle{W}_{k-\lfloor k / n\rfloor n},\]
so that 
\[|\flagstyle{W}_{k}(k;\C)|=4\lfloor k / n\rfloor+|\flagstyle{W}_{k-\lfloor k / n\rfloor n}(k;\C)|.\]
Because $k\equiv k-\lfloor k/n\rfloor n\bmod n$, we have that
\[\flagstyle{W}_{k-\lfloor k / n\rfloor n}(k;\C)=\flagstyle{W}_{k-\lfloor k / n\rfloor n}(k-\lfloor k / n\rfloor;\C),\] so now all that remains to compute the quantity $|\flagstyle{W}_{k}(k;\C)|$ in general is to do so in the case that $0\leq k\leq n-1$. It is straightforward to check that when $0\leq k\leq n-1$,
\[|\flagstyle{W}_{k}(k;\C)|=\left.\begin{cases}
0 & \text{if } k\leq \lfloor\frac{n}{2}\rfloor,\\
2 & \text{otherwise}
\end{cases}\right\}=2\bigl\lfloor\tfrac{2k}{n+1}\bigr\rfloor\]
Therefore the quantity
\begin{align*}
|\cellstyle{w}_k^{C_n}|=|\flagstyle{W}_k(k;\C)|&=4\left\lfloor \frac{k}{n}\right\rfloor+ \begin{cases}
0 & \text{if } k-\lfloor k/n\rfloor n\leq \lfloor\frac{n}{2}\rfloor,\\
2 & \text{otherwise}
\end{cases}\\
&=4\left\lfloor \frac{k}{n}\right\rfloor+ 2\left\lfloor\frac{2(k-\lfloor k/n\rfloor n)}{n+1}\right\rfloor
\end{align*}
is non-decreasing with $k$.
\end{proof}

\begin{theorem}
\label{thm:PH-even-monotone}
Let $Q$ be a complete quaternionic $C_n$-universe, and let $\flagstyle{W}$ be the split full flag on $Q$ from \autoref{def:PH-even-monotone}. Then the $\flagstyle{W}$-filtration of $\PH(Q)$ is a properly even $C_n$-cell complex structure.
\end{theorem}
\begin{proof}
We know that the cell added to go from $\PH(\flagstyle{W}_k)$ to $\PH(\flagstyle{W}_{k+1})$ is $\unitdisk{\cellstyle{w}_k}$, and we can check each condition from \autoref{def:even-monotone}, although condition (b) is the only one requiring any work.
\begin{itemize}
\item[(a)] For any group $G$ and any finite-dimensional complex $G$-representation $V$, we must have that $|V^H|$ is even for any subgroup $H\leq G$, because $V^H$ will be a complex subrepresentation of $V$. Observe that $\cellstyle{w}_k$ is just $\bigoplus_{i=0}^{k-1}(\irredC_{i-k}\oplus\irredC_{-i-k})$ considered as a real $C_n$-representation, but that the aforementioned fact is still true about $\cellstyle{w}_k$. Thus, every cell $\unitdisk{\cellstyle{w}_k}$ is even.

\item[(b)] In order to prove the desired property of the filtration, it will suffice to show that $\cellstyle{w}_k\ll \cellstyle{w}_{k+1}$ for every positive integer $k$. Because $|\cellstyle{w}_k|=4k$, we have that \[|\cellstyle{w}_k^{nC_n}|=4k<4k+4=|\cellstyle{w}_{k+1}^{nC_n}|\]
for the trivial subgroup $nC_n\leq C_n$. Therefore, to prove that $\cellstyle{w}_k\ll \cellstyle{w}_{k+1}$ for a given $n$ and $k$, we must establish that $|\cellstyle{w}_k^{dC_n}|\leq|\cellstyle{w}_{k+1}^{dC_n}|$ for every $d\mid n$. 

We have established the result for any $n$ and $k$, with $d=1$, in \autoref{lem:key-even-monotone-lemma}.

Continuing the notational pattern established in \autoref{subsec:fixed-restricted}, we will now write ${}_n\cellstyle{w}_k$, instead of just $\cellstyle{w}_k$, to indicate which $C_n$ we are working over. Thus, we can write ${}_n\flagstyle{W}_k\cong\bigoplus_{i=0}^{k-1}{}_n\irredH_i$ and  ${}_n\cellstyle{w}_k={}_n\irredC_{-k}\otimes_{\C} {}_n\flagstyle{W}_k$, for example.

For $d\mid n$ with $d>1$, we have by \autoref{prop:dCn-fixed} that
\[|{}_n\cellstyle{w}_k^{dC_n}|=|{}_{n/d}\cellstyle{w}_k^{C_{n/d}}|\]
which reduces the problem to the case we have already shown.

%
%
%

\item[(c)] By our work for part (b), we know that for any positive integer $a$, $|\flagstyle{W}_{an}(i;\C)|\geq 2a$ for any $i$, so that the same is true of $\flagstyle{W}_k$ for any $k\geq an$, and thus $\cellstyle{w}_k=\irredC_{-k}\otimes_{\C}\flagstyle{W}_k$ contains at least $2a$ copies of $\irredC_{0}$, which implies $|\cellstyle{w}_k^{C_n}|\geq 2a$, and hence $|\cellstyle{w}_k^{dC_n}|\geq 2a$ as well. Thus, for any $N$, there are only finitely many cells $\unitdisk{\cellstyle{w}_k}$ with $|\cellstyle{w}_k^{dC_n}|\leq N$ for all subgroups $dC_n\leq C_n$.

\item[(d)] We start with exactly one point $\PH(\flagstyle{W}_1)=\pt$ and add one cell $\unitdisk{\cellstyle{w}_k}$ at a time, so each $\PH(\flagstyle{W}_k)$ is a finite $C_n$-cell complex.\qedhere

\end{itemize}
\end{proof}

\begin{remark}
If we define $\flagstyle{U}$ to be the split full flag on $Q$ defined by $U_k=\irredH_{k\bmod (\lfloor n/2\rfloor + 1)}$, so that figuratively, the irreducibles in $Q$ are put in order like $\irredH_0\oplus \irredH_1\oplus\cdots \oplus \irredH_{\lfloor\frac{n}{2}\rfloor}\oplus \irredH_0\oplus \irredH_1\oplus\cdots$,
then the $\flagstyle{U}$-filtration of $\PH(Q)$ is not properly even (except when $n=2$).
\end{remark}


\begin{remark}
\label{rem:proper-meaning}
The geometric meaning of a properly even cell structure is described here from \cite[p.72]{lewis_complex}. Though he states it for the complex projective space of a complex $C_p$-representation, it remains true in the quaternionic case.
\begin{quotation}
The $G$-fixed subspace of $\mathrm{P}(\Phi)$ is a disjoint union of complex projective spaces, one for each isomorphism class of irreducibles in $\Phi$. The (complex) dimension of the complex projective space in $\mathrm{P}(\Phi)^G$ associated to the irreducible $\phi$ is one less than the multiplicity of $\phi$ in $\Phi$. Thus, the effect of properly ordering the irreducibles is that the maximal dimension of the components of the $G$-fixed subspace of $\mathrm{P}(\{\phi_i\}_{0\leq i\leq k})$ increases as slowly as possible with increasing $k$.
\end{quotation}
\end{remark}


\chapter{\texorpdfstring{Additive Structure of $\HHreduced_{G}^{*}(B_{G}\SU(2)_{+}\coeffs*{A})$}{Additive Structure of H\_G*(B\_GSU(2)\_+;A)}}


\section{Dimensions of generators}

Here, we can record the following result about the additive structure of $\HHreduced_{G}^{*}(B_{G}\SU(2)_{+}\coeffs*{A})$ for those groups $G$ about which an even-dimensional freeness result has been established.

\begin{theorem}
\label{thm:additive}
Let $G=C_p$ for any prime $p$, or $C_{pq}$ for distinct odd primes $p,q$. As a module over $\HHreduced_{G}^{*}(\S{0}\coeffs*{A})$, the cohomology $\HHreduced_{G}^{*}(B_{G}\SU(2)_{+}\coeffs*{A})$ is free, and decomposes as a direct sum
\[\HHreduced_{G}^{*}(B_{G}\SU(2)_{+}\coeffs*{A})\cong \bigoplus_{k\geq 0}\Sigma^{\cellstyle{w}_k}\HHreduced_{G}^{*}(\S{0}\coeffs*{A})\]
where the $\cellstyle{w}_k$ are the representations introduced in \autoref{def:PH-even-monotone}.
\end{theorem}
\begin{proof}
We directly combine \autoref{thm:even-dim-free} about even-dimensional freeness and \autoref{thm:PH-even-monotone} proving that the $\flagstyle{W}$-filtration gives a properly even $C_n$-cell structure on $B_{C_n}\SU(2)$.
\end{proof}

In \autoref{fig:additive}, we have placed a red circle at the points $(|\cellstyle{w}_k^{C_p}|,|\cellstyle{w}_k|)$ for $1\leq k\leq 8$. That is, we have marked certain information about the dimensions $\alpha\in\RO(C_p)$ in which there is a generator of $\HHreduced_{C_p}^*(B_{C_p}\SU(2)_{+}\coeffs*{A})$ as an $\HHreduced_{C_p}^*(\S{0}\coeffs*{A})$-module. 

Observe that with a different choice of split full flag, we would end up with generators in different dimensions in $\RO(C_p)$. However, because the resulting $\HHreduced_{C_p}^*(\S{0}\coeffs*{A})$-module structure of $\HHreduced_{C_p}^*(B_{C_p}\SU(2)_{+}\coeffs*{A})$ would have to be isomorphic to the one described in \autoref{thm:additive}, the generators would have to be in dimensions $\alpha$ with the same values of $(|\alpha^{C_p}|,|\alpha|)$. That is, the resulting diagram would have the red circles in the same locations.

We see an interesting phenomenon for odd primes $p$, which is due to the following fact: even though the $C_p$-cell complex structure on $B_{C_p}\SU(2)$ adds the irreducible quaternionic representations ``in order'' $\irredH_0,\irredH_1,\irredH_2,\ldots$, this adds irreducible complex representations (of which there are an odd number) in pairs and ``out of order'', so that it takes $p$ cells until all irreducible complex representations are present in an equal amount.

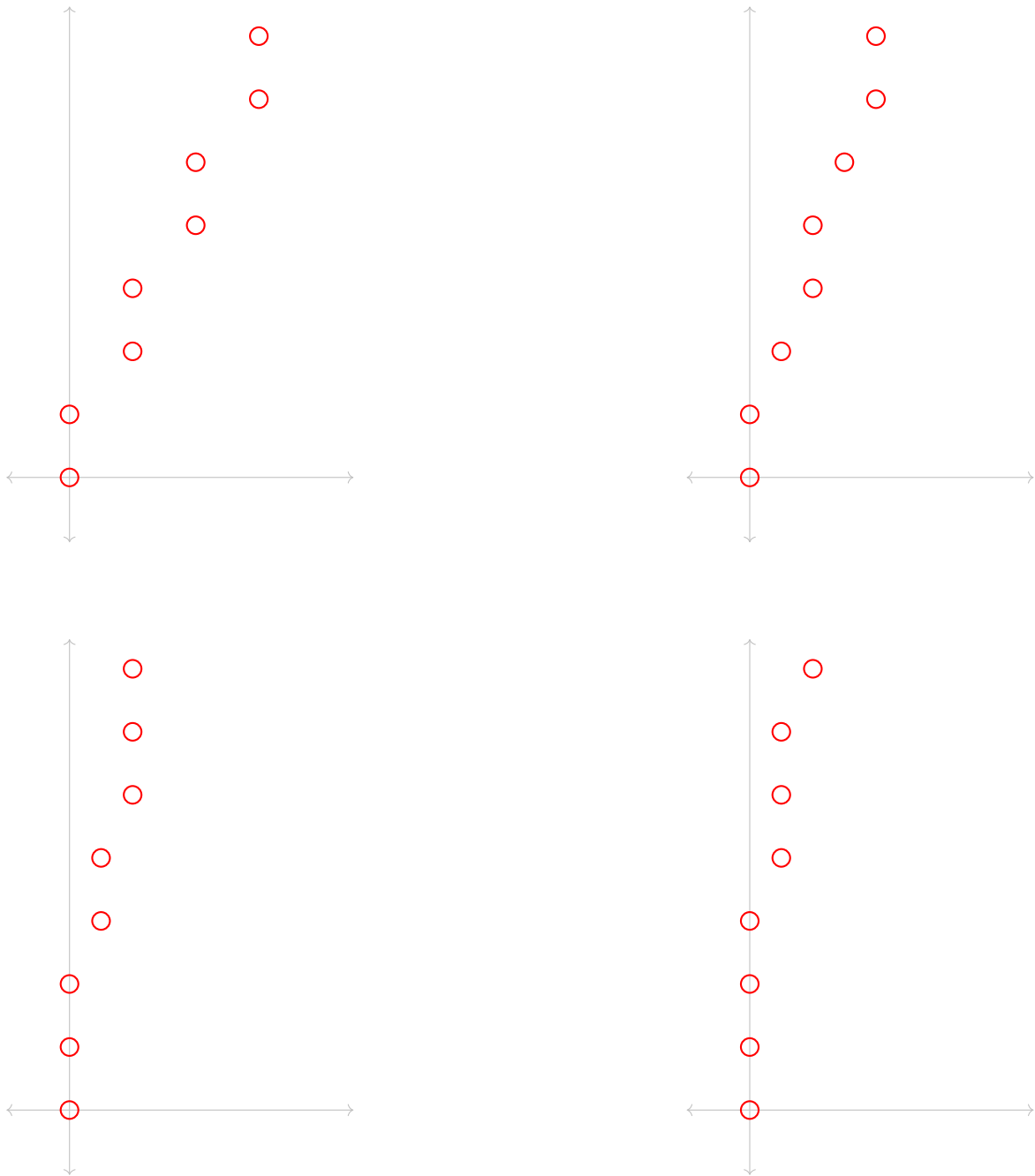
\begin{figure}
\begin{center}
\begin{tikzpicture}[xscale=0.95,yscale=0.95, every node/.style={fill opacity=0, text opacity=1,inner sep=2pt,outer sep=2pt}]
\path[use as bounding box] (-1.1,-1.1) rectangle (5.5,8.9);
\draw[gray!50,<->] (-1,0.03) -- (4.5,0.03);
\draw[gray!50,<->] (0,-1) -- (0,7.5);
\node at (2,8.5) {$p=2$};
\node at (0,7.9) {$|\alpha|$};
\node at (5.1,0) {$|\alpha^{C_2}|$};
\foreach \x in {0.25,0.5,...,4.25} {%
   \foreach \y in {0.25,0.5,...,7.25} {%
       \node[text=gray] at (\x,\y) {$\cdot$};
   }%
   \foreach \y in {-0.5,-0.25} {%
       \node[text=gray] at (\x,\y) {$\cdot$};
   }%
}%
\foreach \x in {-0.5,-0.25} {%
   \foreach \y in {0.25,0.5,...,7.25} {%
       \node[text=gray] at (\x,\y) {$\cdot$};
   }%
   \foreach \y in {-0.5,-0.25} {%
       \node[text=gray] at (\x,\y) {$\cdot$};
   }%
}%
\foreach \x in {0,1,...,7} {%
   \pgfmathsetmacro\result{floor(\x / 2)}
   \draw[thick,red] (\result,\x+0.03) circle (0.14);%
}%
\end{tikzpicture}\hfill \begin{tikzpicture}[xscale=0.95,yscale=0.95, every node/.style={fill opacity=0, text opacity=1}]
\path[use as bounding box] (-1.1,-1.1) rectangle (5.5,8.9);
\draw[gray!50,<->] (-1,0.03) -- (4.5,0.03);
\draw[gray!50,<->] (0,-1) -- (0,7.5);
\node at (2,8.5) {$p=3$};
\node at (0,7.9) {$|\alpha|$};
\node at (5.1,0) {$|\alpha^{C_p}|$};
\foreach \x in {0.25,0.75,...,4.25} {%
   \foreach \y in {0.25,0.75,...,7.25} {%
       \node[text=gray] at (\x,\y) {$\cdot$};
   }%
   \foreach \y in {-0.25} {%
       \node[text=gray] at (\x,\y) {$\cdot$};
   }%
}%
\foreach \x in {0.5,1,...,4} {%
   \foreach \y in {0.5,1,...,7} {%
       \node[text=gray] at (\x,\y) {$\cdot$};
   }%
   \foreach \y in {-0.5} {%
       \node[text=gray] at (\x,\y) {$\cdot$};
   }%
}%
\foreach \x in {-0.25} {%
   \foreach \y in {0.25,0.75,...,7.25} {%
       \node[text=gray] at (\x,\y) {$\cdot$};
   }%
   \foreach \y in {-0.25} {%
       \node[text=gray] at (\x,\y) {$\cdot$};
   }%
}%
\foreach \x in {-0.5} {%
   \foreach \y in {0.5,1,...,7} {%
       \node[text=gray] at (\x,\y) {$\cdot$};
   }%
   \foreach \y in {-0.5} {%
       \node[text=gray] at (\x,\y) {$\cdot$};
   }%
}%
\foreach \x in {0,1,...,7} {%
   \pgfmathsetmacro\result{floor(2 * \x / 3) / 2}
   \draw[thick,red] (\result,\x+0.03) circle (0.14);%
}%
\end{tikzpicture}
\end{center}
\begin{center}
\begin{tikzpicture}[xscale=0.95,yscale=0.95, every node/.style={fill opacity=0, text opacity=1}]
\path[use as bounding box] (-1.1,-1.1) rectangle (5.5,8.9);
\draw[gray!50,<->] (-1,0.03) -- (4.5,0.03);
\draw[gray!50,<->] (0,-1) -- (0,7.5);
\node at (2,8.5) {$p=5$};
\node at (0,7.9) {$|\alpha|$};
\node at (5.1,0) {$|\alpha^{C_p}|$};
\foreach \x in {0.25,0.75,...,4.25} {%
   \foreach \y in {0.25,0.75,...,7.25} {%
       \node[text=gray] at (\x,\y) {$\cdot$};
   }%
   \foreach \y in {-0.25} {%
       \node[text=gray] at (\x,\y) {$\cdot$};
   }%
}%
\foreach \x in {0.5,1,...,4} {%
   \foreach \y in {0.5,1,...,7} {%
       \node[text=gray] at (\x,\y) {$\cdot$};
   }%
   \foreach \y in {-0.5} {%
       \node[text=gray] at (\x,\y) {$\cdot$};
   }%
}%
\foreach \x in {-0.25} {%
   \foreach \y in {0.25,0.75,...,7.25} {%
       \node[text=gray] at (\x,\y) {$\cdot$};
   }%
   \foreach \y in {-0.25} {%
       \node[text=gray] at (\x,\y) {$\cdot$};
   }%
}%
\foreach \x in {-0.5} {%
   \foreach \y in {0.5,1,...,7} {%
       \node[text=gray] at (\x,\y) {$\cdot$};
   }%
   \foreach \y in {-0.5} {%
       \node[text=gray] at (\x,\y) {$\cdot$};
   }%
}%
\foreach \x in {0,1,...,7} {%
   \pgfmathsetmacro\result{floor(2 * \x / 5) / 2}
   \draw[thick,red] (\result,\x+0.03) circle (0.14);%
}%
\end{tikzpicture}\hfill
\begin{tikzpicture}[xscale=0.95,yscale=0.95, every node/.style={fill opacity=0, text opacity=1}]
\path[use as bounding box] (-1.1,-1.1) rectangle (5.5,8.9);
\draw[gray!50,<->] (-1,0.03) -- (4.5,0.03);
\draw[gray!50,<->] (0,-1) -- (0,7.5);
\node at (2,8.5) {$p=7$};
\node at (0,7.9) {$|\alpha|$};
\node at (5.1,0) {$|\alpha^{C_p}|$};
\foreach \x in {0.25,0.75,...,4.25} {%
   \foreach \y in {0.25,0.75,...,7.25} {%
       \node[text=gray] at (\x,\y) {$\cdot$};
   }%
   \foreach \y in {-0.25} {%
       \node[text=gray] at (\x,\y) {$\cdot$};
   }%
}%
\foreach \x in {0.5,1,...,4} {%
   \foreach \y in {0.5,1,...,7} {%
       \node[text=gray] at (\x,\y) {$\cdot$};
   }%
   \foreach \y in {-0.5} {%
       \node[text=gray] at (\x,\y) {$\cdot$};
   }%
}%
\foreach \x in {-0.25} {%
   \foreach \y in {0.25,0.75,...,7.25} {%
       \node[text=gray] at (\x,\y) {$\cdot$};
   }%
   \foreach \y in {-0.25} {%
       \node[text=gray] at (\x,\y) {$\cdot$};
   }%
}%
\foreach \x in {-0.5} {%
   \foreach \y in {0.5,1,...,7} {%
       \node[text=gray] at (\x,\y) {$\cdot$};
   }%
   \foreach \y in {-0.5} {%
       \node[text=gray] at (\x,\y) {$\cdot$};
   }%
}%
\foreach \x in {0,1,...,7} {%
   \pgfmathsetmacro\result{floor(2 * \x / 7) / 2}
   \draw[thick,red] (\result,\x+0.03) circle (0.14);%
}%
\end{tikzpicture}
\end{center}
\caption[Dimensions where additive generators of $\HHreduced_{C_p}^{*}(B_{C_p}\SU(2)_{+}\coeffs*{A})$ occur]{The equivariant dimensions $\alpha$ in which additive generators of the cohomology $\HHreduced_{C_p}^{*}(B_{C_p}\SU(2)_{+}\coeffs*{A})$ occur, plotted according to $|\alpha^{C_p}|$ and $|\alpha|$. Dimensions of elements of $\RO(C_p)$ are plotted with dots, whereas points not corresponding to any element of $\RO(C_p)$ are left blank. Observe the uneven stair pattern for odd primes $p$.}
\label{fig:additive}
\end{figure}

\section{Inverse limit}

\begin{theorem}
\label{thm:inverse-limit}
With $Q$ a complete quaternionic $C_p$-universe, and $\flagstyle{W}$ the split full flag of $Q$ defined in \autoref{thm:PH-even-monotone}, we have that
\[\HHreduced_{C_p}^*(\PH(Q)_{+}\coeffs*{A})=\varprojlim\limits_{n}\HHreduced_{C_p}^*(\PH(\flagstyle{W}_n)_{+}\coeffs*{A}).\]
\end{theorem}
\begin{proof}
\autoref{thm:even-dim-free} implies that $\HHreduced_{C_p}^*(\PH(\flagstyle{W}_n)_{+}\coeffs*{A})$ is a direct summand of $\HHreduced_{C_p}^*(\PH(\flagstyle{W}_{n+1})_{+}\coeffs*{A})$ for every $n$, with the map \[\HHreduced_{C_p}^*(\PH(\flagstyle{W}_{n+1})_{+}\coeffs*{A})\to\HHreduced_{C_p}^*(\PH(\flagstyle{W}_n)_{+}\coeffs*{A})\]
being left inverse to the inclusion, and therefore surjective. Thus, $\underset{n}{\varprojlim}^1\HHreduced_{C_p}^*(\PH(\flagstyle{W}_n)_{+}\coeffs*{A})=0$, so that the map $\HHreduced_{C_p}^*(\PH(Q)_{+}\coeffs*{A})\to\varprojlim\limits_{n}\HHreduced_{C_p}^*(\PH(\flagstyle{W}_n)_{+}\coeffs*{A})$ is an isomorphism.
\end{proof}

\chapter{\texorpdfstring{Multiplicative Structure of $\HHreduced_{C_2}^{*}(B_{C_2}\SU(2)_{+}\coeffs*{A})$}{Multiplicative Structure of H\_\{C\_2\}*(B\_\{C\_2\}SU(2)\_+;A)}}
\label{chap:mult}

In this chapter, we will establish the complete multiplicative structure of $\HHreduced_{C_2}^{*}(B_{C_2}\SU(2)_{+}\coeffs*{A})$ in \autoref{thm:main}.


\section{Preliminaries and notation}

\label{sec:prelim}

In this section, we will set notation for important maps arising in the cohomology of $\PH(Q)$, where $Q$ is a complete quaternionic $C_p$-universe, for any prime $p$. This draws on knowledge about quaternionic $C_p$-representations and their projective spaces from Chapter 2, in particular the structure of the fixed points $\PH(Q)^{C_p}$. In the subsequent sections, we study the case of $p=2$ and establish \autoref{thm:main}.


\begin{definition}
\label{def:ch5-towhole}
The maps $\Htowhole{r}$ and $\Htowhole[n]{r}$ are defined to be the inclusions
\[
\begin{tikzcd}[row sep=1.6cm,column sep=1.6cm]
\PH[r](Q)_{+} \ar[hook]{r}{\Htowhole{r}} & \PH(Q)_{+},
\end{tikzcd}\quad\begin{tikzcd}[row sep=1.6cm,column sep=1.6cm]
\PH[r](\flagstyle{W}_n)_{+} \ar[hook]{r}{\Htowhole[n]{r}} & \PH(\flagstyle{W}_n)_{+}
\end{tikzcd}
\] arising from the inclusions of the isotypical components $Q(r;\H)\hookrightarrow Q$ and $\flagstyle{W}_n(r;\H)\hookrightarrow \flagstyle{W}_n$, respectively.
\end{definition}

\begin{definition}
\label{def:ch5-tonext}
The maps $\Htonext{n}$ and $\Htonext[r]{n}$ are defined to be the inclusions
\[\begin{tikzcd}[row sep=1.6cm,column sep=1.6cm]
\PH(\flagstyle{W}_n)_{+} \ar[hook]{r}{\Htonext{n}} & \PH(\flagstyle{W}_{n+1})_{+} ,
\end{tikzcd}\quad\begin{tikzcd}[row sep=1.6cm,column sep=1.6cm]
\PH[r](\flagstyle{W}_n)_{+} \ar[hook]{r}{\Htonext[r]{n}} & \PH[r](\flagstyle{W}_{n+1})_{+} 
\end{tikzcd}\]  arising from the inclusions of the subrepresentations $\flagstyle{W}_n\hookrightarrow \flagstyle{W}_{n+1}$ and their isotypical components, respectively.
\end{definition}

\begin{definition}
\label{def:ch5-tocofiber}
The maps $\Htocofiber{n}$ and $\Htocofiber[r]{n}$ are defined to be the cofibers of $\Htonext{n}$ and $\Htonext[r]{n}$, respectively, shown in the figure below together with the other maps defined so far, as well as another map $\Htowholecofiber[r]{n}$.\par
\begin{figure}[ht]
	\[
	\begin{tikzcd}[row sep=1.6cm,column sep=1.6cm]
	\PH(\flagstyle{W}_n)_{+} \ar[hook]{r}{\Htonext{n}} & \PH(\flagstyle{W}_{n+1})_{+} \ar{r}{\Htocofiber{n}} & \S{\cellstyle{w}_n}\\
	\PH[r](\flagstyle{W}_n)_{+} \ar[hook]{r}[swap]{\Htonext[r]{n}} \ar[hook]{u}{\Htowhole[n]{r}} & \PH[r](\flagstyle{W}_{n+1})_{+} \ar{r}[swap]{\Htocofiber[r]{n}} \ar[hook]{u}{\Htowhole[n+1]{r}} & \Hcofiber[r]{n} \ar[hook]{u}[swap]{\Htowholecofiber[r]{n}}
	\end{tikzcd}
	\]
	\caption{Two important cofiber sequences}
	\label{fig:cofiber}
\end{figure}
By \autoref{thm:cell-structure-h}, the cofiber of the inclusion $\Htonext{n}$ is $\S{\cellstyle{w}_{n}}$. Keeping in mind that $\flagstyle{W}_n(n;\C)=\flagstyle{W}_n(r;\C)$ when $n\equiv \pm r\bmod p$ due to \autoref{thm:class-of-irred-H}, by \autoref{thm:fixed-cell-structure-h} and \autoref{lem:key-even-monotone-lemma} we have that the cofiber of the inclusion $\Htonext[r]{n}$ is
\begin{align*}
\Hcofiber[r]{n}&=\begin{cases}
\S{|\flagstyle{W}_n(n;\C)|} & \text{if }n\equiv \pm r\bmod p,\\
\pt & \text{otherwise}
\end{cases}\\
&=\begin{cases}
\S{4\lfloor n/p\rfloor + 2\lfloor 2(n - \lfloor n/p\rfloor p)/(p+1)\rfloor} & \text{if }n\equiv \pm r\bmod p,\\
\pt & \text{otherwise}
\end{cases}
\end{align*}
where $\cellstyle{w}_{n}$ is as introduced in \autoref{def:PH-even-monotone}. Finally, the map $\Htowholecofiber[r]{n}$ comes from the general fact that an inclusion of pairs $(X,A)\hookrightarrow(Y,B)$ induces an inclusion $B/A\hookrightarrow Y/X$.
\end{definition}

\begin{definition}
\label{def:ch5-totop}
The maps $\Htotop{n}$ and $\Htotop[r]{n}$ are defined to be the inclusions
\[\begin{tikzcd}[row sep=1.6cm,column sep=1.6cm]
\PH(\flagstyle{W}_n)_{+} \ar[hook]{r}{\Htotop{n}} & \PH(Q)_{+} ,
\end{tikzcd}\quad\begin{tikzcd}[row sep=1.6cm,column sep=1.6cm]
\PH[r](\flagstyle{W}_n)_{+} \ar[hook]{r}{\Htotop[r]{n}} & \PH[r](Q)_{+} 
\end{tikzcd}\]  arising from the inclusions of the subrepresentations $\flagstyle{W}_n\hookrightarrow Q$ and their isotypical components, respectively.
\end{definition}




We have shown the commutative diagram of Mackey functors induced by all of these maps in \autoref{fig:cofiber-cohomology}. Note in particular that because the $\flagstyle{W}$-filtration is a properly even $C_p$-cell structure, the boundary maps in cohomology induced by the cofiber sequences in \autoref{fig:cofiber} will be zero, which gives the short exact sequences shown.
\begin{figure}[ht]
\[
\begin{tikzcd}[row sep=1.3cm,column sep=1.3cm]
{} & & \HHreduced_{G}^{\alpha}(\PH(Q)_{+}) \ar{d}[swap]{(\Htotop{n+1})^*} \ar{dr}{(\Htotop{n})^*}\\
0 \ar{r} & \HHreduced_G^{\alpha}(\S{\cellstyle{w}_n}) \ar{d}[swap]{(\Htowholecofiber[r]{n})^*} \ar{r}{(\Htocofiber{n})^*} & \HHreduced_G^{\alpha}(\PH(\flagstyle{W}_{n+1})_{+}) \ar{r}{(\Htonext{n})^*} \ar{d}{(\Htowhole[n+1]{r})^*} &  \HHreduced_G^{\alpha}(\PH(\flagstyle{W}_n)_{+}) \ar{d}{(\Htowhole[n]{r})^*} \ar{r} & 0\\
0\ar{r} & \HHreduced_G^{\alpha}(\Hcofiber[r]{n}) \ar{r}[swap]{(\Htocofiber[r]{n})^*} & \HHreduced_G^{\alpha}(\PH[r](\flagstyle{W}_{n+1})_{+}) \ar{r}[swap]{(\Htonext[r]{n})^*}  & \HHreduced_G^{\alpha}(\PH[r](\flagstyle{W}_n)_{+}) \ar{r} & 0\\
{} & & \HHreduced_{G}^{\alpha}(\PH[r](Q)_{+}) \ar{u}{(\Htotop[r]{n+1})^*} \ar{ur}[swap]{(\Htotop[r]{n})^*}
\end{tikzcd}
\]
\caption{Some important maps in cohomology}
\label{fig:cofiber-cohomology}
\end{figure}

\begin{definition}
Recall the symbols $\GG$ and $\Ge$ from \autoref{rem:dot-and-sun}. Define $\HtoGe$, $\HtoGe[n]$, $\HparttoGe[n]{r}$, and $\HcofibertoGe{n}$ to be the restriction maps in the Mackey functors
\begin{center}
\begin{tikzcd}
\HHreduced_{C_p}^{\alpha}(\PH(Q)_{+})(\GG) \ar[bend right=40]{d}[swap]{\HtoGe} \\
\HHreduced_{C_p}^{\alpha}(\PH(Q)_{+})(\Ge) \ar[bend right=40]{u}
\end{tikzcd}\qquad
\begin{tikzcd}
\HHreduced_{C_p}^{\alpha}(\PH(\flagstyle{W}_n)_{+})(\GG) \ar[bend right=40]{d}[swap,pos=0.48]{\HtoGe[n]} \\
\HHreduced_{C_p}^{\alpha}(\PH(\flagstyle{W}_n)_{+})(\Ge) \ar[bend right=40]{u}
\end{tikzcd}\qquad
\begin{tikzcd}
\HHreduced_{C_p}^{\alpha}(\PH[r](\flagstyle{W}_n)_{+})(\GG) \ar[bend right=40]{d}[swap]{\HparttoGe[n]{r}} \\
\HHreduced_{C_p}^{\alpha}(\PH[r](\flagstyle{W}_n)_{+})(\Ge) \ar[bend right=40]{u}
\end{tikzcd}\qquad
\begin{tikzcd}
\HHreduced_{C_p}^{\alpha}(\S{\cellstyle{w}_n})(\GG) \ar[bend right=40]{d}[swap]{\HcofibertoGe{n}} \\
\HHreduced_{C_p}^{\alpha}(\S{\cellstyle{w}_n})(\Ge) \ar[bend right=40]{u}
\end{tikzcd}
\end{center}
The dependence on $\alpha$ is left implicit. 
\end{definition}


\begin{remark}
For a complex $C_n$-representation $V$ and a quaternionic $C_n$-representation $W$, non-equivariantly we have that
\[\PC(V)\cong\C\P^{\frac{|V|}{2}-1},\qquad \qquad\PH(W)\cong \HP^{\frac{|W|}{4}-1}.\]
Therefore 
\begin{align*}
\PH(Q)&\cong\HP^\infty, & \PH(\flagstyle{W}_n)&\cong\HP^{n-1},\\
\PH[r](Q)&\cong \begin{cases}
\HP^{\infty} & \text{if }2r\equiv 0\bmod n,\\
\C\P^{\infty} & \text{otherwise},
\end{cases} & \PH[r](\flagstyle{W}_n)&\cong \begin{cases}
\HP^{\frac{|\flagstyle{W}_n(r;\C)|}{4}-1} & \text{if }2r\equiv 0\bmod n,\\
\C\P^{\frac{|\flagstyle{W}_n(r;\C)|}{2}-1} & \text{otherwise}.
\end{cases}
\end{align*}
The cohomology rings of complex and quaternionic projective spaces are classically known.
\end{remark}

\begin{definition}
\label{def:generators-x-xr}
We will make a choice of the following generators.
\begin{itemize}
\item Fix an element $x\in \Hhreduced^4(\PH(Q)_{+})$ for which $\Hhreduced^*(\PH(Q)_{+})\cong\Z[x]$.
\item If $2r\equiv 0\bmod p$, fix an element $x_r\in \Hhreduced^4(\PH[r](Q)_{+})$ for which $\Hhreduced^*(\PH[r](Q)_{+})\cong\Z[x_r]$.
\item If $2r\not\equiv0\bmod p$, fix an element $y_r\in\Hhreduced^2(\PH[r](Q)_{+})$ for which $\Hhreduced^*(\PH[r](Q)_{+})\cong\Z[y_r]$.
\end{itemize}
\end{definition}

\begin{remark}
\label{x-stuff}
In light of \autoref{nonequivariant}, we will view $x$ as an element of $\HHreduced_{C_p}^{*}(\PH(Q)_{+})(\Ge)$ when $|{*}|=4$, and view the map
\[(\Htotop{n})^*:\HHreduced_{C_p}^{*}(\PH(Q)_{+})(\Ge)\to\HHreduced_{G}^*(\PH(\flagstyle{W}_n)_{+})(\Ge)\]
as being the quotient map $\Z[x]\to\Z[x]/(x^{n})$.
\end{remark}

\begin{remark}
\label{xr-stuff}
In light of \autoref{thm:trivial-even-cohomology}, we will view $x_r$ as an element of $\HHreduced_{C_p}^{4}(\PH[r](Q)_{+})(\GG)$ when $2r\equiv0\bmod p$, and view the map
\[(\Htotop[r]{n})^*:\HHreduced_{C_p}^{*}(\PH[r](Q)_{+})\to\HHreduced_{G}^*(\PH[r](\flagstyle{W}_n)_{+})\]
as being the quotient map $\HHreduced_{G}^*(\S{0})\otimes\Z[x_r]\to\HHreduced_{G}^*(\S{0})\otimes\Z[x_r]/(x_r^{d})$ where $d$ is the appropriate power for that value of $n$. Similarly for $y_r$ when $2r\not\equiv0\bmod p$, except with $2$ instead of $4$.
\end{remark}

\section{\texorpdfstring{Some specifics for $p=2$}{Some specifics for p=2}}
\label{rem:cells-for-2}

In \autoref{sec:prelim}, we established notation for several important objects related to the cohomology of $\PH(Q)$, as well as maps between them, without making any restriction on the prime $p$. Starting with this section, we focus on the case $p=2$.

Observe that with $p=2$, we have
\[\flagstyle{W}_n=\bigoplus_{k=0}^{n-1}\irredH_k\cong\irredH_0^{\lceil n/2\rceil}\oplus\irredH_1^{\lfloor n/2\rfloor},\]
and therefore
\[\Hcofiberrep{n}=\irredC_{-n}\otimes_{\C}\flagstyle{W}_n\cong\irredH_0^{\lfloor n/2\rfloor}\oplus\irredH_1^{\lceil n/2\rceil}=4\lfloor\tfrac{n}{2}\rfloor+4\lceil\tfrac{n}{2}\rceil\signrep.\]
Consequently, when $n\equiv \pm r\bmod 2$ (indeed, when $n\equiv r\bmod 2$), we have
\[\Hcofiber[r]{n}=\S{|\flagstyle{W}_n(r;\C)|}=\left.\begin{cases}
\S{4\lceil n/2\rceil} & \text{if $n$ is even},\\
\S{4\lfloor n/2\rfloor} & \text{if $n$ is odd}
\end{cases}\right\}=\S{4\lfloor n/2\rfloor},\]
which we could also see by applying the formula from \autoref{def:ch5-tocofiber} with $p=2$. In this case, the map $\Htowholecofiber[r]{n}:\Hcofiber[r]{n}\hookrightarrow\S{\cellstyle{w}_n}$ induces multiplication by $\epsilon^{4\lceil n/2\rceil}$,  where $\epsilon$ is as in \autoref{def:epsilon-for-2}.

Lastly, for later convenience, we will also make the following definition.
\begin{definition}
\label{def:placeholder}
Let $\placeholder{r}=\begin{cases}
0 & \text{if }r\equiv 0\bmod 2,\\
\epsilon^4 & \text{if }r\equiv 1\bmod 2,
\end{cases}$ where $\epsilon$ is as in \autoref{def:epsilon-for-2}.
\end{definition}

\section{Outline of main theorem}

In the subsequent sections, we will prove the following result.

\begin{theorem}
\label{thm:main}
As an algebra over $\HHreduced_{C_2}^{*}(\S{0}\coeffs*{A})$, the cohomology $\HHreduced_{C_2}^{*}(B_{C_2}\SU(2)_{+}\coeffs*{A})$ is generated by two elements $\gen\in\HHreduced_{C_2}^{4\signrep}(B_{C_2}\SU(2)_{+}\coeffs*{A})$ and $\Gen\in\HHreduced_{C_2}^{4+4\signrep}(B_{C_2}\SU(2)_{+}\coeffs*{A})$, satisfying the relation
\[\gen^2=\epsilon^4\gen+\xi^2\Gen,\]
where $\epsilon\in\HHreduced_{C_2}^{\signrep}(\S{0}\coeffs*{A})$ and $\xi\in\HHreduced_{C_2}^{2\signrep-2}(\S{0}\coeffs*{A})$ are from \hyperref[def:epsilon-for-2]{Definitions \begin{NoHyper}\ref{def:epsilon-for-2}\end{NoHyper}} and \hyperref[def:xi-for-2]{\ref{def:xi-for-2}}, respectively.
\end{theorem}

Broadly speaking, the proof will proceed in four steps: 
\begin{itemize}
	\item construct the element $\gen$ (\autoref{subsection:construct-c}),
	\item construct the element $\Gen$ (\autoref{subsection:construct-C}),
	\item prove that together they generate $\HHreduced_{C_2}^{*}(B_{C_2}\SU(2)_{+}\coeffs*{A})$ (\autoref{subsection:proof-cC-generators}), and 
	\item establish the relation between them (\autoref{subsection:cC-relation}).
\end{itemize}

\section{\texorpdfstring{Construction of the generator $\gen$ in dimension $4\signrep$}{Construction of the generator c in dimension 4σ}}
\label{subsection:construct-c}

In this section, we will identify and construct the element $\gen$ in the cohomology Mackey functor $\HHreduced_{C_2}^{4\signrep}(\PH(Q)_{+})(\GG)$, one of the two elements we claim are generators in \autoref{thm:main}.

To identify the element $\gen$, we provide the information shown to be sufficient in \autoref{rem:mult-comparison-useful}, i.e., its images under $\rho$, $(\Htowhole{0})^*$, and $(\Htowhole{1})^*$. The element $\gen\in \HHreduced_{C_2}^{4\signrep}(\PH(Q)_{+})(\GG)$ has 
\begin{align*}
(\Htowhole{0})^*(\gen)&=\placeholder{0}+\xi^2x_0,\\
(\Htowhole{1})^*(\gen)&=\placeholder{1}+\xi^2x_1,\\
\HtoGe(\gen)&=x.
\end{align*}

To construct the element $\gen$, i.e., to prove that an element as specified above really does exist, our goal will be to inductively construct elements $\gen_n\in \HHreduced_{G}^{4\signrep}(\PH(\flagstyle{W}_n)_{+})(\GG)$, each such that
\begin{align*}
(\Htowhole[n]{0})^*(\gen_n)&=(\Htotop[0]{n})^*(\placeholder{0}+\xi^2x_0),\\
(\Htowhole[n]{1})^*(\gen_n)&=(\Htotop[1]{n})^*(\placeholder{1}+\xi^2x_1),\\
\HtoGe[n](\gen_n)&=(\Htotop{n})^*(x).
\end{align*}
Since the $\flagstyle{W}_n$ are a cofinal collection of finite-dimensional sub-$\H$-modules of $Q$, and since
\[(\Htotop[r]{n})^*\circ (\Htowhole{r})^*=(\Htowhole[n]{r})^*\circ (\Htotop{n})^*,\]
by \autoref{thm:inverse-limit} this process yields an element $\gen\in \HHreduced_{C_2}^{4\signrep}(\PH(Q)_{+})(\GG)$ with 
\begin{align*}
(\Htowhole{0})^*(\gen)&=\placeholder{0}+\xi^2x_0,\\
(\Htowhole{1})^*(\gen)&=\placeholder{1}+\xi^2x_1,\\
\HtoGe(\gen)&=x
\end{align*}
as desired.

First, \autoref{lem:gen-initial} will provide a base case by constructing an element $\gen_2$ with the specified properties, then \autoref{thm:gen} will perform the main work of proving the inductive step, that an element $\gen_n$ can always be lifted to an element $\gen_{n+1}$.

\begin{lemma}
\label{lem:gen-initial}
An element $\gen_2\in \HHreduced_{C_2}^{4\signrep}(\PH(\flagstyle{W}_2)_{+})(\GG)$ exists with the following properties:
\begin{align*}
(\Htowhole[2]{0})^*(\gen_2)&=(\Htotop[0]{2})^*(\placeholder{0}+\xi^2x_0),\\
(\Htowhole[2]{1})^*(\gen_2)&=(\Htotop[1]{2})^*(\placeholder{1}+\xi^2x_1),\\
\HtoGe[2](\gen_2)&=(\Htotop{2})^*(x).
\end{align*}
\end{lemma}
\begin{proof}
The proof may be broken into sections.

\subsubsection{Define the element $\gen_2$}


First, note that $\S{\cellstyle{w}_1}=\S{4\signrep}$.

Let $\gen_2\in \HHreduced_{G}^{4\signrep}(\PH(\flagstyle{W}_2)_{+})(\GG)$ be the image of $1\in A(\GG)\cong \HHreduced_{G}^{4\signrep}(\S{4\signrep})(\GG)$ under $(\Htocofiber{1})^*$.

\subsubsection{Check the value of $(\Htowhole[2]{0})^*(\gen_2)$}

We have that $\Hcofiber[0]{1}=\pt$, so within \autoref{fig:cofiber-cohomology} we find the commutative square of Mackey functors
\begin{center}
\begin{tikzcd}[row sep=1.5cm,column sep=1.5cm]
\HHreduced_{C_2}^{4\signrep}(\S{4\signrep})
\ar{d}[swap]{(\Htowholecofiber[0]{1})^*}
\ar{r}{(\Htocofiber{1})^*}
&
\HHreduced_{C_2}^{4\signrep}(\PH(\flagstyle{W}_2)_{+})
\ar{d}{(\Htowhole[0]{2})^*}\\
\HHreduced_{C_2}^{4\signrep}(\pt)
\ar{r}[swap]{(\Htocofiber[0]{1})^*}
&
\HHreduced_{C_2}^{4\signrep}(\PH[0](\flagstyle{W}_2)_{+})
\end{tikzcd}
\end{center}
Since $\HHreduced_{C_2}^{4\signrep}(\pt)=0$, we have that $(\Htowhole[2]{0})^*(\gen_2)=0$.

Note that $(\Htotop[0]{2})^*(\placeholder{0}+\xi^2 x_0)=(\Htotop[0]{2})^*(\xi^2 x_0)=0$ by \autoref{xr-stuff}.


\subsubsection{Check the value of $(\Htowhole[2]{1})^*(\gen_2)$}

We have that $\Hcofiber[1]{1}=\S{0}$, so within \autoref{fig:cofiber-cohomology} we find the commutative square of Mackey functors
\begin{center}
\begin{tikzcd}[row sep=1.5cm,column sep=1.5cm]
\HHreduced_{C_2}^{4\signrep}(\S{4\signrep})
\ar{d}[swap]{(\Htowholecofiber[1]{1})^*}
\ar{r}{(\Htocofiber{1})^*}
&
\HHreduced_{C_2}^{4\signrep}(\PH(\flagstyle{W}_2)_{+})
\ar{d}{(\Htowhole[1]{2})^*}\\
\HHreduced_{C_2}^{4\signrep}(\S{0})
\ar{r}[swap]{(\Htocofiber[1]{1})^*}
&
\HHreduced_{C_2}^{4\signrep}(\PH[1](\flagstyle{W}_2)_{+})
\end{tikzcd}
\end{center}

Since $(\Htocofiber[1]{1})^*$ is an isomorphism of $\HHreduced_{C_2}^*(\S{0})$-algebras and $\Htowholecofiber[1]{1}$ is the inclusion $\S{0}\to\S{4\signrep}$, we have that $(\Htowhole[2]{1})^*(\gen_2)=\epsilon^4$ (see \autoref{def:epsilon-for-2}).

Note that $(\Htotop[1]{2})^*(\placeholder{1}+\xi^2x_1)=(\Htotop[1]{2})^*(\epsilon^4+\xi^2x_1)=\epsilon^4$ by \autoref{xr-stuff}.

%
%

\subsubsection{Check the value of $\rho_2(\gen_2)$}

Within \autoref{fig:cofiber-cohomology} we find the maps $(\Htocofiber{1})^*$ and $(\Htotop{2})^*$, which we expand into their $\GG$ and $\Ge$ levels.

\begin{center}
\begin{tikzcd}[row sep=0.2cm,column sep=1.7cm]
{}&\HHreduced_{C_2}^{4\signrep}(\PH(Q)_{+})(\GG) \ar[bend right=40]{dd}[swap]{\rho}\\
{}\\
{}&\HHreduced_{C_2}^{4\signrep}(\PH(Q)_{+})(\Ge) \ar[bend right=40]{uu} \ar[start anchor = {[yshift=-0.4cm]}, end anchor = {[yshift=0.4cm]}]{ddd}{(\Htotop{2})^*} \\
\strut & \\
\strut & \\
\HHreduced_{C_2}^{4\signrep}(\S{4\signrep})(\GG) \ar[bend right=40]{dd}[swap]{\HcofibertoGe{1}}& \HHreduced_{C_2}^{4\signrep}(\PH(\flagstyle{W}_2)_{+})(\GG) \ar[bend right=40]{dd}[swap]{\HtoGe[2]}\\
{} \ar[start anchor = {[xshift=1.2cm]}, end anchor = {[xshift=-1.7cm]}]{r}[swap]{(\Htocofiber{1})^*} & {} \\
\HHreduced_{C_2}^{4\signrep}(\S{4\signrep})(\Ge) \ar[bend right=40]{uu} & \HHreduced_{C_2}^{4\signrep}(\PH(\flagstyle{W}_2)_{+})(\Ge) \ar[bend right=40]{uu}
\end{tikzcd}
\end{center}

Non-equivariantly, we have \[\PH(Q)_{+}\cong\HP^\infty_{+},\qquad \PH(\flagstyle{W}_2)_{+}\cong\HP^1_{+}\cong\S{4}_{+},\qquad \S{4\signrep}\cong\S{4}\]
and $\Htocofiber{1}$ is just the map $\S{4}_{+}\to\S{4}$ that is the identity on $\S{4}$ while gluing the disjoint basepoint somewhere. 
By \autoref{nonequivariant}, the $\Ge$ level of this diagram just reflects what happens in non-equivariant cohomology:
\begin{center}
\begin{tikzcd}[row sep=1.7cm,column sep=1.7cm]
{} & \Hhreduced^{4}(\HP^\infty) \ar{d}{(\Htotop{2})^*}\\
\Hhreduced^{4}(\S{4}) \ar{r}[swap]{(\Htocofiber{1})^*} & \Hhreduced^{4}(\HP^1_{+})
\end{tikzcd}
\end{center}
namely, that $(\Htocofiber{1})^*$ sends the generator $1\in\Hhreduced^4(\S{4})$ to the generator $(\Htotop{2})^*(x)\in\Hhreduced^{4}(\HP^1_{+})$. 

Because $\HHreduced_{C_2}^{4\signrep}(\S{4\signrep})\cong \HHreduced_{C_2}^0(\S{0})\cong A$, we can see that $\HcofibertoGe{1}$ sends the element $1\in\HHreduced_{C_2}^{4\signrep}(\S{4\signrep})(\GG)$ to the element $1\in\HHreduced_{C_2}^{4\signrep}(\S{4\signrep})(\Ge)\cong\Hhreduced^4(\S{4})$ (see \autoref{def:important-mackeys}), so that
\[\HtoGe[2](\gen_2)=\HtoGe[2]((\Htocofiber{1})^*(1))=(\Htocofiber{1})^*(\HcofibertoGe{1})(1)=(\Htotop{2})^*(x).\qedhere\]
\end{proof}

Now we prove the inductive step, that an element $\gen_n$ with the desired properties can always be lifted to an element $\gen_{n+1}$ with the desired properties.

\begin{theorem}
\label{thm:gen}
If $\gen_n\in \HHreduced_{C_2}^{4\signrep}(\PH(\flagstyle{W}_n)_{+})(\GG)$ exists, then  $\gen_{n+1}\in\HHreduced_{C_2}^{4\signrep}(\PH(\flagstyle{W}_{n+1})_{+})(\GG)$ exists.
\end{theorem}

\begin{proof}
Suppose that a class $\gen_n\in \HHreduced_{C_2}^{4\signrep}(\PH(\flagstyle{W}_n)_{+})(\GG)$ has been defined with 
\begin{align*}
(\Htowhole[n]{0})^*(\gen_n)&=(\Htotop[0]{n})^*(\placeholder{0}+\xi^2x_0),\\
(\Htowhole[n]{1})^*(\gen_n)&=(\Htotop[1]{n})^*(\placeholder{1}+\xi^2x_1),\\
\HtoGe[n](\gen_n)&=(\Htotop{n})^*(x),
\end{align*}
for some $n\geq 2$.

%

\subsubsection{Existence and uniqueness of an element $\gen_{n+1}$}

We will lift the class $\gen_n\in  \HHreduced_{C_2}^{4\signrep}(\PH(\flagstyle{W}_n)_{+})(\GG)$ along the map $(\Htonext{n})^*$, which fits into an exact sequence of abelian groups
\begin{center}
\begin{tikzcd}[column sep=1cm]
0\ar{r} & \HHreduced_{C_2}^{4\signrep}(\S{\cellstyle{w}_n})(\GG)\ar{r}{(\Htocofiber{n})^*} & \HHreduced_{C_2}^{4\signrep}(\PH(\flagstyle{W}_{n+1})_{+})(\GG)\ar{r}{(\Htonext{n})^*} &\HHreduced_{C_2}^{4\signrep}(\PH(\flagstyle{W}_n)_{+})(\GG)\ar{r} & 0
\end{tikzcd}
\end{center}
From \autoref{rem:cells-for-2}, we have $|\cellstyle{w}_n|=4n\geq 8$ and $|\cellstyle{w}_n^{C_2}|=4\lfloor\frac{n}{2}\rfloor\geq 4$  because $n\geq 2$, and therefore
\[|4\signrep-\cellstyle{w}_n|\leq 4-8=-4,\qquad |(4\signrep-\cellstyle{w}_n)^{C_2}|\leq 0-4=-4\]
so that by considering \autoref{figure-HS0-2},
\[\HHreduced_{C_2}^{4\signrep}(\S{\cellstyle{w}_n})(\GG)\cong \HHreduced_{C_2}^{4\signrep-\cellstyle{w}_n}(\S{0})(\GG)=0.\]
Therefore $(\Htonext{n})^*$ is an isomorphism in dimension $4\signrep$, and there is exactly one lifting of $\gen_n$ along the map $(\Htonext{n})^*$, which we will define to be our element $\gen_{n+1}$.

\subsubsection{Check the value of $\HtoGe[n+1](\gen_{n+1})$}

For clarity, we expand the maps $(\Htonext{n})^*$, $(\Htotop{n})^*$, and $(\Htotop{n+1})^*$ into their $\GG$ and $\Ge$ levels.

\begin{center}
\begin{tikzcd}[row sep=0.2cm,column sep=1.7cm]
\HHreduced_{C_2}^{4\signrep}(\PH(Q)_{+})(\GG) \ar[bend right=40]{dd}[swap]{\HtoGe}\\
{}\\
\HHreduced_{C_2}^{4\signrep}(\PH(Q)_{+})(\Ge) \ar[bend right=40]{uu} \ar[start anchor = {[yshift=-0.4cm]}, end anchor = {[yshift=0.4cm]}]{ddd}[swap]{(\Htotop{n+1})^*} \ar[start anchor = {[xshift=1.2cm,yshift=-0.0cm]}, end anchor = {[xshift=-0.4cm,yshift=0.4cm]}]{dddr}{(\Htotop{n})^*} \\
& \strut \\
& \strut \\
 \HHreduced_{C_2}^{4\signrep}(\PH(\flagstyle{W}_{n+1})_{+})(\GG) \ar[bend right=40]{dd}[swap]{\HtoGe[n+1]} & \HHreduced_{C_2}^{4\signrep}(\PH(\flagstyle{W}_{n})(\GG) \ar[bend right=40]{dd}[swap]{\HtoGe[n]}\\
{} \ar[start anchor = {[xshift=1.5cm]}, end anchor = {[xshift=-1.3cm]}]{r}[swap]{(\Htonext{n})^*} & {} \\ \HHreduced_{C_2}^{4\signrep}(\PH(\flagstyle{W}_{n+1})_{+})(\Ge) \ar[bend right=40]{uu} & 
\HHreduced_{C_2}^{4\signrep}(\PH(\flagstyle{W}_{n})(\Ge) \ar[bend right=40]{uu} 
\end{tikzcd}
\end{center}

Non-equivariantly, we have \[\PH(Q)_{+}\cong\HP^\infty_{+},\qquad \PH(\flagstyle{W}_{n+1})_{+}\cong\HP^{n}_{+},\qquad \PH(\flagstyle{W}_{n})_{+}\cong\HP^{n-1}_{+}\]
and $\Htonext{n}$ is the inclusion map $\HP^{n-1}_{+}\to\HP^{n}_{+}$ that adds a $4n$-cell to $\HP^{n-1}_{+}$. By \autoref{nonequivariant}, the $\Ge$ level of this diagram reflects what happens in non-equivariant cohomology:
\begin{center}
\begin{tikzcd}[row sep=1.7cm,column sep=1.7cm]
\Hhreduced^{4}(\HP^\infty_{+}) \ar{d}[swap]{(\Htotop{n+1})^*} \ar{dr}{(\Htotop{n})^*}\\
\Hhreduced^{4}(\HP^{n}_{+}) \ar{r}[swap]{(\Htonext{n})^*}{\cong}  &  \Hhreduced^{4}(\HP^{n-1}_{+})
\end{tikzcd}
\end{center}
namely, that the map $(\Htonext{n})^*$ is an isomorphism in dimension $4$, and sends the generator $(\Htotop{n+1})^*(x)\in\Hhreduced^{4}(\HP^{n}_{+})$ to the generator $(\Htotop{n})^*(x)\in\Hhreduced^{4}(\HP^{n-1}_{+})$.

By the induction hypothesis we know that $\HtoGe[n](\gen_{n})=(\Htotop{n})^*(x)$, and because $\gen_{n+1}$ is a lift
of $\gen_n$ along $(\Htonext{n})^*$, we can conclude $\HtoGe[n+1](\gen_{n+1})=(\Htotop{n+1})^*(x)$, as desired.

\subsubsection{Check the value of $(\Htowhole[n+1]{r})^*(\gen_{n+1})$ for $n\not\equiv r\bmod 2$}

Since $n\not\equiv r\bmod 2$, we have that $\Hcofiber[r]{n}=\pt$, so that $\HHreduced_{C_2}^{4\signrep}(\Hcofiber[r]{n})=0$, and hence $(\Htonext[r]{n})^*$ is an isomorphism in dimension $4\signrep$. 
Now the relevant piece of \autoref{fig:cofiber-cohomology} is
\begin{center}
\begin{tikzcd}[row sep=1.3cm,column sep=1.3cm]
\HHreduced_{C_2}^{4\signrep}(\PH(\flagstyle{W}_{n+1})_{+}) \ar{r}{(\Htonext{n})^*}[swap]{} \ar{d}[swap]{(\Htowhole[n+1]{r})^*} &  \HHreduced_{C_2}^{4\signrep}(\PH(\flagstyle{W}_n)_{+}) \ar{d}{(\Htowhole[n]{r})^*} 
\\
\HHreduced_{C_2}^{4\signrep}(\PH[r](\flagstyle{W}_{n+1})_{+}) \ar{r}[swap]{(\Htonext[r]{n})^*}{\cong}  & \HHreduced_{C_2}^{4\signrep}(\PH[r](\flagstyle{W}_n)_{+})
\\
\HHreduced_{C_2}^{4\signrep}(\PH[r](Q)_{+}) \ar{u}{(\Htotop[r]{n+1})^*} \ar{ru}[swap]{(\Htotop[r]{n})^*}
\end{tikzcd}
\end{center}

By our inductive hypothesis, we know that $(\Htowhole[n]{r})^*(\gen_n)=(\Htotop[r]{n})^*(\placeholder{r}+\xi^2x_r)$, so we can conclude
\begin{align*}
(\Htonext[r]{n})^*\circ(\Htowhole[n+1]{r})^*(\gen_{n+1})&=(\Htowhole[n]{r})^*\circ(\Htonext{n})^*(\gen_{n+1})\\
(\Htonext[r]{n})^*\circ(\Htowhole[n+1]{r})^*(\gen_{n+1})&=(\Htowhole[n]{r})^*(\gen_{n})\\
(\Htonext[r]{n})^*\circ(\Htowhole[n+1]{r})^*(\gen_{n+1})&=(\Htotop[r]{n})^*(\placeholder{r}+\xi^2x_r)\\
(\Htowhole[n+1]{r})^*(\gen_{n+1})&=((\Htonext[r]{n})^*)^{-1}\circ(\Htotop[r]{n})^*(\placeholder{r}+\xi^2x_r)\\
(\Htowhole[n+1]{r})^*(\gen_{n+1})&=(\Htotop[r]{n+1})^*(\placeholder{r}+\xi^2x_r)
\end{align*}
which is what is desired.

\subsubsection{Check the value of $(\Htowhole[n+1]{r})^*(\gen_{n+1})$ for $n\equiv r\bmod 2$}

Since $n\equiv r\bmod 2$, we have that $\Hcofiber[r]{n}=\S{4\lfloor n/2\rfloor}$ by \autoref{rem:cells-for-2}. 
Because $n\geq 2$, we have $4\lfloor \frac{n}{2}\rfloor\geq 4$.

We now consider this section of \autoref{fig:cofiber-cohomology}.
\begin{center}
\begin{tikzcd}[row sep=1.3cm,column sep=1.3cm]
0\ar{r} & \HHreduced_{C_2}^{4\signrep}(\S{4\lfloor n/2\rfloor}) \ar{r}{(\Htocofiber[r]{n})^*} & \HHreduced_{C_2}^{4\signrep}(\PH[r](\flagstyle{W}_{n+1})_{+}) \ar{r}{(\Htonext[r]{n})^*}  & \HHreduced_{C_2}^{4\signrep}(\PH[r](\flagstyle{W}_n)_{+}) \ar{r} & 0\\
{} & & \HHreduced_{C_2}^{4\signrep}(\PH[r](Q)_{+}) \ar{u}{(\Htotop[r]{n+1})^*} \ar{ru}[swap]{(\Htotop[r]{n})^*}
\end{tikzcd}
\end{center}

If in fact $4\lfloor \frac{n}{2}\rfloor>4$, then it must be at least $8$, hence
\[\bigr|4\signrep-4\lfloor \tfrac{n}{2}\rfloor\bigr|\leq 4-8=-4,\qquad \bigl|(4\signrep-4\lfloor \tfrac{n}{2}\rfloor)^{C_2}\bigr|\leq 0-8=-8.\]
By considering \autoref{figure-HS0-2}, we see that
\[\HHreduced_{C_2}^{4\signrep}(\S{4\lfloor \frac{n}{2}\rfloor})\cong \HHreduced_{C_2}^{4\signrep-4\lfloor \frac{n}{2}\rfloor}(\S{0})=0,\]
so that $(\Htonext[r]{n})^*$ is an isomorphism, which implies $(\Htowhole[n+1]{r})^*(\gen_{n+1})$ is the desired value of $(\Htotop[r]{n+1})^*(\placeholder{r}+\xi^2x_r)$ by the same argument as in the previous section.

If instead $4\lfloor \frac{n}{2}\rfloor=4$, then it must be the case that $n=2$ or $3$, so that -- \emph{equivariantly}, since these spaces have a trivial $G$-action -- we have
\[\Hcofiber[r]{n}\cong\S{4},\qquad \PH[r](\flagstyle{W}_{n+1})_{+}\cong\HP^1_{+},\qquad \PH[r](\flagstyle{W}_n)_{+}\cong\HP^0_{+}\]
and hence our diagram is
\begin{center}
\begin{tikzcd}[row sep=1.3cm,column sep=1.3cm]
0\ar{r} & \HHreduced_{C_2}^{4\signrep}(\S{4}) \ar{r}{(\Htocofiber[r]{n})^*} & \HHreduced_G^{4\signrep}(\HP^1_{+}) \ar{r}{(\Htonext[r]{n})^*}  & \HHreduced_{C_2}^{4\signrep}(\HP^0_{+}) \ar{r} & 0\\
{} & & \HHreduced_{C_2}^{4\signrep}(\HP^\infty_{+}) \ar{u}{(\Htotop[r]{n+1})^*} \ar{ru}[swap]{(\Htotop[r]{n})^*}
\end{tikzcd}
\end{center}
Note that $|4\signrep-4|=0$ and $|(4\signrep-4)^{C_2}|=-4$, 
so by observing \autoref{figure-HS0-2}, we see that \[\HHreduced_{C_2}^{4\signrep}(\S{4})\cong \HHreduced_{C_2}^{4\signrep-4}(\S{0})\cong R.\]
By \autoref{def:xi-for-2}, we see that $\xi^2$ is the generator of $R(\GG)\cong\Z$. In non-equivariant cohomology, the map $(\Htocofiber[r]{n})^*:\Hhreduced^4(\S{4})\to\Hhreduced^4(\HP^1_{+})$ sends the generator $1$ to the generator $(\Htotop[r]{n+1})^*(x_r)$, so by \autoref{xr-stuff}, we conclude that the map 
\[(\Htocofiber[r]{n})^*\from\HHreduced_{C_2}^{4\signrep}(\S{4})\cong\HHreduced_{C_2}^{4\signrep-4}(\S{0})\otimes \Hhreduced^4(\S{4})\to\HHreduced_{C_2}^{4\signrep-4}(\S{0})\otimes\Hhreduced^4(\HP^1_{+})\subseteq \HHreduced_{C_2}^{4\signrep}(\HP^1_{+})\]
sends the generator $\xi^2\in\HHreduced_{C_2}^{4\signrep}(\S{4})(\GG)\cong R(\GG)$ to the element $(\Htotop[r]{n+1})^*(\xi^2x_r)\in\HHreduced_{C_2}^{4\signrep}(\HP^1_{+})(\GG)$.

By our induction hypothesis,
\[((\Htonext[r]{n})^*\circ(\Htowhole[n+1]{r})^*)(c_{n+1})=(\Htowhole[n]{r})^*(c_{n})=(\Htotop[r]{n})^*(\placeholder{r}+\xi^2x_r)\]
and therefore $(\Htowhole[n+1]{r})^*(c_{n+1})$ is a lift of $(\Htotop[r]{n})^*(\placeholder{r}+\xi^2x_r)$ along the map $(\Htonext[r]{n})^*$. Of course, the lift we would like it to be is precisely $(\Htotop[r]{n+1})^*(\placeholder{r}+\xi^2x_r)$.

Because $\PH[r](\flagstyle{W}_n)_{+}\cong\HP^0_{+}$, we have that $(\Htotop[r]{n})^*(x_r)=0$ (see \autoref{xr-stuff}). Thus $(\Htotop[r]{n})^*(\placeholder{r}+\xi^2x_r)=\placeholder{r}$, and therefore, \emph{a priori}, the value of a lift of $(\Htotop[r]{n})^*(\placeholder{r}+\xi^2x_r)$ along the map $(\Htonext[r]{n})^*$ could be $(\Htotop[r]{n+1})^*(\placeholder{r}+k\xi^2x_r)$ for any $k\in\Z$, since the image of $(\Htocofiber[r]{n})^*$ is the kernel of $(\Htonext[r]{n})^*$.

Consider the map $(\Htowhole[n+1]{r})^*$ in \autoref{fig:cofiber-cohomology}, expanded into its $\GG$ and $\Ge$ components.
\begin{center}
\begin{tikzcd}[row sep=0.2cm,column sep=1.7cm]
\HHreduced_{C_2}^{4\signrep}(\PH(\flagstyle{W}_{n+1})_{+})(\GG) \ar[bend right=40]{dd}[swap]{\HtoGe[n+1]}\\
{}\\
\HHreduced_{C_2}^{4\signrep}(\PH(\flagstyle{W}_{n+1})_{+})(\Ge) \ar[bend right=40]{uu} \ar[start anchor = {[yshift=-0.4cm]}, end anchor = {[yshift=0.4cm]}]{ddd}[swap]{(\Htowhole[n+1]{r})^*}\\
\strut\\
\strut\\
 \HHreduced_{C_2}^{4\signrep}(\PH[r](\flagstyle{W}_{n+1})_{+})(\GG) \ar[bend right=40]{dd}[swap]{\HparttoGe[n+1]{r}} \\
{} \\ \HHreduced_{C_2}^{4\signrep}(\PH[r](\flagstyle{W}_{n+1})_{+})(\Ge) \ar[bend right=40]{uu}
\end{tikzcd}
\end{center}
Again by \autoref{nonequivariant}, the $\Ge$ level of this diagram just reflects what happens in non-equivariant cohomology:
\begin{center}
\begin{tikzcd}[row sep=1.7cm,column sep=1.7cm]
\Hhreduced^{4}(\HP^n_{+}) \ar{d}[swap]{(\Htowhole[n+1]{r})^*}\\
\Hhreduced^{4}(\HP^{1}_{+})
\end{tikzcd}
\end{center}
namely, that the map $(\Htowhole[n+1]{r})^*$ is an isomorphism in dimension $4$, and sends the generator $(\Htotop{n+1})^*(x)\in\Hhreduced^{4}(\HP^{n}_{+})\cong\Z$ to the generator $(\HparttoGe[n+1]{r})^*((\Htotop[r]{n+1})^*(\xi^2x_r))\in\Hhreduced^{4}(\HP^{1}_{+})\cong\Z$.

However, we have already established that $(\HtoGe[n+1])^*(\gen_{n+1})=(\Htotop{n+1})^*(x)$, and therefore \begin{align*}
(\HparttoGe[n+1]{r})^*((\Htowhole[n+1]{r})^*(\gen_{n+1}))&=(\Htowhole[n+1]{r})^*((\HtoGe[n+1])^*(\gen_{n+1}))\\
&=(\Htowhole[n+1]{r})^*((\Htotop{n+1})^*(x))\\
&=(\HparttoGe[n+1]{r})^*((\Htotop[r]{n+1})^*(\xi^2x_r))
\end{align*}
Because $\HparttoGe[n+1]{r}$ is an isomorphism (see the definition of $R$ in \autoref{def:important-mackeys}), the value of $(\Htowhole[n+1]{r})^*(\gen_{n+1})$, which we showed must be $(\Htotop[r]{n+1})^*(\placeholder{r}+k\xi^2x_r)$ for some $k\in\Z$, must in fact be $(\Htotop[r]{n+1})^*(\placeholder{r}+\xi^2x_r)$, as desired.
\end{proof}

\section{\texorpdfstring{Construction of the generator $\Gen$ in dimension $4+4\signrep$}{Construction of the generator C in dimension 4+4σ}}
\label{subsection:construct-C}

In this section, we will identify and construct the element $\Gen$ in the cohomology Mackey functor $\HHreduced_{C_2}^{4+4\signrep}(\PH(Q)_{+})(\GG)$, the other of the two elements we claim are generators in \autoref{thm:main}.

To identify the element $\Gen$, we provide the information shown to be sufficient in \autoref{rem:mult-comparison-useful}, i.e., its images under $\rho$, $(\Htowhole{0})^*$, and $(\Htowhole{1})^*$. The element $\Gen\in \HHreduced_{C_2}^{4+4\signrep}(\PH(Q)_{+})(\GG)$ has 
\begin{align*}
(\Htowhole{0})^*(\Gen)&=x_0(\epsilon^4+\xi^2x_0),\\
(\Htowhole{1})^*(\Gen)&=x_1(\epsilon^4+\xi^2x_1),\\
\HtoGe(\Gen)&=x^2.
\end{align*}

To construct the element $\Gen$, i.e., to prove that an element as specified above really does exist, our goal will be to inductively construct elements $\Gen_n\in \HHreduced_{G}^{4+4\signrep}(\PH(\flagstyle{W}_n)_{+})(\GG)$, each such that
\begin{align*}
(\Htowhole[n]{0})^*(\Gen_n)&=(\Htotop[0]{n})^*(x_0(\epsilon^4+\xi^2x_0)),\\
(\Htowhole[n]{1})^*(\Gen_n)&=(\Htotop[1]{n})^*(x_1(\epsilon^4+\xi^2x_1)),\\
\HtoGe[n](\Gen_n)&=(\Htotop{n})^*(x^2).
\end{align*}
Since the $\flagstyle{W}_n$ are a cofinal collection of finite-dimensional sub-$\H$-modules of $Q$, and since
\[(\Htotop[r]{n})^*\circ (\Htowhole{r})^*=(\Htowhole[n]{r})^*\circ (\Htotop{n})^*,\]
by \autoref{thm:inverse-limit} this process yields an element $\Gen\in \HHreduced_{C_2}^{4+4\signrep}(\PH(Q)_{+})(\GG)$ with 
\begin{align*}
(\Htowhole{0})^*(\Gen)&=x_0(\epsilon^4+\xi^2x_0),\\
(\Htowhole{1})^*(\Gen)&=x_1(\epsilon^4+\xi^2x_1),\\
\HtoGe(\Gen)&=x^2
\end{align*}
as desired.

First, \autoref{lem:Gen-initial} will provide a base case by constructing an element $\Gen_3$ with the specified properties, then \autoref{thm:Gen} will perform the main work of proving the inductive step, that an element $\Gen_n$ can always be lifted to an element $\Gen_{n+1}$.

\begin{lemma}
\label{lem:Gen-initial}
An element $\Gen_3\in \HHreduced_{C_2}^{4+4\signrep}(\PH(\flagstyle{W}_3)_{+})$ exists with the following properties:
\begin{align*}
(\Htowhole[3]{0})^*(\Gen_3)&=(\Htotop[0]{3})^*(x_0(\epsilon^4+\xi^2 x_0)),\\
(\Htowhole[3]{1})^*(\Gen_3)&=(\Htotop[1]{3})^*(x_1(\epsilon^4+\xi^2 x_1)),\\
\HtoGe[3](\Gen_3)&=(\Htotop{3})^*(x^2).
\end{align*}
\end{lemma}

\begin{proof}

The proof may be broken into sections.

\subsubsection{\texorpdfstring{Define the element $\Gen_3$}{Define the element C\_3}}


First, note that $\S{\cellstyle{w}_2}=\S{4+4\signrep}$.

Let $\Gen_3\in \HHreduced_{C_2}^{4+4\signrep}(\PH(\flagstyle{W}_3)_{+})(\GG)$ be the image of $1\in A(\GG)\cong \HHreduced_{C_2}^{4+4\signrep}(\S{4+4\signrep})(\GG)$ under $(\Htocofiber{2})^*$. 

\subsubsection{Check the value of $\Htowhole[3]{0}(\Gen_3)$}

We have that $\Hcofiber[0]{2}=\S{|\flagstyle{W}_2(0;\C)|}=\S{4}$, so within \autoref{fig:cofiber-cohomology} we find the commutative square of Mackey functors
\begin{center}
\begin{tikzcd}[row sep=1.5cm,column sep=1.5cm]
\HHreduced_{C_2}^{4+4\signrep}(\S{4+4\signrep})
\ar{d}[swap]{(\Htowholecofiber[0]{2})^*}
\ar{r}{(\Htocofiber{2})^*}
&
\HHreduced_{C_2}^{4+4\signrep}(\PH(\flagstyle{W}_3)_{+})
\ar{d}{(\Htowhole[3]{0})^*}\\
\HHreduced_{C_2}^{4+4\signrep}(\S{4})
\ar{r}[swap]{(\Htocofiber[0]{2})^*}
&
\HHreduced_{C_2}^{4+4\signrep}(\PH[0](\flagstyle{W}_3)_{+})
\end{tikzcd}
\end{center}

In non-equivariant cohomology, the map $(\Htocofiber[0]{2})^*:\Hhreduced^4(\S{4})\to\Hhreduced^4(\HP^1_{+})$ sends the generator $1$ to the generator $(\Htotop[0]{3})^*(x_0)$, so by \autoref{xr-stuff}, we conclude that the map 
\[(\Htocofiber[0]{2})^*:\HHreduced_{C_2}^{4+4\signrep}(\S{4})\cong\HH_{C_2}^{4\signrep}(\S{0})\otimes \Hhreduced^4(\S{4})\to\HH_{C_2}^{4\signrep}(\S{0})\otimes\Hhreduced^4(\HP^1_{+})\subseteq \HHreduced_{C_2}^{4+4\signrep}(\HP^1_{+})\]
sends the generator
$\epsilon^4\in\HHreduced_{C_2}^{4+4\signrep}(\S{4})\cong \langle\Z\rangle(\GG)$ to the element $(\Htotop[0]{3})^*(\epsilon^4x_0)\in\HHreduced_{C_2}^{4+4\signrep}(\PH[0](\flagstyle{W}_3)_{+})(\GG)$.

The map $(\Htowholecofiber[0]{2})^*$ sends $1\in\HHreduced_{C_2}^{4+4\signrep}(\S{4+4\signrep})$ to $\epsilon^4\in\HHreduced_{C_2}^{4+4\signrep}(\S{4})$, because the composition 
\[\HHreduced_{C_2}^{0}(\S{0})\cong\HHreduced_{C_2}^{4+4\signrep}(\S{4+4\signrep})\xrightarrow{(\Htowholecofiber[0]{2})^*}\HHreduced_{C_2}^{4+4\signrep}(\S{4})\cong\HHreduced_{C_2}^{4\signrep}(\S{0})\]
is the same as the map in cohomology induced by $\epsilon^4\from\S{0}\to\S{4\signrep}$ (see \autoref{def:epsilon-for-2}).

Thus, we have that $(\Htowhole[3]{0})^*(\Gen_3)=(\Htotop[0]{3})^*(\epsilon^4x_0)$.

Note that $(\Htotop[0]{3})^*(x_0(\epsilon^4+\xi^2x_0))=(\Htotop[0]{3})^*(\epsilon^4x_0)$ by \autoref{xr-stuff}, since $\PH[0](\flagstyle{W}_3)_{+}\cong\HP^1_{+}$.

\subsubsection{Check the value of $\Htowhole[3]{1}(\Gen_3)$}

We have that $\Hcofiber[1]{2}=\pt$, so within \autoref{fig:cofiber-cohomology} we find the commutative square of Mackey functors
\begin{center}
\begin{tikzcd}[row sep=1.5cm,column sep=1.5cm]
\HHreduced_{C_2}^{4+4\signrep}(\S{4+4\signrep})
\ar{d}[swap]{(\Htowholecofiber[1]{2})^*}
\ar{r}{(\Htocofiber{2})^*}
&
\HHreduced_{C_2}^{4+4\signrep}(\PH(\flagstyle{W}_3)_{+})
\ar{d}{(\Htowhole[1]{3})^*}\\
\HHreduced_{C_2}^{4+4\signrep}(\pt)
\ar{r}[swap]{(\Htocofiber[1]{2})^*}
&
\HHreduced_{C_2}^{4+4\signrep}(\PH[1](\flagstyle{W}_3)_{+})
\end{tikzcd}
\end{center}

Since $\HHreduced_{C_2}^{4+4\signrep}(\pt)=0$, we have that $(\Htowhole[3]{1})^*(\Gen_3)=0$.

Note that $(\Htotop[1]{3})^*(x_1(\epsilon^4+\xi^2 x_1))=0$ by \autoref{xr-stuff}, since $\PH[1](\flagstyle{W}_3)_{+}\cong \HP^0_{+}\cong\S{0}$.

\subsubsection{Check the value of $\HtoGe[3](\Gen_3)$}

Within \autoref{fig:cofiber-cohomology} we find the maps $(\Htocofiber{2})^*$ and $(\Htotop{3})^*$, which we expand into their $\GG$ and $\Ge$ levels.

\begin{center}
\begin{tikzcd}[row sep=0.2cm,column sep=1.7cm]
{}&\HHreduced_{C_2}^{4+4\signrep}(\PH(Q)_{+})(\GG) \ar[bend right=40]{dd}[swap]{\rho}\\
{}\\
{}&\HHreduced_{C_2}^{4+4\signrep}(\PH(Q)_{+})(\Ge) \ar[bend right=40]{uu} \ar[start anchor = {[yshift=-0.4cm]}, end anchor = {[yshift=0.4cm]}]{ddd}{(\Htotop{3})^*} \\
\strut & \\
\strut & \\
\HHreduced_{C_2}^{4+4\signrep}(\S{4+4\signrep})(\GG) \ar[bend right=40]{dd}[swap]{\HcofibertoGe{2}}& \HHreduced_{C_2}^{4+4\signrep}(\PH(\flagstyle{W}_3)_{+})(\GG) \ar[bend right=40]{dd}[swap]{\HtoGe[3]}\\
{} \ar[start anchor = {[xshift=1.2cm]}, end anchor = {[xshift=-1.7cm]}]{r}[swap]{(\Htocofiber{2})^*} & {} \\
\HHreduced_{C_2}^{4+4\signrep}(\S{4+4\signrep})(\Ge) \ar[bend right=40]{uu} & \HHreduced_{C_2}^{4+4\signrep}(\PH(\flagstyle{W}_3)_{+})(\Ge) \ar[bend right=40]{uu}
\end{tikzcd}
\end{center}

Non-equivariantly, we have \[\PH(Q)_{+}\cong\HP^\infty_{+},\qquad \PH(\flagstyle{W}_3)_{+}\cong\HP^2_{+},\qquad \S{4+4\signrep}\cong\S{8}\]
and $\Htocofiber{2}$ is just the map $\HP^2_{+}\to\S{8}$ that collapses the $4$-cell of $\HP^2_{+}$ and glues the disjoint basepoint somewhere. 
By \autoref{nonequivariant}, the $\Ge$ level of this diagram just reflects what happens in non-equivariant cohomology:
\begin{center}
\begin{tikzcd}[row sep=1.7cm,column sep=1.7cm]
{} & \Hhreduced^{8}(\HP^\infty) \ar{d}{(\Htotop{3})^*}\\
\Hhreduced^{8}(\S{8}) \ar{r}[swap]{(\Htocofiber{2})^*} & \Hhreduced^{8}(\HP^2_{+})
\end{tikzcd}
\end{center}
namely, that $(\Htocofiber{2})^*$ sends the generator $1\in\Hhreduced^8(\S{8})$ to the generator $(\Htotop{3})^*(x^2)\in\Hhreduced^{8}(\HP^2_{+})$. 

Because $\HHreduced_{C_2}^{4+4\signrep}(\S{4+4\signrep})\cong \HHreduced_{C_2}^0(\S{0})\cong A$, we see that $\HcofibertoGe{2}$ sends the element $1\in\HHreduced_{C_2}^{4+4\signrep}(\S{4+4\signrep})(\GG)$ to the element $1\in\HHreduced_{C_2}^{4+4\signrep}(\S{4+4\signrep})(\Ge)\cong\Hhreduced^8(\S{8})$ (see \autoref{def:important-mackeys}), so that
\[\HtoGe[3](\Gen_3)=\HtoGe[3]((\Htocofiber{2})^*(1))=(\Htocofiber{2})^*(\HcofibertoGe{2})(1)=(\Htotop{3})^*(x^2).\qedhere\]

\end{proof}

Now we prove the inductive step, that an element $\Gen_n$ with the desired properties can always be lifted to an element $\Gen_{n+1}$ with the desired properties.

\begin{theorem}
\label{thm:Gen}
If $\Gen_n\in \HHreduced_{C_2}^{4+4\signrep}(\PH(\flagstyle{W}_n)_{+})$ exists, then  $\Gen_{n+1}\in\HHreduced_{C_2}^{4+4\signrep}(\PH(\flagstyle{W}_{n+1})_{+})$ exists.
\end{theorem}
\begin{proof}
Suppose that a class $\Gen_n\in \HHreduced_{C_2}^{4+4\signrep}(\PH(\flagstyle{W}_n)_{+})(\GG)$ has been defined with
\begin{align*}
(\Htowhole[n]{0})^*(\Gen_n)&=(\Htotop[0]{n})^*(x_0(\epsilon^4+\xi^2x_0))\\
(\Htowhole[n]{1})^*(\Gen_n)&=(\Htotop[1]{n})^*(x_1(\epsilon^4+\xi^2x_1))\\
\HtoGe[n](\Gen_n)&=(\Htotop{n})^*(x^2)
\end{align*}
for some $n\geq 3$.

\subsubsection{\texorpdfstring{Existence and uniqueness of an element $\Gen_{n+1}$}{Existence and uniqueness of an element C_\{n+1\}}}

We will lift the class $\Gen_n\in  \HHreduced_{C_2}^{4+4\signrep}(\PH(\flagstyle{W}_n)_{+})(\GG)$ along the map $(\Htonext{n})^*$, which fits into an exact sequence
\begin{center}
\begin{tikzcd}[column sep=0.9cm]
0\ar{r} & \HHreduced_{C_2}^{4+4\signrep}(\S{\cellstyle{w}_n})(\GG)\ar{r}{(\Htocofiber{n})^*} & \HHreduced_{C_2}^{4+4\signrep}(\PH(\flagstyle{W}_{n+1})_{+})(\GG)\ar{r}{(\Htonext{n})^*} &\HHreduced_{C_2}^{4+4\signrep}(\PH(\flagstyle{W}_n)_{+})(\GG)\ar{r} & 0
\end{tikzcd}
\end{center}
Note that \[\HH_{C_2}^{4+4\signrep}(\S{\cellstyle{w}_n})(\GG)\cong \HH_{C_2}^{4+4\signrep-\cellstyle{w}_n}(\S{0})(\GG).\]
If $n\geq 4$, we have $|\cellstyle{w}_n|\geq 16$ and $|\cellstyle{w}_n^{C_2}|\geq 8$, and therefore
\[|4+4\signrep-\cellstyle{w}_n|\leq 8-16=-8,\qquad |(4+4\signrep-\cellstyle{w}_n)^{C_2}|\leq 4-8=-4\]
so that by considering \autoref{figure-HS0-2}, $\HHreduced_{C_2}^{4+4\signrep}(\S{\cellstyle{w}_n})\cong\HHreduced_{C_2}^{4+4\signrep-\cellstyle{w}_n}(\S{0})=0$. Thus $(\Htonext{n})^*$ is an isomorphism in dimension $4+4\signrep$, hence there is exactly one lift of $\Gen_n$ along the map $(\Htonext{n})^*$, which we will define to be our element $\Gen_{n+1}$.

However, if $n=3$, then we have $\cellstyle{w}_3=4+8\signrep$, so that $|\cellstyle{w}_3|=12$ and $|\cellstyle{w}_3^{C_2}|=4$, and therefore \[|4+4\signrep-\cellstyle{w}_3|= 8-12=-4,\qquad |(4+4\signrep-\cellstyle{w}_3)^{C_2}|= 4-4=0\]
so that by considering \autoref{figure-HS0-2}, $\HHreduced_{C_2}^{4+4\signrep}(\S{\cellstyle{w}_3})\cong\HHreduced_{C_2}^{-4\signrep}(\S{0})=\langle\Z\rangle$, and there are infinitely many lifts of $\Gen_3$ along $(\Htonext{3})^*$. We will show that exactly one of these lifts has the desired value under $(\Htowhole[4]{1})^*$, and this will be our choice of $\Gen_4$.

Note that -- \emph{equivariantly}, since these spaces have a trivial $C_2$-action -- we have
\[\Hcofiber[1]{3}\cong\S{4},\qquad \PH[1](\flagstyle{W}_{4})_{+}\cong\HP^1_{+},\qquad \PH[1](\flagstyle{W}_3)_{+}\cong\HP^0_{+}\]
and consider this section of \autoref{fig:cofiber-cohomology}.
\begin{center}
\begin{tikzcd}[row sep=1.3cm,column sep=1.3cm]
0 \ar{r} & \HHreduced_{C_2}^{4+4\signrep}(\S{\cellstyle{w}_3}) \ar{d}[swap]{(\Htowholecofiber[1]{3})^*} \ar{r}{(\Htocofiber{3})^*} & \HHreduced_{C_2}^{4+4\signrep}(\PH(\flagstyle{W}_{4})_{+}) \ar{r}{(\Htonext{3})^*} \ar{d}{(\Htowhole[4]{1})^*} &  \HHreduced_{C_2}^{4+4\signrep}(\PH(\flagstyle{W}_3)_{+}) \ar{d}{(\Htowhole[3]{1})^*} \ar{r} & 0\\
0\ar{r} & \HHreduced_{C_2}^{4+4\signrep}(\S{4}) \ar{r}[swap]{(\Htocofiber[1]{3})^*} & \HHreduced_{C_2}^{4+4\signrep}(\PH[1](\flagstyle{W}_{4})_{+}) \ar{r}[swap]{(\Htonext[1]{3})^*}  & \HHreduced_{C_2}^{4+4\signrep}(\PH[1](\flagstyle{W}_3)_{+}) \ar{r} & 0
\end{tikzcd}
\end{center}
As we showed in the proof of \autoref{lem:Gen-initial}, we have $(\Htowhole[3]{1})^*(\Gen_3)=0$. By the exactness of the bottom row, this means that for any lift $\hat{\Gen}_3$ of $\Gen_3$ along the map $(\Htonext{3})^*$, we must have that $(\Htowhole[4]{1})^*(\hat{\Gen}_3)$ is in the image of $(\Htocofiber[1]{3})^*$.

In non-equivariant cohomology, the map $(\Htocofiber[1]{3})^*:\Hhreduced^4(\S{4})\to\Hhreduced^4(\HP^1_{+})$ sends the generator $1$ to the generator $(\Htotop[1]{4})^*(x_1)$, so by \autoref{xr-stuff}, we conclude that the map 
\[(\Htocofiber[1]{3})^*:\HHreduced_{C_2}^{4+4\signrep}(\S{4})\cong\HH_{C_2}^{4\signrep}(\S{0})\otimes \Hhreduced^4(\S{4})\to\HH_{C_2}^{4\signrep}(\S{0})\otimes\Hhreduced^4(\HP^1_{+})\subseteq \HHreduced_{C_2}^{4+4\signrep}(\HP^1_{+})\]
sends the generator
$\epsilon^4\in\HHreduced_{C_2}^{4+4\signrep}(\S{4})(\GG)\cong \langle\Z\rangle(\GG)$ to  $(\Htotop[1]{4})^*(\epsilon^4x_1)\in\HHreduced_{C_2}^{4+4\signrep}(\PH[1](\flagstyle{W}_4)_{+})(\GG)$. Since $(\Htowholecofiber[1]{3})^*$ is an isomorphism from
\[\HHreduced_{C_2}^{4+4\signrep}(\S{4+8\signrep})\cong\HHreduced_{C_2}^{-4\signrep}(\S{0})\cong\langle\Z\rangle\]
to
\[\HHreduced_{C_2}^{4\signrep}(\S{0})\cong\HHreduced_{C_2}^{4+4\signrep}(\S{4})\cong\langle\Z\rangle\]
(see the footnote at \cite[p.36]{megan}), each element in the image of $(\Htocofiber[1]{3})^*$ occurs as the image under $(\Htowhole[4]{1})^*$ of some lift of $\Gen_3$ . In particular, we can choose  $\Gen_4\in\HHreduced_{C_2}^{4+4\signrep}(\PH(\flagstyle{W}_4)_{+})(\GG)$ with $(\Htowhole[4]{1})^*(C_4)=(\Htotop[1]{4})^*(\epsilon^4x_1)$, which equals the desired value of $(\Htotop[1]{4})^*(x_1(\epsilon^4+\xi^2x_1))$ because $\PH[1](\flagstyle{W}_4)_{+}\cong\HP^1_{+}$.


\subsubsection{Check the value of $\HtoGe[n+1](\Gen_{n+1})$}

For clarity, we expand the maps $(\Htonext{n})^*$, $(\Htotop{n})^*$, and $(\Htotop{n+1})^*$ into their $\GG$ and $\Ge$ levels.

\begin{center}
\begin{tikzcd}[row sep=0.2cm,column sep=1.7cm]
\HHreduced_{C_2}^{4+4\signrep}(\PH(Q)_{+})(\GG) \ar[bend right=40]{dd}[swap]{\HtoGe}\\
{}\\
\HHreduced_{C_2}^{4+4\signrep}(\PH(Q)_{+})(\Ge) \ar[bend right=40]{uu} \ar[start anchor = {[yshift=-0.4cm]}, end anchor = {[yshift=0.4cm]}]{ddd}[swap]{(\Htotop{n+1})^*} \ar[start anchor = {[xshift=1.2cm,yshift=-0.0cm]}, end anchor = {[xshift=-0.4cm,yshift=0.4cm]}]{dddr}{(\Htotop{n})^*} \\
& \strut \\
& \strut \\
 \HHreduced_{C_2}^{4+4\signrep}(\PH(\flagstyle{W}_{n+1})_{+})(\GG) \ar[bend right=40]{dd}[swap]{\HtoGe[n+1]} & \HHreduced_{C_2}^{4+4\signrep}(\PH(\flagstyle{W}_{n})(\GG) \ar[bend right=40]{dd}[swap]{\HtoGe[n]}\\
{} \ar[start anchor = {[xshift=1.5cm]}, end anchor = {[xshift=-1.3cm]}]{r}[swap]{(\Htonext{n})^*} & {} \\ \HHreduced_{C_2}^{4+4\signrep}(\PH(\flagstyle{W}_{n+1})_{+})(\Ge) \ar[bend right=40]{uu} & 
\HHreduced_{C_2}^{4+4\signrep}(\PH(\flagstyle{W}_{n})(\Ge) \ar[bend right=40]{uu} 
\end{tikzcd}
\end{center}

Non-equivariantly, we have \[\PH(Q)_{+}\cong\HP^\infty_{+},\qquad \PH(\flagstyle{W}_{n+1})_{+}\cong\HP^{n}_{+},\qquad \PH(\flagstyle{W}_{n})_{+}\cong\HP^{n-1}_{+}\]
and $\Htonext{n}$ is just the inclusion map $\HP^{n-1}_{+}\to\HP^{n}_{+}$ that adds a $4n$-cell to $\HP^{n-1}_{+}$. By \autoref{nonequivariant}, the $\Ge$ level of this diagram just reflects what happens in non-equivariant cohomology:
\begin{center}
\begin{tikzcd}[row sep=1.7cm,column sep=1.7cm]
\Hhreduced^{8}(\HP^\infty_{+}) \ar{d}[swap]{(\Htotop{n+1})^*} \ar{dr}{(\Htotop{n})^*}\\
\Hhreduced^{8}(\HP^{n}_{+}) \ar{r}[swap]{(\Htonext{n})^*}{\cong}  &  \Hhreduced^{8}(\HP^{n-1}_{+})
\end{tikzcd}
\end{center}
namely, that the map $(\Htonext{n})^*$ is an isomorphism in dimension $8$, and sends the generator $(\Htotop{n+1})^*(x^2)\in\Hhreduced^{8}(\HP^{n}_{+})$ to the generator $(\Htotop{n})^*(x^2)\in\Hhreduced^{8}(\HP^{n-1}_{+})$.

By the induction hypothesis we know that $\HtoGe[n](\Gen_{n})=(\Htotop{n})^*(x^2)$, and because $\Gen_{n+1}$ is a lift of $\Gen_n$ along $(\Htonext{n})^*$, we can conclude $\HtoGe[n+1](\Gen_{n+1})=(\Htotop{n+1})^*(x^2)$, as desired.

\subsubsection{\texorpdfstring{Check the value of $(\Htowhole[n+1]{r})^*(\Gen_{n+1})$ for $n\not\equiv r\bmod 2$}{Check the value of (q\_{n+1}\^r)*(C\_\{n+1\}) for n != r mod 2}}

Since $n\not\equiv r\bmod 2$, we have that $\Hcofiber[r]{n}=\pt$, so that $\HHreduced_{C_2}^{4+4\signrep}(\Hcofiber[r]{n})=0$, and hence $(\Htonext[r]{n})^*$ is an isomorphism in dimension $4+4\signrep$.
Now the relevant piece of \autoref{fig:cofiber-cohomology} is
\begin{center}
\begin{tikzcd}[row sep=1.3cm,column sep=1.3cm]
\HHreduced_{C_2}^{4+4\signrep}(\PH(\flagstyle{W}_{n+1})_{+}) \ar{r}{(\Htonext{n})^*}[swap]{} \ar{d}[swap]{(\Htowhole[n+1]{r})^*} &  \HHreduced_{C_2}^{4+4\signrep}(\PH(\flagstyle{W}_n)_{+}) \ar{d}{(\Htowhole[n]{r})^*} 
\\
\HHreduced_{C_2}^{4+4\signrep}(\PH[r](\flagstyle{W}_{n+1})_{+}) \ar{r}[swap]{(\Htonext[r]{n})^*}{\cong}  & \HHreduced_{C_2}^{4+4\signrep}(\PH[r](\flagstyle{W}_n)_{+})
\\
\HHreduced_{C_2}^{4+4\signrep}(\PH[r](Q)_{+}) \ar{u}{(\Htotop[r]{n+1})^*} \ar{ru}[swap]{(\Htotop[r]{n})^*}
\end{tikzcd}
\end{center}
By our inductive hypothesis, we know that $(\Htowhole[n]{r})^*(\Gen_n)=(\Htotop[r]{n})^*(x_r(\epsilon^4+\xi^2x_r))$, so we can conclude
\begin{align*}
(\Htonext[r]{n})^*\circ(\Htowhole[n+1]{r})^*(\Gen_{n+1})&=(\Htowhole[n]{r})^*\circ(\Htonext{n})^*(\Gen_{n+1})\\
(\Htonext[r]{n})^*\circ(\Htowhole[n+1]{r})^*(\Gen_{n+1})&=(\Htowhole[n]{r})^*(\Gen_{n})\\
(\Htonext[r]{n})^*\circ(\Htowhole[n+1]{r})^*(\Gen_{n+1})&=(\Htotop[r]{n})^*(x_r(\epsilon^4+\xi^2x_r))\\
(\Htowhole[n+1]{r})^*(\Gen_{n+1})&=((\Htonext[r]{n})^*)^{-1}\circ(\Htotop[r]{n})^*(x_r(\epsilon^4+\xi^2x_r))\\
(\Htowhole[n+1]{r})^*(\Gen_{n+1})&=(\Htotop[r]{n+1})^*(x_r(\epsilon^4+\xi^2x_r))
\end{align*}
which is what is desired.

\subsubsection{\texorpdfstring{Check the value of $(\Htowhole[n+1]{r})^*(\Gen_{n+1})$ for $n\equiv r\bmod 2$}{Check the value of (q\_{n+1}\^r)*(C\_\{n+1\}) for n = r mod 2}}

We started with the assumption that $n\geq 3$, and we have already chosen our element $\Gen_4$ such that $(\Htowhole[4]{1})^*(\Gen_4)$ is the desired value. Therefore in this section we can assume $n\geq 4$.

Since $n\equiv r\bmod 2$, we have that $\Hcofiber[r]{n}=\S{|\flagstyle{W}_n(r;\C)|}=\S{4\lfloor n/2\rfloor}$ by \autoref{rem:cells-for-2}. Note that because $n\geq 4$, we have that $4\lfloor n/2\rfloor\geq 8$.

We now consider this section of \autoref{fig:cofiber-cohomology}.
\begin{center}
\begin{tikzcd}[row sep=1.3cm,column sep=1.1cm]
0\ar{r} & \HHreduced_{C_2}^{4+4\signrep}(\S{4\lfloor n/2\rfloor}) \ar{r}{(\Htocofiber[r]{n})^*} & \HHreduced_{C_2}^{4+4\signrep}(\PH[r](\flagstyle{W}_{n+1})_{+}) \ar{r}{(\Htonext[r]{n})^*}  & \HHreduced_{C_2}^{4+4\signrep}(\PH[r](\flagstyle{W}_n)_{+}) \ar{r} & 0\\
{} & & \HHreduced_{C_2}^{4+4\signrep}(\PH[r](Q)_{+}) \ar{u}{(\Htotop[r]{n+1})^*} \ar{ru}[swap]{(\Htotop[r]{n})^*}
\end{tikzcd}
\end{center}

If in fact $4\lfloor \frac{n}{2}\rfloor>8$, then it must be at least $12$, hence
\[\bigr|4+4\signrep-4\lfloor \tfrac{n}{2}\rfloor\bigr|\leq 8-12=-4,\qquad \bigl|(4+4\signrep-4\lfloor \tfrac{n}{2}\rfloor)^{G}\bigr|\leq 4-12=-8.\]
By considering \autoref{figure-HS0-2}, we see that
\[\HHreduced_{C_2}^{4+4\signrep}(\S{4\lfloor n/2\rfloor})\cong \HHreduced_{C_2}^{4+4\signrep-4\lfloor n/2\rfloor}(\S{0})=0,\]
so that $(\Htonext[r]{n})^*$ is an isomorphism, which implies that $(\Htowhole[n+1]{r})^*(\Gen_{n+1})$ is the desired value of $(\Htotop[r]{n+1})^*(x_r(\epsilon^4+\xi^2x_r^2))$ by the same argument as in the previous section.

If instead $4\lfloor \frac{n}{2}\rfloor=8$, then it must be the case that $n=4$ or $5$, so that --  \emph{equivariantly}, since these spaces have a trivial $C_2$-action --  we have
\[\Hcofiber[r]{n}\cong\S{8},\qquad \PH[r](\flagstyle{W}_{n+1})_{+}\cong\HP^2_{+},\qquad \PH[r](\flagstyle{W}_n)_{+}\cong\HP^1_{+}\]
and hence our diagram is
\begin{center}
\begin{tikzcd}[row sep=1.3cm,column sep=1.3cm]
0\ar{r} & \HHreduced_{C_2}^{4+4\signrep}(\S{8}) \ar{r}{(\Htocofiber[r]{n})^*} & \HHreduced_{C_2}^{4+4\signrep}(\HP^2_{+}) \ar{r}{(\Htonext[r]{n})^*}  & \HHreduced_{C_2}^{4+4\signrep}(\HP^1_{+}) \ar{r} & 0\\
{} & & \HHreduced_{C_2}^{4+4\signrep}(\HP^\infty_{+}) \ar{u}{(\Htotop[r]{n+1})^*} \ar{ru}[swap]{(\Htotop[r]{n})^*}
\end{tikzcd}
\end{center}
Note that $|4+4\signrep-8|=0$ and $|(4+4\signrep-8)^{C_2}|=-4$, so by observing \autoref{figure-HS0-2}, we see that \[\HHreduced_{C_2}^{4+4\signrep}(\S{8})\cong \HHreduced_{C_2}^{4+4\signrep-8}(\S{0})\cong R.\]
By \autoref{def:xi-for-2}, we see that $\xi^2$ is the generator of $R(\GG)\cong\Z$. In non-equivariant cohomology, the map $(\Htocofiber[r]{n})^*:\Hhreduced^8(\S{8})\to\Hhreduced^8(\HP^2_{+})$ sends the generator $1$ to the generator $(\Htotop[r]{n+1})^*(x_r^2)$, so by \autoref{xr-stuff}, we conclude that the map 
\[(\Htocofiber[r]{n})^*:\HHreduced_{C_2}^{4+4\signrep}(\S{8})\cong\HHreduced_{C_2}^{4+4\signrep-8}(\S{0})\otimes \Hhreduced^8(\S{8})\to\HHreduced_{C_2}^{4+4\signrep-8}(\S{0})\otimes\Hhreduced^8(\HP^2_{+})\subseteq \HHreduced_{C_2}^{4+4\signrep}(\HP^2_{+})\]
sends the generator $\xi^2\in\HHreduced_{C_2}^{4+4\signrep}(\S{8})(\GG)\cong R(\GG)$ to the element $(\Htotop[r]{n+1})^*(\xi^2x_r^2)\in\HHreduced_{C_2}^{4+4\signrep}(\HP^2_{+})(\GG)$.

By our induction hypothesis,
\[((\Htonext[r]{n})^*\circ(\Htowhole[n+1]{r})^*)(\Gen_{n+1})=(\Htowhole[n]{r})^*(\Gen_{n})=(\Htotop[r]{n})^*(x_r(\epsilon^4+\xi^2x_r))\]
and therefore $(\Htowhole[n+1]{r})^*(\Gen_{n+1})$ is a lift of $(\Htotop[r]{n})^*(x_r(\epsilon^4+\xi^2x_r))$ along the map $(\Htonext[r]{n})^*$. Of course, the lift we would like it to be is precisely $(\Htotop[r]{n+1})^*(x_r(\epsilon^4+\xi^2x_r))$.

Because $\PH[r](\flagstyle{W}_n)_{+}\cong\HP^1_{+}$, we have that $(\Htotop[r]{n})^*(x_r^2)=0$ by \autoref{xr-stuff}. Thus \[(\Htotop[r]{n})^*(x_r(\epsilon^4+\xi^2x_r))=\epsilon^4x_r,\] and therefore, \emph{a priori}, the value of a lift of $(\Htotop[r]{n})^*(x_r(\epsilon^4+\xi^2x_r))$ along the map $(\Htonext[r]{n})^*$ could be $(\Htotop[r]{n+1})^*(x_r(\epsilon^4+k\xi^2x_r))$ for any $k\in\Z$, since the image of $(\Htocofiber[r]{n})^*$ is the kernel of $(\Htonext[r]{n})^*$.

Consider the map $(\Htowhole[n+1]{r})^*$ in \autoref{fig:cofiber-cohomology}, expanded into its $\GG$ and $\Ge$ components.
\begin{center}
\begin{tikzcd}[row sep=0.2cm,column sep=1.7cm]
\HHreduced_{C_2}^{4+4\signrep}(\PH(\flagstyle{W}_{n+1})_{+})(\GG) \ar[bend right=40]{dd}[swap]{\HtoGe[n+1]}\\
{}\\
\HHreduced_{C_2}^{4+4\signrep}(\PH(\flagstyle{W}_{n+1})_{+})(\Ge) \ar[bend right=40]{uu} \ar[start anchor = {[yshift=-0.4cm]}, end anchor = {[yshift=0.4cm]}]{ddd}[swap]{(\Htowhole[n+1]{r})^*}\\
\strut\\
\strut\\
 \HHreduced_{C_2}^{4+4\signrep}(\PH[r](\flagstyle{W}_{n+1})_{+})(\GG) \ar[bend right=40]{dd}[swap]{\HparttoGe[n+1]{r}} \\
{} \\ \HHreduced_{C_2}^{4+4\signrep}(\PH[r](\flagstyle{W}_{n+1})_{+})(\Ge) \ar[bend right=40]{uu}
\end{tikzcd}
\end{center}
Again by \autoref{nonequivariant}, the $\Ge$ level of this diagram just reflects what happens in non-equivariant cohomology:
\begin{center}
\begin{tikzcd}[row sep=1.7cm,column sep=1.7cm]
\Hhreduced^{8}(\HP^n_{+}) \ar{d}[swap]{(\Htowhole[n+1]{r})^*}\\
\Hhreduced^{8}(\HP^{2}_{+})
\end{tikzcd}
\end{center}
namely, that the map $(\Htowhole[n+1]{r})^*$ is an isomorphism in dimension $8$, and sends the generator $(\Htotop{n+1})^*(x^2)\in\Hhreduced^{8}(\HP^{n}_{+})\cong\Z$ to the generator $(\HparttoGe[n+1]{r})^*((\Htotop[r]{n+1})^*(\xi^2x_r^2))\in\Hhreduced^{8}(\HP^{2}_{+})\cong\Z$.

However, we have already established that $(\HtoGe[n+1])^*(\Gen_{n+1})=(\Htotop{n+1})^*(x^2)$, and therefore
\begin{align*}
(\HparttoGe[n+1]{r})^*((\Htowhole[n+1]{r})^*(\Gen_{n+1}))&=(\Htowhole[n+1]{r})^*((\HtoGe[n+1])^*(\Gen_{n+1}))\\
&=(\Htowhole[n+1]{r})^*((\Htotop{n+1})^*(x^2))\\
&=(\HparttoGe[n+1]{r})^*((\Htotop[r]{n+1})^*(\xi^2x_r^2))
\end{align*}
Because $\HparttoGe[n+1]{r}$ is an isomorphism (see the definition of $R$ in \autoref{def:important-mackeys}), the value of $(\Htowhole[n+1]{r})^*(\Gen_{n+1})$, which we showed must be $(\Htotop[r]{n+1})^*(x_r(\epsilon^4+k\xi^2x_r^2))$ for some $k\in\Z$, must in fact be \[(\Htotop[r]{n+1})^*(x_r(\epsilon^4+\xi^2x_r^2)),\] as desired.
\end{proof}

\section{\texorpdfstring{Proof that $\gen$ and $\Gen$ generate $\HHreduced_{C_2}^{*}(B_{C_2}\SU(2)_{+}\coeffs*{A})$}{Proof that c and C generate H\_\{C\_2\}*(B\_\{C\_2\}SU(2)\_+;A)}}
\label{subsection:proof-cC-generators}

The elements $\gen_n$ and $\Gen_{n}$ constructed in \autoref{subsection:construct-c} and \autoref{subsection:construct-C} may seem to only be tools used to obtain the desired elements $\gen$ and $\Gen$. However, in proving that $\gen$ and $\Gen$ generate $\HHreduced_{C_2}^{*}(\PH(Q)_{+}\coeffs*{A})$, the elements $\gen_n$ and $\Gen_n$ will remain important for us because they generate $\HHreduced_{C_2}^*(\PH(\flagstyle{W}_n)_{+}\coeffs*{A})$, i.e., the cohomology of the $n$th piece of the filtration of $\PH(Q)$.

We will first establish this in \autoref{generate-finite-levels-2}, which comprises the main part of the work, and which again uses an inductive argument based on checking images under $\HtoGe[n]$, $(\Htowhole[n]{0})^*$, and $(\Htowhole[n]{1})^*$. Then in \autoref{thm:generate}, we use this to conclude that $\gen$ and $\Gen$ generate $\HHreduced_{C_2}^{*}(\PH(Q)_{+}\coeffs*{A})$, thereby proving the first half of our main \autoref{thm:main}.   The key observation in this latter step is that $\gen_n$ and $\Gen_n$ are the images of $\gen$ and $\Gen$, respectively, in $\HHreduced_{C_2}^*(\PH(\flagstyle{W}_n)_{+}\coeffs*{A})$.

\begin{lemma}
\label{generate-finite-levels-2}
For any $n\geq1$, the elements $(\Htotop{n})^*(\gen)=\gen_n$ and $(\Htotop{n})^*(\Gen)=\Gen_n$ together generate $\HHreduced_{C_2}^*(\PH(\flagstyle{W}_n)_{+}\coeffs*{A})$  as an $\HHreduced_{C_2}^{*}(\S{0}\coeffs*{A})$-algebra.
\end{lemma}
\begin{proof}
This is obvious for $n=1$ since $\PH(\flagstyle{W}_1)_{+}\cong \S{0}$, so assume the result for $n$. 


From \autoref{fig:cofiber-cohomology}, we have a split short exact sequence
\[0\to\HHreduced_{C_2}^*(\S{\cellstyle{w}_n})\xrightarrow{\;\;(\Htocofiber{n})^*\;\;}\HHreduced_{C_2}^*(\PH(\flagstyle{W}_{n+1})_{+})\xrightarrow{(\Htonext{n})^*}\HHreduced_{C_2}^*(\PH(\flagstyle{W}_{n})_{+})\to0\]
and let $\nu=(\Htocofiber{n})^*(1)$ be the image of $1\in\HHreduced_{C_2}^{\cellstyle{w}_n}(\S{\cellstyle{w}_n})(\GG)\cong\HHreduced_{C_2}^0(S^0)(\GG)\cong A(\GG)$ under the map $(\Htocofiber{n})^*$. Because the elements
\[(\Htonext{n})^*((\Htotop{n+1})^*(\gen))=(\Htotop{n})^*(\gen),\qquad (\Htonext{n})^*((\Htotop{n+1})^*(\Gen))=(\Htotop{n})^*(\Gen)\]
already generate $\HHreduced_{C_2}^*(\PH(\flagstyle{W}_{n})_{+})$ by hypothesis, we only have to check that $(\Htotop{n+1})^*(\gen)$ and $(\Htotop{n+1})^*(\Gen)$ generate $\nu$. To demonstrate this, we claim in particular that
\[\nu=\begin{cases}
(\Htotop{n+1})^*(\Gen^{n/2}) & \text{if $n$ is even},\\
(\Htotop{n+1})^*(\gen\Gen^{(n-1)/2}) & \text{if $n$ is odd}.
\end{cases}\]
It suffices to check this under the maps $\HtoGe[n+1]$, $(\Htowhole[n+1]{0})^*$, and $(\Htowhole[n+1]{1})^*$.

\subsubsection{Compute the desired value of $\HtoGe[n+1](\nu)$}
First, note that in $\HHreduced_{C_2}^{*}(\PH(Q)_{+})(\Ge)$
\[\HtoGe(\Gen^k)=x^{2k},\qquad \HtoGe(\gen\Gen^k)=x^{2k+1},\]
so that in $\HHreduced_{C_2}^{*}(\PH(\flagstyle{W}_{n+1})_{+})(\Ge)$, we have
\[\HtoGe[n+1]((\Htotop{n+1})^*(\Gen^{k}))=(\Htotop{n+1})^*(x^{2k}),\qquad \HtoGe[n+1]((\Htotop{n+1})^*(\gen\Gen^k))=(\Htotop{n+1})^*(x^{2k+1}).\]
Therefore, in particular, we have for even $n$ that 
\[\HtoGe[n+1]((\Htotop{n+1})^*(\Gen^{n/2}))=(\Htotop{n+1})^*(x^{n})\]
and for odd $n$ that
\[\HtoGe[n+1]((\Htotop{n+1})^*(\gen\Gen^{(n-1)/2}))=(\Htotop{n+1})^*(x^{n}).\]
Thus, regardless of the parity of $n$, we must show that $\HtoGe[n+1](\nu)=(\Htotop{n+1})^*(x^n)$.

\subsubsection{Compute the desired value of $(\Htowhole[n+1]{0})^*(\nu)$}

In $\HHreduced_{C_2}^*(\PH[0](Q)_{+})(\GG)$, we have for any $k$ that
\begin{align*}
(\Htowhole{0})^*(\Gen^k)&=(x_0(\epsilon^4+\xi^2x_0))^k\\
&=x_0^k\sum_{s=0}^k\binom{k}{s}(\epsilon^4)^{k-s}(\xi^2x_0)^s\\
&=\epsilon^{4k}x_0^k+\binom{k}{1}\epsilon^{4k-4}\xi^2x_0^{k+1}+\cdots
\intertext{and therefore}
(\Htowhole{0})^*(\gen\Gen^k)&=(\xi^2x_0)(\Htowhole{0})^*(\Gen^k)  \\
&=\epsilon^{4k}\xi^2x_0^{k+1}+\binom{k}{1}\epsilon^{4k-4}\xi^4x_0^{k+2}+\cdots
\end{align*}
Because $\flagstyle{W}_{m}=\irredH_0^{\lceil m/2\rceil}\oplus\irredH_1^{\lfloor m/2\rfloor}$, we have that
\[\PH[0](\flagstyle{W}_m)_{+}=\PH(\flagstyle{W}_m(0;\H))_{+}\cong\HP^{\lceil m/2\rceil -1}_{+}.\] 
By \autoref{xr-stuff}, this implies that $(\Htotop[0]{m})^*(x_0^d)=0$ when $d\geq \lceil m/2\rceil$. Therefore, for even $n$, we have $(\Htotop[0]{n+1})^*(x_0^d)=0$ when $d\geq (n/2)+1$, and for odd $n$, we have $(\Htotop[0]{n+1})^*(x_0^d)=0$ when $d\geq (n+1)/2$. Now we can see that for even $n$, 
\begin{align*}
(\Htowhole[n+1]{0})^*((\Htotop{n+1})^*(\Gen^{n/2}))&=(\Htotop[0]{n+1})^*((\Htowhole{0})^*(\Gen^{n/2}))\\
&=(\Htotop[0]{n+1})^*\left(\epsilon^{2n}x_0^{n/2}+\binom{n/2}{1}\epsilon^{2n-4}\xi^2x_0^{(n/2)+1}+\cdots\right)\\
&=(\Htotop[0]{n+1})^*(\epsilon^{2n}x_0^{n/2})
\intertext{and for odd $n$,}
(\Htowhole[n+1]{0})^*((\Htotop{n+1})^*(\gen\Gen^{(n-1)/2}))&=(\Htotop[0]{n+1})^*((\Htowhole{0})^*(\gen\Gen^{(n-1)/2}))\\
&=(\Htotop[0]{n+1})^*\left(\epsilon^{2n-2}\xi^2x_0^{(n+1)/2}+\binom{(n-1)/2}{1}\epsilon^{2n-6}\xi^4x_0^{(n+3)/2}+\cdots\right)\\
&=0.
\end{align*}
Thus, we must show that
\[(\Htowhole[n+1]{0})^*(\nu)=\begin{cases}
(\Htotop[0]{n+1})^*(\epsilon^{2n}x_0^{n/2}) & \text{if $n$ is even,}\\
0 & \text{if $n$ is odd.}
\end{cases}\]

\subsubsection{Compute the desired value of $(\Htowhole[n+1]{1})^*(\nu)$}

In $\HHreduced_{C_2}^*(\PH[1](Q)_{+})(\GG)$, we have for any $k$ that
\begin{align*}
(\Htowhole{1})^*(\Gen^k)&=(x_1(\epsilon^4+\xi^2x_1))^k\\
&=x_1^k\sum_{s=0}^k\binom{k}{s}(\epsilon^4)^{k-s}(\xi^2x_1)^s\\
&=\epsilon^{4k}x_1^k+\binom{k}{1}\epsilon^{4k-4}\xi^2x_1^{k+1}+\cdots
\intertext{and therefore}
(\Htowhole{1})^*(\gen\Gen^k)&=(\epsilon^4+\xi^2x_1)(\Htowhole{1})^*(\Gen^k)  \\
&=\epsilon^{4k+4}x_1^{k}+\binom{k+1}{1}\epsilon^{4k}\xi^2x_1^{k+1}+\cdots
\end{align*}
Because $\flagstyle{W}_{m}=\irredH_0^{\lceil m/2\rceil}\oplus\irredH_1^{\lfloor m/2\rfloor}$, we have that
\[\PH[1](\flagstyle{W}_m)_{+}=\PH(\flagstyle{W}_m(1;\H))_{+}\cong\HP^{\lfloor m/2\rfloor -1}_{+}.\] 
By \autoref{xr-stuff}, this implies that $(\Htotop[1]{m})^*(x_1^d)=0$ when $d\geq \lfloor m/2\rfloor$. Therefore, for even $n$, we have $(\Htotop[1]{n+1})^*(x_1^d)=0$ when $d\geq n/2$, and for odd $n$, we have $(\Htotop[1]{n+1})^*(x_0^d)=0$ when $d\geq (n+1)/2$. Now we can see that for even $n$, 
\begin{align*}
(\Htowhole[n+1]{1})^*((\Htotop{n+1})^*(\Gen^{n/2}))&=(\Htotop[1]{n+1})^*((\Htowhole{1})^*(\Gen^{n/2}))\\
&=(\Htotop[1]{n+1})^*\left(\epsilon^{2n}x_1^{n/2}+\binom{n/2}{1}\epsilon^{2n-4}\xi^2x_1^{(n/2)+1}+\cdots\right)\\
&=0
\intertext{and for odd $n$,}
(\Htowhole[n+1]{1})^*((\Htotop{n+1})^*(\gen\Gen^{(n-1)/2}))&=(\Htotop[1]{n+1})^*((\Htowhole{1})^*(\gen\Gen^{(n-1)/2}))\\
&=(\Htotop[1]{n+1})^*\left(\epsilon^{2n+2}x_1^{(n-1)/2}+\binom{(n+1)/2}{1}\epsilon^{2n-2}\xi^2x_1^{(n+1)/2}+\cdots\right)\\
&=(\Htotop[1]{n+1})^*(\epsilon^{2n+2}x_1^{(n-1)/2}).
\end{align*}
Thus, we will want to show that
\[(\Htowhole[n+1]{1})^*(\nu)=\begin{cases}
0 & \text{if $n$ is even,}\\
(\Htotop[1]{n+1})^*(\epsilon^{2n+2}x_1^{(n-1)/2}) & \text{if $n$ is odd.}
\end{cases}\]

\subsubsection{Check that $\HtoGe[n+1](\nu)$ equals the desired value}
Within \autoref{fig:cofiber-cohomology} we find the maps $(\Htocofiber{n})^*$ and $(\Htotop{n+1})^*$, which we expand into their $\GG$ and $\Ge$ levels.

\begin{center}
\begin{tikzcd}[row sep=0.2cm,column sep=1.7cm]
{}&\HHreduced_{C_2}^{\cellstyle{w}_n}(\PH(Q)_{+})(\GG) \ar[bend right=40]{dd}[swap]{\rho}\\
{}\\
{}&\HHreduced_{C_2}^{\cellstyle{w}_n}(\PH(Q)_{+})(\Ge) \ar[bend right=40]{uu} \ar[start anchor = {[yshift=-0.4cm]}, end anchor = {[yshift=0.4cm]}]{ddd}{(\Htotop{n+1})^*} \\
\strut & \\
\strut & \\
\HHreduced_{C_2}^{\cellstyle{w}_n}(\S{\cellstyle{w}_n})(\GG) \ar[bend right=40]{dd}[swap]{\HcofibertoGe{n}}& \HHreduced_{C_2}^{\cellstyle{w}_n}(\PH(\flagstyle{W}_{n+1})_{+})(\GG) \ar[bend right=40]{dd}[swap]{\HtoGe[n+1]}\\
{} \ar[start anchor = {[xshift=1.2cm]}, end anchor = {[xshift=-1.7cm]}]{r}[swap]{(\Htocofiber{n})^*} & {} \\
\HHreduced_{C_2}^{\cellstyle{w}_n}(\S{\cellstyle{w}_n})(\Ge) \ar[bend right=40]{uu} & \HHreduced_{C_2}^{\cellstyle{w}_n}(\PH(\flagstyle{W}_{n+1})_{+})(\Ge) \ar[bend right=40]{uu}
\end{tikzcd}
\end{center}
Non-equivariantly, we have \[\PH(Q)_{+}\cong\HP^\infty_{+},\qquad \PH(\flagstyle{W}_{n+1})_{+}\cong\HP^n_{+},\qquad \S{\cellstyle{w}_n}\cong\S{4n}\]
and $\Htocofiber{n}$ is just the map $\HP^n_{+}\to\S{4n}$ that collapses the $4k$-cells of $\HP^n_{+}$ for $k<n$, and glues the disjoint basepoint somewhere. 
By \autoref{nonequivariant}, the $\Ge$ level of this diagram just reflects what happens in non-equivariant cohomology:
\begin{center}
\begin{tikzcd}[row sep=1.7cm,column sep=1.7cm]
{} & \Hhreduced^{4n}(\HP^\infty) \ar{d}{(\Htotop{n+1})^*}\\
\Hhreduced^{4n}(\S{4n}) \ar{r}[swap]{(\Htocofiber{n})^*} & \Hhreduced^{4n}(\HP^n_{+})
\end{tikzcd}
\end{center}
namely, that $(\Htocofiber{n})^*$ sends the generator $1\in\Hhreduced^{4n}(\S{4n})$ to the generator $(\Htotop{n+1})^*(x^n)\in\Hhreduced^{4n}(\HP^n_{+})$. 

Because $\HHreduced_{C_2}^{\cellstyle{w}_n}(\S{\cellstyle{w}_n})\cong \HHreduced_{C_2}^0(\S{0})\cong A$, we can see that $\HcofibertoGe{n}$ sends the element $1\in\HHreduced_{C_2}^{\cellstyle{w}_n}(\S{\cellstyle{w}_n})(\GG)$ to the element $1\in\HHreduced_{C_2}^{\cellstyle{w}_n}(\S{\cellstyle{w}_n})(\Ge)\cong\Hhreduced^{4n}(\S{4n})$ (see \autoref{def:important-mackeys}), so that
\[\HtoGe[n+1](\nu)=\HtoGe[n+1]((\Htocofiber{n})^*(1))=(\Htocofiber{n})^*(\HcofibertoGe{n})(1)=(\Htotop{n+1})^*(x^n).\]

\subsubsection{Check that $(\Htowhole[n+1]{0})^*(\nu)$ and $(\Htowhole[n+1]{1})^*(\nu)$ equal the desired values for even $n$}

When $n$ is even, we have that
\begin{align*}
\flagstyle{W}_{n+1}&=\irredH_0^{(n/2)+1}\oplus\irredH_1^{n/2},& \PH[0](\flagstyle{W}_{n+1})&\cong\HP^{n/2} & \PH[1](\flagstyle{W}_{n+1})&\cong\HP^{(n/2)-1} \\
\S{\cellstyle{w}_n}&=\S{2n+2n\signrep}, & \Hcofiber[0]{n}&=\S{2n},& \Hcofiber[1]{n}&=\pt
\end{align*}
Therefore, in \autoref{fig:cofiber-cohomology}, we have the following two diagrams (the top for $r=0$, the bottom for $r=1$):
\[
\begin{tikzcd}[row sep=1.3cm,column sep=1.3cm]
\HHreduced_{C_2}^{2n+2n\signrep}(\S{2n+2n\signrep}) \ar{d}[swap]{(\Htowholecofiber[0]{n})^*} \ar{r}{(\Htocofiber{n})^*} & \HHreduced_{C_2}^{2n+2n\signrep}(\PH(\flagstyle{W}_{n+1})_{+}) \ar{d}{(\Htowhole[n+1]{0})^*}\\
\HHreduced_{C_2}^{2n+2n\signrep}(\S{2n}) \ar{r}[swap]{(\Htocofiber[0]{n})^*} & \HHreduced_{C_2}^{2n+2n\signrep}(\HP^{n/2}_{+})
\end{tikzcd}\]

\[\begin{tikzcd}[row sep=1.3cm,column sep=1.3cm]
\HHreduced_{C_2}^{2n+2n\signrep}(\S{2n+2n\signrep}) \ar{d}[swap]{(\Htowholecofiber[1]{n})^*} \ar{r}{(\Htocofiber{n})^*} & \HHreduced_{C_2}^{2n+2n\signrep}(\PH(\flagstyle{W}_{n+1})_{+}) \ar{d}{(\Htowhole[n+1]{1})^*}\\
\HHreduced_{C_2}^{2n+2n\signrep}(\pt) \ar{r}[swap]{(\Htocofiber[1]{n})^*} & \HHreduced_{C_2}^{2n+2n\signrep}(\HP^{(n/2)-1}_{+})
\end{tikzcd}
\]
For the $r=1$ case, note that since $\HHreduced_{C_2}^{2n+2n\signrep}(\pt)=0$, we have that $(\Htowhole[n+1]{1})^*(\nu)=0$.

For the $r=0$ case, note that in non-equivariant cohomology, the map $(\Htocofiber[0]{n})^*\from\Hhreduced^{2n}(\S{2n})\to\Hhreduced^{2n}(\HP^{n/2}_{+})$ sends the generator $1$ to the generator $(\Htotop[0]{n+1})^*(x_0^{n/2})$, so by \autoref{xr-stuff}, we conclude that the map 
\[(\Htocofiber[0]{n})^*\from\HHreduced_{C_2}^{2n+2n\signrep}(\S{2n})\cong\HH_{C_2}^{2n\signrep}(\S{0})\otimes \Hhreduced^{2n}(\S{2n})\to\HH_{C_2}^{2n\signrep}(\S{0})\otimes\Hhreduced^{2n}(\HP^{n/2}_{+})\subseteq \HHreduced_{C_2}^{2n+2n\signrep}(\HP^{n/2}_{+})\]
sends the generator
$\epsilon^{2n}\in\HHreduced_{C_2}^{2n+2n\signrep}(\S{2n})(\GG)\cong \langle\Z\rangle(\GG)$ to the element \[(\Htotop[0]{n+1})^*(\epsilon^{2n}x_0^{n/2})\in\HHreduced_{C_2}^{2n+2n\signrep}(\PH[0](\flagstyle{W}_{n+1})_{+})(\GG).\]

The map $(\Htowholecofiber[0]{n})^*$ sends $1\in\HHreduced_{C_2}^{2n+2n\signrep}(\S{2n+2n\signrep})$ to $\epsilon^{2n}\in\HHreduced_{C_2}^{2n+2n\signrep}(\S{2n})$, because the composition 
\[\HHreduced_{C_2}^{0}(\S{0})\cong\HHreduced_{C_2}^{2n+2n\signrep}(\S{2n+2n\signrep})\xrightarrow{(\Htowholecofiber[0]{n})^*}\HHreduced_{C_2}^{2n+2n\signrep}(\S{2n})\cong\HHreduced_{C_2}^{2n\signrep}(\S{0})\]
is the same as the map in cohomology induced by $\epsilon^{2n}\from\S{0}\to\S{2n\signrep}$ (see \autoref{def:epsilon-for-2}).

Thus, we have that $(\Htowhole[n+1]{0})^*(\nu)=(\Htotop[0]{n+1})^*(\epsilon^{2n}x_0^{n/2})$.

\subsubsection{Check that $(\Htowhole[n+1]{0})^*(\nu)$ and $(\Htowhole[n+1]{1})^*(\nu)$ equal the desired values for odd $n$}

When $n$ is odd, we have that
\begin{align*}
\flagstyle{W}_{n+1}&=\irredH_0^{(n+1)/2}\oplus\irredH_1^{(n+1)/2},& \PH[0](\flagstyle{W}_{n+1})&\cong\HP^{(n-1)/2} & \PH[1](\flagstyle{W}_{n+1})&\cong\HP^{(n-1)/2} \\
\S{\cellstyle{w}_n}&=\S{(2n-2)+(2n+2)\signrep}, & \Hcofiber[0]{n}&=\pt,& \Hcofiber[1]{n}&=\S{2n-2}
\end{align*}
Therefore, in \autoref{fig:cofiber-cohomology}, we have the following two diagrams (the top for $r=0$, the bottom for $r=1$):
\[
\begin{tikzcd}[row sep=1.3cm,column sep=1.3cm]
\HHreduced_{C_2}^{(2n-2)+(2n+2)\signrep}(\S{(2n-2)+(2n+2)\signrep}) \ar{d}[swap]{(\Htowholecofiber[0]{n})^*} \ar{r}{(\Htocofiber{n})^*} & \HHreduced_{C_2}^{(2n-2)+(2n+2)\signrep}(\PH(\flagstyle{W}_{n+1})_{+}) \ar{d}{(\Htowhole[n+1]{0})^*}\\
\HHreduced_{C_2}^{(2n-2)+(2n+2)\signrep}(\pt) \ar{r}[swap]{(\Htocofiber[0]{n})^*} & \HHreduced_{C_2}^{(2n-2)+(2n+2)\signrep}(\HP^{(n-1)/2}_{+})
\end{tikzcd}\]

\[\begin{tikzcd}[row sep=1.3cm,column sep=1.3cm]
\HHreduced_{C_2}^{(2n-2)+(2n+2)\signrep}(\S{(2n-2)+(2n+2)\signrep}) \ar{d}[swap]{(\Htowholecofiber[1]{n})^*} \ar{r}{(\Htocofiber{n})^*} & \HHreduced_{C_2}^{(2n-2)+(2n+2)\signrep}(\PH(\flagstyle{W}_{n+1})_{+}) \ar{d}{(\Htowhole[n+1]{1})^*}\\
\HHreduced_{C_2}^{(2n-2)+(2n+2)\signrep}(\S{2n-2}) \ar{r}[swap]{(\Htocofiber[1]{n})^*} & \HHreduced_{C_2}^{(2n-2)+(2n+2)\signrep}(\HP^{(n-1)/2}_{+})
\end{tikzcd}
\]
For the $r=0$ case, note that since $\HHreduced_{C_2}^{(2n-2)+(2n+2)\signrep}(\pt)=0$, we have that $(\Htowhole[n+1]{0})^*(\nu)=0$.

For the $r=1$ case, note that in non-equivariant cohomology, the map $(\Htocofiber[1]{n})^*:\Hhreduced^{2n-2}(\S{2n-2})\to\Hhreduced^{2n-2}(\HP^{(n-1)/2}_{+})$ sends the generator $1$ to the generator $(\Htotop[1]{n+1})^*(x_1^{(n-1)/2})$, so by \autoref{xr-stuff}, we conclude that the map
{\small 
\[\mkern-90mu(\Htocofiber[1]{n})^*:\HHreduced_{C_2}^{(2n-2)+(2n+2)\signrep}(\S{2n-2})\cong\HH_{C_2}^{(2n+2)\signrep}(\S{0})\otimes \Hhreduced^{2n-2}(\S{2n-2})\to\]
\[\HH_{C_2}^{(2n+2)\signrep}(\S{0})\otimes\Hhreduced^{2n-2}(\HP^{(n-1)/2}_{+})\subseteq \HHreduced_{C_2}^{(2n-2)+(2n+2)\signrep}(\HP^{(n-1)/2}_{+})\]
}
sends the generator
$\epsilon^{2n+2}\in\HHreduced_{C_2}^{(2n-2)+(2n+2)\signrep}(\S{2n-2})(\GG)\cong \langle\Z\rangle(\GG)$ to the element \[(\Htotop[1]{n+1})^*(\epsilon^{2n+2}x_1^{(n-1)/2})\in\HHreduced_{C_2}^{(2n-2)+(2n+2)\signrep}(\PH[1](\flagstyle{W}_{n+1})_{+})(\GG).\]

The map $(\Htowholecofiber[1]{n})^*$ sends $1\in\HHreduced_{C_2}^{(2n-2)+(2n+2)\signrep}(\S{(2n-2)+(2n+2)\signrep})$ to $\epsilon^{2n+2}\in\HHreduced_{C_2}^{(2n-2)+(2n+2)\signrep}(\S{2n-2})$, because the composition 
\[\HHreduced_{C_2}^{0}(\S{0})\cong\HHreduced_{C_2}^{(2n-2)+(2n+2)\signrep}(\S{(2n-2)+(2n+2)\signrep})\xrightarrow{(\Htowholecofiber[1]{n})^*}\HHreduced_{C_2}^{(2n-2)+(2n+2)\signrep}(\S{2n-2})\cong\HHreduced_{C_2}^{(2n+2)\signrep}(\S{0})\]
is the same as the map in cohomology induced by $\epsilon^{2n+2}\from\S{0}\to\S{(2n+2)\signrep}$ (see \autoref{def:epsilon-for-2}). 

Thus, we have that $(\Htowhole[n+1]{1})^*(\nu)=(\Htotop[1]{n+1})^*(\epsilon^{2n+2}x_1^{(n+1)/2})$.
\end{proof}

Having established that the elements $\gen_n$ and $\Gen_n$ generate the cohomology of the $n$th piece of the filtration, we are now in a position to prove the first part of \autoref{thm:main}.

\begin{theorem}
\label{thm:generate}
The elements $\gen$ and $\Gen$ generate $\HHreduced_{C_2}^{*}(B_{C_2}\SU(2)_{+}\coeffs*{A})$ as an $\HHreduced_{C_2}^{*}(\S{0}\coeffs*{A})$-algebra.
\end{theorem}
\begin{proof}
We have already done the main work in \autoref{generate-finite-levels-2}, where we proved that $(\Htotop{n})^*(\gen)=\gen_n$ and $(\Htotop{n})^*(\Gen)=\Gen_n$ generate $\HHreduced_{C_2}^*(\PH(\flagstyle{W}_n)_{+})$ for every $n\geq 1$. This implies that, if $Y$ is the sub-$\HHreduced_{C_2}^{*}(\S{0})$-algebra of $\HHreduced_{C_2}^{*}(B_{C_2}\SU(2)_{+})$ generated by $\gen$ and $\Gen$, then the composite map 
\[Y\hookrightarrow \HHreduced_{C_2}^{*}(B_{C_2}\SU(2)_{+})\xrightarrow{(\Htotop{n})^*} \HHreduced_{C_2}^{*}(\PH(\flagstyle{W}_n)_{+})\]
is surjective for every $n$. 

By combining \autoref{thm:even-dim-free} and \autoref{thm:PH-even-monotone} (cf. \autoref{thm:additive}), we see that as modules over $\HHreduced_{C_p}^*(\S{0}\coeffs*{A})$, we have
\[\HHreduced_{C_2}^{*}(\PH(\flagstyle{W}_{n+1})_{+})\cong \HHreduced_{C_2}^{*}(\PH(\flagstyle{W}_n)_{+})\oplus\Sigma^{\cellstyle{w}_{n}}\HHreduced_{C_2}^{*}(\S{0}\coeffs*{A}),\]
i.e., there is a single generator added in dimension $\cellstyle{w}_n$. Therefore, the quotient map
\[(\Htonext{n})^*\from \HHreduced_{C_2}^{*}(\PH(\flagstyle{W}_{n+1})_{+})\to \HHreduced_{C_2}^{*}(\PH(\flagstyle{W}_{n})_{+})\] is an isomorphism outside of dimension $\cellstyle{w}_n$, and in particular, it is an isomorphism in dimension $\alpha$ whenever $|\alpha|<|\cellstyle{w}_n|=4n$. Because \autoref{thm:inverse-limit} shows that the natural map \[\HHreduced_{C_2}^{*}(B_{C_2}\SU(2)_{+})\to\varprojlim\HHreduced_{C_2}^{*}(\PH(\flagstyle{W}_n)_{+})\]
is an isomorphism, we conclude that the map $(\Htotop{n})^*\from \HHreduced_{C_2}^{*}(B_{C_2}\SU(2)_{+})\to\HHreduced_{C_2}^{*}(\PH(\flagstyle{W}_n)_{+})$ is an isomorphism in dimension $\alpha$ for all $n\geq |\alpha|/4$.

Because the composite map 
\[Y\hookrightarrow \HHreduced_{C_2}^{*}(B_{C_2}\SU(2)_{+})\xrightarrow{(\Htotop{n})^*} \HHreduced_{C_2}^{*}(\PH(\flagstyle{W}_n)_{+})\]
is surjective for every $n$, we conclude that for any $\alpha\in\RO(C_2)$, there is an $n$ for which this composite map is surjective:
\[Y^\alpha\hookrightarrow \HHreduced_{C_2}^{\alpha}(B_{C_2}\SU(2)_{+})\xrightarrow[\cong]{(\Htotop{n})^*} \HHreduced_{C_2}^{\alpha}(\PH(\flagstyle{W}_n)_{+})\]
Since it is also evidently injective, it must be an isomorphism, so that $Y^\alpha$ is in fact equal to  $\HHreduced_{C_2}^{\alpha}(B_{C_2}\SU(2)_{+})$, and thus $Y=\HHreduced_{C_2}^{*}(B_{C_2}\SU(2)_{+})$.
%
%
%
\end{proof}

\section{\texorpdfstring{Relation between the generators $\gen$ and $\Gen$}{Relation between the generators c and C}}
\label{subsection:cC-relation}

Finally, in this section we establish the remainder of our main result \autoref{thm:main}, namely, the relation between the generators $\gen$ and $\Gen$.

\begin{theorem}
The elements $\gen$ and $\Gen$ satisfy the relation $\gen^2=\epsilon^4\gen+\xi^2\Gen$.
\end{theorem}

\begin{proof}
It suffices to check this under the maps $\HtoGe$, $(\Htowhole{0})^*$, and $(\Htowhole{1})^*$. We have
\begin{align*}
\HtoGe(\gen^2) & = \HtoGe(\gen)^2  & \HtoGe(\epsilon^4\gen+\xi^2\Gen) & = 0+\rho(\xi^2\Gen)\\
&= x^2             &    &=x^2
\end{align*}
and
\begin{align*}
(\Htowhole{0})^*(\gen^2) & = (\Htowhole{0})^*(\gen)^2   &   (\Htowhole{0})^*(\epsilon^4\gen+\xi^2\Gen) & = \epsilon^4\xi^2x_0 + \xi^2x_0(\epsilon^4+\xi^2x_0)\\
&=(\xi^2x_0)^2      &   &=2\epsilon^4\xi^2x_0+\xi^4x_0^2\\
&=\xi^4x_0^2        &   &=\xi^4x_0^2
\end{align*}
and
\begin{align*}
(\Htowhole{1})^*(\gen^2) & = (\Htowhole{1})^*(\gen)^2 & (\Htowhole{1})^*(\epsilon^4\gen+\xi^2\Gen) & = \epsilon^4(\epsilon^4+\xi^2x_1) + \xi^2x_1(\epsilon^4+\xi^2x_1)\\
&=(\epsilon^4+\xi^2x_0)^2 & &=\epsilon^8+2\epsilon^4\xi^2x_1+\xi^4x_1^2\\
&=\epsilon^8+2\epsilon^4\xi^2x_1+\xi^4x_1^2 & &=\epsilon^8+\xi^4x_1^2\\
&=\epsilon^8+\xi^4x_1^2 & &\qedhere
\end{align*}
\end{proof}

\nocite{*}
\printbibliography
\addcontentsline{toc}{chapter}{Bibliography}
\end{document}